\documentclass[10pt,letterpaper,leqno,pdftex]{amsart}
\usepackage[all]{xy}
\usepackage{url}
\usepackage[mathscr]{euscript}
\usepackage{amsmath, amsfonts, amssymb, mathrsfs, stmaryrd}
\usepackage{appendix}
\usepackage{amsthm}
\usepackage{comment}
\usepackage{bm}
\usepackage{lipsum}
\newcommand{\sm}{\smallsetminus}
\newcommand{\mapp}[1]{\xrightarrow{#1}}
\newcommand{\Oo}{\mathcal{O}}

\newcommand{\AAA}{\mathbf{A}}

\newcommand{\Aa}{\mathscr{A}}
\newcommand{\CC}{\mathbb{C}}
\newcommand{\RR}{\mathbb{R}}
\newcommand{\FF}{\mathbb{F}}
\newcommand{\ZZ}{\mathbb{Z}}
\newcommand{\QQ}{\mathbb{Q}}
\newcommand{\MM}{\text{\upshape M}}
\newcommand{\NN}{\text{\upshape N}}

\newcommand{\lqq}{\leqslant}
\newcommand{\gqq}{\geqslant}
\newcommand{\YYY}{\mathscr{Y}}
\newcommand{\ZZZ}{\mathscr{Z}}
\newcommand{\DDD}{\mathscr{D}}

\newcommand{\MMM}{\mathscr{M}}
\newcommand{\TTT}{\mathscr{T}}

\newcommand{\FFF}{\mathscr{F}}
\newcommand{\OOO}{\mathscr{O}}

\newcommand{\Ee}{\mathbf{E}}
\newcommand{\LLL}{\mathscr{L}}
\newcommand{\KKK}{\mathscr{K}}
\newcommand{\XXX}{\mathscr{X}}
\newcommand{\CCC}{\mathscr{C}}
\newcommand{\VVV}{\mathscr{V}}
\newcommand{\WWW}{\mathscr{W}}

\DeclareMathOperator{\Diff}{Diff}
\DeclareMathOperator{\id}{id}

\DeclareMathOperator{\Tr}{Tr}
\DeclareMathOperator{\CM}{CM}
\DeclareMathOperator{\GL}{GL}
\DeclareMathOperator{\Cl}{Cl}

\DeclareMathOperator{\Aut}{Aut}

\DeclareMathOperator{\End}{End}
\DeclareMathOperator{\Gal}{Gal}
\DeclareMathOperator{\Hom}{Hom}
\DeclareMathOperator{\disc}{disc}
\DeclareMathOperator{\Tor}{Tor}

\DeclareMathOperator{\imm}{im}

\DeclareMathOperator{\Br}{Br}
\DeclareMathOperator{\Spec}{Spec}
\DeclareMathOperator{\Nrd}{Nrd}
\DeclareMathOperator{\Trd}{Trd}
\DeclareMathOperator{\charr}{char}

\DeclareMathOperator{\Lie}{Lie}
\DeclareMathOperator{\inv}{inv}
\DeclareMathOperator{\Def}{Def}

\newcommand{\frakL}{\mathfrak L}
\newcommand{\frakD}{\mathfrak D}

\newcommand{\frakb}{\mathfrak b}
\newcommand{\frakm}{\mathfrak m}
\newcommand{\frakM}{\mathfrak M}

\newcommand{\fraka}{\mathfrak a}
\newcommand{\frakg}{\mathfrak g}

\newcommand{\frakQ}{\mathfrak Q}

\newcommand{\frakp}{\mathfrak p}
\newcommand{\frakq}{\mathfrak q}
\newcommand{\frakn}{\mathfrak n}

\newcommand{\frakP}{\mathfrak P}
\newcommand{\frakl}{\mathfrak l}

\newcommand{\ord}{\mathrm{ord}}
\newcommand{\mil}{\varprojlim}
\newcommand{\dlim}{\varinjlim}
\newcommand{\map}[1]{\xrightarrow{#1}}

\theoremstyle{plain}
\newtheorem{theorem}{Theorem}[section]
\newtheorem{corollary}[theorem]{Corollary}
\newtheorem{lemma}[theorem]{Lemma}

\newtheorem{proposition}[theorem]{Proposition}

\theoremstyle{definition}
\newtheorem{definition}[theorem]{Definition}

\newtheorem{remark}[theorem]{Remark}
\numberwithin{equation}{section}

\theoremstyle{plain}
\newtheorem*{theorem*}{Theorem}

\theoremstyle{plain}
\newtheorem*{theorem 1}{Theorem 1}

\theoremstyle{plain}
\newtheorem*{theorem 2}{Theorem 2}

\theoremstyle{plain}
\newtheorem*{theorem 3}{Theorem 3}

\theoremstyle{plain}
\newtheorem*{theorem 4}{Theorem 4}

\begin{document}
\title{The Gross-Zagier formula on singular moduli for Shimura curves}
\author[Andrew Phillips]{Andrew Phillips}

\maketitle

\begin{abstract}
The Gross-Zagier formula on singular moduli can be seen as a calculation of the intersection multiplicity of
two CM divisors on the integral model of a modular curve. We prove a generalization of this result to a Shimura curve.

\vspace{3mm}
\vspace{3mm} 
\end{abstract}

\section{Introduction}

In this paper we study a moduli problem involving QM abelian surfaces with complex multiplication (CM), generalizing
a theorem about the arithmetic degree of a certain moduli stack of CM elliptic curves. 
This moduli problem is the main arithmetic content
of \cite{Howard}. The result of that paper can be seen as a refinement of the well-known formula
of Gross and Zagier on singular moduli in \cite{Gross-Zagier}. We begin by describing how the Gross-Zagier
formula and the result of \cite{Howard} can be interpreted as statements about intersection theory on a modular curve.
Our generalization of \cite{Howard} has a similar interpretation as a result about intersection theory, but now
on a Shimura curve. 

\subsection{Elliptic curves}
Let $K_1$ and $K_2$ be non-isomorphic imaginary quadratic fields and set $K = K_1 \otimes_{\QQ} K_2$.
Let $F$ be the real quadratic subfield of $K$ and let $\frakD \subset \Oo_F$ be the different of $F$. We assume
$K_1$ and $K_2$ have relatively prime discriminants $d_1$ and $d_2$, so $K$ is a field, $K/F$ is unramified at all finite places, and $\Oo_{K_1}
\otimes_{\ZZ} \Oo_{K_2}$ is the maximal order in $K$. 

Let $\MMM$ be the category fibered in groupoids over $\Spec(\Oo_K)$ with $\MMM(S)$ the category of elliptic
curves over the $\Oo_K$-scheme $S$. The category $\MMM$ is an algebraic stack 
(in the sense of \cite{Vistoli}, also known as a Deligne-Mumford stack) which is smooth of relative dimension $1$
over $\Spec(\Oo_K)$ (so it is $2$-dimensional).
For $i \in \{1, 2\}$ let $\YYY_i$ be the algebraic stack over $\Spec(\Oo_K)$ with $\YYY_i(S)$
the category of elliptic curves over the $\Oo_K$-scheme $S$ with complex multiplication by $\Oo_{K_i}$.
When we speak of an elliptic curve $E$ over an $\Oo_K$-scheme $S$ with complex multiplication by $\Oo_{K_i}$, we are assuming that the action $\Oo_{K_i} \to \End_{\OOO_S}(\Lie(E))$ is through the structure map
$\Oo_{K_i} \hookrightarrow \Oo_K \to \OOO_S(S)$. The stack $\YYY_i$ is finite and \'etale over
$\Spec(\Oo_K)$, so in particular it is $1$-dimensional and regular. 
There is a finite morphism $\YYY_i \to \MMM$
defined by forgetting the complex multiplication structure.
 
Even though the morphism $\YYY_i \to \MMM$ is not a closed immersion, we view
$\YYY_i$ as a divisor on $\MMM$ through its image (\cite[Definition 1.7]{Vistoli}).
A natural question to now ask is: what is the intersection multiplicity, defined in the appropriate sense below, 
of the two divisors $\YYY_1$ and $\YYY_2$ on $\MMM$? More generally, if $T_m : \text{Div}(\MMM) \to
\text{Div}(\MMM)$ is the 
$m$-th Hecke correspondence on $\MMM$, what is the intersection multiplicity of 
$T_m\YYY_1$ and $\YYY_2$?

If $\ZZZ$ is an irreducible algebraic stack of dimension $0$ over $\Spec(\Oo_K)$, define its \textit{arithmetic degree} to be
\begin{equation}\label{arithmetic degree}
\deg(\ZZZ) = \sum_{\frakP \subset \Oo_K}\log(|\FF_\frakP|)\sum_{x \in [\ZZZ(\overline{\FF}_\frakP)]}\frac{\text{lg}(\OOO_{\ZZZ, x}^{
\text{sh}})}{|\Aut(x)|},
\end{equation}
where $[\ZZZ(S)]$ is the set of isomorphism classes of objects in the category $\ZZZ(S)$, $\OOO_{\ZZZ, x}^{\text{sh}}$ is the strictly Henselian local ring of $\ZZZ$ at the geometric point $x$ (the local ring for the \'etale topology), and $\text{lg}(\OOO_{\ZZZ, x}^{\text{sh}})$ is the length of the ring.
Also, the outer sum is over all prime ideals $\frakP \subset \Oo_K$, $\FF_\frakP = \Oo_K/\frakP$, and
$\Spec(\overline{\FF}_\frakP)$ is an $\Oo_K$-scheme through the reduction map $\Oo_K \to \FF_\frakP$.
This definition is extended linearly to all $0$-dimensional algebraic stacks over $\Spec(\Oo_K)$.
Returning to the context above, if $\DDD_1$ and $\DDD_2$ are two prime divisors on $\MMM$ intersecting properly, meaning $\DDD_1 \cap
\DDD_2 = \DDD_1 \times_{\MMM} \DDD_2$ is an algebraic stack of dimension $0$, 
define the \textit{intersection multiplicity} of $\DDD_1$ and $\DDD_2$ on $\MMM$ to be the number
$I(\DDD_1, \DDD_2) = \deg(\DDD_1 \cap \DDD_2)$.
The definition of $I(\DDD_1, \DDD_2)$ is extended to all divisors $\DDD_1$
and $\DDD_2$ by bilinearity, assuming $\DDD_1$ and $\DDD_2$ intersect properly.

The intersection multiplicity $I(\YYY_1, \YYY_2)$ relates to the Gross-Zagier formula on singular moduli as follows.
Let $L \supset K$ be a number field and suppose $E_1$ and $E_2$ are elliptic curves over $\Spec(\Oo_L)$.
The $j$-invariant determines an isomorphism of schemes $M_{/\Oo_L} \cong \Spec(\Oo_L[x])$,
where $M \to \Spec(\Oo_K)$ is the coarse moduli scheme associated with $\MMM$,
and the elliptic curves $E_1$ and $E_2$ determine morphisms $\Spec(\Oo_L) \rightrightarrows M_{/\Oo_L}$.
These morphisms correspond to ring homomorphisms $\Oo_L[x] \rightrightarrows \Oo_L$ defined by
$x \mapsto j(E_1)$ and $x \mapsto j(E_2)$. Let $D_1$ and $D_2$ be the divisors on $M_{/\Oo_L}$ defined
by the morphisms $\Spec(\Oo_L) \rightrightarrows M_{/\Oo_L}$. Then 
$$
D_1 \cap D_2 = \Spec(\Oo_L \otimes_{\Oo_L[x]} \Oo_L) \cong \Spec(\Oo_L/(j(E_1) - j(E_2))).
$$
For $\tau$ an imaginary quadratic integer in the complex upper half plane, let $[\tau]$ be its equivalence class under the action
of $\text{PSL}_2(\ZZ)$. As in \cite{Gross-Zagier} define
\begin{equation}\label{GZ}
J(d_1, d_2) = \Bigg(\prod_{\substack{[\tau_1], [\tau_2] \\ \text{disc}(\tau_i) = d_i}}(j(\tau_1) - 
j(\tau_2))\Bigg)^{4/(w_1w_2)},
\end{equation}
where $w_i = |\Oo_{K_i}^\times|$. It follows from the above discussion that the main result of \cite{Gross-Zagier},
which is a formula for the prime factorization of the integer $J(d_1, d_2)^2$, is essentially the same as giving a formula
for $\deg(\YYY_1 \cap \YYY_2) = I(\YYY_1, \YYY_2)$. 

For each positive integer $m$ define $\TTT_m$ to be the algebraic stack over $\Spec(\Oo_K)$ with 
$\TTT_m(S)$ the category of triples $(E_1, E_2, f)$ where $E_i$ is an object of $\YYY_i(S)$ and 
$f \in \Hom_S(E_1, E_2)$ satisfies $\deg(f) = m$ on every connected component of $S$.
In \cite{Howard} it is shown there is a decomposition
$$
\TTT_m = \bigsqcup_{\substack{\alpha \in F^\times \\ \Tr_{F/\QQ}(\alpha) = m}}\XXX_{\alpha}
$$
for some $0$-dimensional stacks $\XXX_\alpha \to \Spec(\Oo_K)$ and then a formula is given for each term in
$$
\deg(\TTT_m) = \sum_{\substack{\alpha \in \frakD^{-1}, \alpha \gg 0 \\ 
\Tr_{F/\QQ}(\alpha) = m}} \deg(\XXX_{\alpha}).
$$
We will prove later (in the appendix) that 
\begin{equation}\label{Hecke 1}
\deg(\TTT_m) = I(T_m\YYY_1, \YYY_2), 
\end{equation}
so the main result
of \cite{Howard} really is a refinement of the Gross-Zagier formula.

Let $\XXX$ be the algebraic stack over $\Spec(\Oo_K)$ with fiber $\XXX(S)$ the category of
pairs $(\Ee_1, \Ee_2)$ where $\Ee_i = (E_i, \kappa_i)$ with $E_i$ an elliptic
curve over the $\Oo_K$-scheme $S$ with complex multiplication $\kappa_i : \Oo_{K_i} \to \End_S(E_i)$.
Let $(\Ee_1, \Ee_2)$ be an object of $\XXX(S)$. The maximal order $\Oo_K = \Oo_{K_1} \otimes_{\ZZ}
\Oo_{K_2}$ acts on the $\ZZ$-module $L(\Ee_1, \Ee_2) = \Hom_S(E_1, E_2)$ by
$$
(t_1 \otimes t_2) \bullet f = \kappa_2(t_2) \circ f \circ \kappa_1(\overline{t}_1),
$$
where $x \mapsto \overline{x}$ is the nontrivial element of $\Gal(K/F)$.
Writing $[\cdot\hspace{.5mm}, \cdot]$ for the bilinear form on $L(\Ee_1, \Ee_2)$ associated with the quadratic
form $\deg$, there is a unique $\Oo_F$-bilinear form
$$
[\cdot\hspace{.5mm}, \cdot]_{\CM} : L(\Ee_1, \Ee_2) \times L(\Ee_1, \Ee_2) \to \frakD^{-1}
$$
satisfying $[f_1, f_2] = \Tr_{F/\QQ}([f_1, f_2]_{\CM})$. Let $\deg_{\CM}$ be the totally positive definite
$F$-quadratic form on $L(\Ee_1, \Ee_2) \otimes_{\ZZ} \QQ$ corresponding to $[\cdot\hspace{.5mm}, \cdot]_{\CM}$, so
$\deg(f) = \Tr_{F/\QQ}(\deg_{\CM}(f))$. 

For any $\alpha \in F^\times$ let $\XXX_\alpha$ be the algebraic stack over $\Spec(\Oo_K)$ with $\XXX_\alpha(S)$
the category of triples $(\Ee_1, \Ee_2, f)$ where $(\Ee_1, \Ee_2)$ is an object of $\XXX(S)$ and $f \in L(\Ee_1, \Ee_2)$
satisfies $\deg_{\CM}(f) = \alpha$ on every connected component of $S$. 
The category $\XXX_\alpha$ is empty unless $\alpha$ is
totally positive and lies in $\frakD^{-1}$. 

Let $\chi$ be the quadratic Hecke character associated with the extension $K/F$ and for $\alpha \in F^\times$
define $\Diff(\alpha)$ to be the set of prime ideals $\frakp \subset \Oo_F$ satisfying $\chi_\frakp(\alpha\frakD) = -1$.
The set $\Diff(\alpha)$ is finite and nonempty. For any fractional $\Oo_F$-ideal $\frakb$ let $\rho(\frakb)$ be the number
of ideals $\mathfrak{B} \subset \Oo_K$ satisfying $\NN_{K/F}(\mathfrak{B}) = \frakb$. For any prime number $\ell$ let $\rho_\ell(\frakb)$ be the number of ideals $\mathfrak{B} \subset 
\Oo_{K, \ell}$ satisfying $\NN_{K_\ell/F_\ell}(\mathfrak{B}) = \frakb\Oo_{F, \ell}$, so there is
a product formula
$$
\rho(\frakb) = \prod_\ell \rho_\ell(\frakb).
$$
The following theorem, which is essentially \cite[Theorem A]{Howard}, is the main result we will generalize.

\begin{theorem 1}[Howard-Yang]\label{theorem 1}
Suppose $\alpha \in F^\times$ is totally positive. If $\alpha \in \frakD^{-1}$ and $\Diff(\alpha) = \{\frakp\}$
then $\XXX_\alpha$ is of dimension zero, is supported in characteristic $p$ {\upshape (}the rational prime below
$\frakp${\upshape )}, and satisfies
$$
\deg(\XXX_\alpha) = \frac{1}{2}\log(p)\cdot\ord_\frakp(\alpha\frakp\frakD)\cdot\rho(\alpha\frakp^{-1}\frakD).
$$
If $\alpha \notin \frakD^{-1}$ or if $\#\Diff(\alpha) > 1$, then $\deg(\XXX_\alpha) = 0$.
\end{theorem 1}

\subsection{QM abelian surfaces}
Our work in generalizing Theorem 1 goes as follows. Let $B$ be an indefinite quaternion algebra over
$\QQ$, let $\Oo_B$ be a maximal order of $B$, and let $d_B$ be the discriminant of $B$. A \textit{QM abelian surface} over a scheme $S$ is a pair $(A, i)$ where $A \to S$ is an abelian scheme of relative dimension $2$ and 
$i : \Oo_B \to \End_S(A)$ is a ring homomorphism. Any QM abelian surface $A$ comes equipped with a principal polarization
$\lambda : A \to A^\vee$ uniquely determined by a condition described below in Proposition \ref{polarization}. If $A_1$ and $A_2$ are QM abelian surfaces over a connected scheme $S$ with corresponding principal polarizations $\lambda_1$ and $\lambda_2$, 
then the map
$$
f \mapsto \lambda_1^{-1} \circ f^\vee \circ \lambda_2 \circ f : \Hom_{\Oo_B}(A_1, A_2) \to \End_{\Oo_B}(A_1)
$$
has image in $\ZZ \subset \End_{\Oo_B}(A_1)$ and defines a positive definite quadratic form, called the \textit{QM degree}
and denoted $\deg^\ast$. 

We retain the same notation of $K_1$, $K_2$, $F$, and $K$ as above. We also assume each prime
dividing $d_B$ is inert in $K_1$ and $K_2$, so in particular, $K_1$ and $K_2$ split $B$. Let $S$ be an $\Oo_K$-scheme.
A \textit{QM abelian surface over $S$ with complex multiplication by $\Oo_{K_j}$}, for $j \in \{1, 2\}$, is a triple
$\AAA = (A, i, \kappa)$ where $(A, i)$ is a QM abelian surface over $S$ and $\kappa : \Oo_{K_j} \to \End_{\Oo_B}(A)$
is an action such that the induced map $\Oo_{K_j} \to \End_{\Oo_B}(\Lie(A))$ is through the
structure map $\Oo_{K_j} \hookrightarrow \Oo_K \to \OOO_S(S)$.
Let $\frakm_B \subset \Oo_B$ be the unique ideal of index $d_B^2$, so $\Oo_B/\frakm_B \cong \prod_{p
\mid d_B}\FF_{p^2}$.

Let $\MMM^B$ be the category fibered in groupoids over $\Spec(\Oo_K)$ with $\MMM^B(S)$ the category whose objects are QM abelian surfaces $(A, i)$ over the $\Oo_K$-scheme $S$ 
satisfying the following condition for any $x \in \Oo_B$: any point of $S$ has an affine open neighborhood $U$ such that $\Lie(A \times_S U)$ is a free $\OOO_U$-module of rank $2$ and there is an equality of polynomials
\begin{equation}\label{Kottwitz}
\text{char}(i(x), \Lie(A \times_S U)) = (T - x)(T - x^\iota)
\end{equation}
in $\OOO_U[T]$, where $x \mapsto x^\iota$ is the main involution on $B$. The category $\MMM^B$ is an algebraic stack which is regular and
flat of relative dimension $1$ over $\Spec(\Oo_K)$, smooth over $\Spec(\Oo_K[d_B^{-1}])$ (if $B$ is a division algebra,
$\MMM^B$ is proper over $\Spec(\Oo_K)$). For $j \in \{1, 2\}$ let $\YYY_j^B$ be the algebraic stack over $\Spec(\Oo_K)$ with $\YYY_j^B(S)$ the category of QM abelian surfaces over the $\Oo_K$-scheme $S$ with complex multiplication by $\Oo_{K_j}$. The stack
$\YYY_j^B$ is finite and \'etale over $\Spec(\Oo_K)$, so in particular it is $1$-dimensional and regular. Any object
of $\YYY_j^B(S)$ automatically satisfies condition (\ref{Kottwitz}) (see Corollary \ref{Kottwitz 2} below). Therefore there is a finite morphism $\YYY_j^B \to \MMM^B$ defined by forgetting the complex multiplication structure.

Our main goal is to calculate the intersection
multiplicity of the two divisors $T_m\YYY_1^B$ and $\YYY_2^B$ on $\MMM^B$, defined just as in 
(\ref{arithmetic degree}), where
$T_m$ is the $m$-th Hecke correspondence on $\MMM^B$. In the course of this calculation we prove the following
result, which should be of independent interest. Let $\bm{k}$ be an imaginary quadratic field and let $\bm{K}$ be any
finite extension of $\bm{k}$. Assume each prime dividing $d_B$ is inert in $\bm{k}$. 
Define $\YYY$ to be the algebraic stack over $\Spec(\Oo_{\bm{K}})$ consisting of all elliptic curves
over $\Oo_{\bm{K}}$-schemes with CM by $\Oo_{\bm{k}}$, and make the analogous definition of $\YYY^B$ for QM abelian surfaces. Then there is a decomposition
$$
\YYY^B = \bigsqcup_{\Oo_{\bm{k}} \to \Oo_B/\frakm_B} \YYY,
$$
where the union is over all ring homomorphisms $\Oo_{\bm{k}} \to \Oo_B/\frakm_B$ (Theorem \ref{Serre type}).

A \textit{CM pair} over an $\Oo_K$-scheme $S$ is a pair $(\AAA_1, \AAA_2)$ where $\AAA_1$ and $\AAA_2$
are QM abelian surfaces over $S$ with complex multiplication by $\Oo_{K_1}$ and $\Oo_{K_2}$,
respectively. For such a pair, set
$
L(\AAA_1, \AAA_2) = \Hom_{\Oo_B}(A_1, A_2).
$
As before, there is a unique $\Oo_F$-quadratic form 
$
\deg_{\CM} : L(\AAA_1, \AAA_2) \to \frakD^{-1}
$
satisfying $\Tr_{F/\QQ}(\deg_{\CM}(f)) = \deg^\ast(f)$. For any QM abelian surface $A$ let $A[\frakm_B]$ be its $\frakm_B$-torsion, defined as a group scheme below. For any
ring homomorphism $\Theta : \Oo_K \to \Oo_B/\frakm_B$ define $\XXX_\Theta^B$ to be the algebraic stack over
$\Spec(\Oo_K)$ where $\XXX_\Theta^B(S)$ is the category of CM pairs $(\AAA_1, \AAA_2)$ over the $\Oo_K$-scheme 
$S$ such that the diagram 
$$
\xymatrix{
\Oo_{K_j} \ar[rr] \ar[dr]_{\Theta|_{\Oo_{K_j}}} && \End_{\Oo_B/\frakm_B}(A_j[\frakm_B]) \\
& \Oo_B/\frakm_B \ar[ur] & }
$$
commutes for $j = 1, 2$, where $\Oo_B/\frakm_B \to \End_{\Oo_B/\frakm_B}(A_j[\frakm_B])$ is the map induced by the action of $\Oo_B$ on $A_j$.
Note that this map makes sense as $\Oo_B/\frakm_B$ is commutative. If $B = \MM_2(\QQ)$ then 
$\frakm_B = \Oo_B$, so any such $\Theta$ is necessarily $0$
and $\XXX_\Theta^B$ is the stack of all CM pairs over $\Oo_K$-schemes.

For any $\alpha \in F^\times$ define $\XXX_{\Theta, \alpha}^B$ to be the algebraic stack over $\Spec(\Oo_K)$
with $\XXX_{\Theta, \alpha}^B(S)$ the category of triples $(\AAA_1, \AAA_2, f)$ where $(\AAA_1, \AAA_2)$
is an object of $\XXX_\Theta^B(S)$ and $f \in L(\AAA_1, \AAA_2)$ satisfies $\deg_{\CM}(f) = \alpha$ on every 
connected component of $S$. Define a nonempty finite set of prime ideals
$$
\Diff_\Theta(\alpha) = \{\frakp \subset \Oo_F : \chi_\frakp(\alpha\fraka_\Theta\frakD) = -1\},
$$
where $\fraka_\Theta = \ker(\Theta) \cap \Oo_F$. Our main result is the following (Proposition \ref{local ring II}
and Theorems \ref{local ring} and \ref{final formula} in the text; see the appendix for the proof of (b)).

\begin{theorem 2}
Let $\alpha \in F^\times$ be totally positive and suppose $\alpha \in \frakD^{-1}$. Let $\Theta : \Oo_K \to \Oo_B/\frakm_B$ 
be a ring homomorphism with $\fraka_\Theta = \ker(\Theta) \cap \Oo_F$, suppose $\Diff_\Theta(\alpha) = \{\frakp\}$, and let $p\ZZ = \frakp \cap \ZZ$. \\
{\upshape (a)} The stack $\XXX_{\Theta, \alpha}^B$ is of dimension zero and is supported in characteristic $p$. \\
{\upshape (b)} There is a decomposition
\begin{equation}\label{Hecke 2}
I(T_m\YYY_1^B, \YYY_2^B) = 
\sum_{\substack{\beta \in \frakD^{-1}, \beta \gg 0 \\ \Tr_{F/\QQ}(\beta) = m}}
\sum_{\Omega : \Oo_K \to \Oo_B/\frakm_B} \deg(\XXX_{\Omega, \beta}^B).
\end{equation}
{\upshape (c)} If $p \nmid d_B$ then 
$$
\deg(\XXX_{\Theta, \alpha}^B) = \frac{1}{2}\log(p)\cdot\ord_\frakp(\alpha\frakp\frakD)\cdot\rho(\alpha\fraka_\Theta^{-1}
\frakp^{-1}\frakD).
$$
{\upshape (d)} Suppose $p \mid d_B$ and let $\frakP \subset \Oo_K$ be the unique prime over $\frakp$. If $\frakP$ 
divides $\ker(\Theta)$  then
$$
\deg(\XXX_{\Theta, \alpha}^B) = \frac{1}{2}\log(p)\cdot \ord_\frakp(\alpha)\cdot \rho(\alpha\fraka_\Theta^{-1}\frakp^{-1}
\frakD).
$$
If $\frakP$ does not divide $\ker(\Theta)$ then
$$
\deg(\XXX_{\Theta, \alpha}^B) = \frac{1}{2}\log(p)\cdot \ord_\frakp(\alpha\frakp)\cdot \rho(\alpha\fraka_\Theta^{-1}\frakp^{-1}\frakD).
$$
If $\alpha \notin \frakD^{-1}$ or if $\#\Diff_\Theta(\alpha) > 1$, then $\deg(\XXX_{\Theta, \alpha}^B) = 0$.
\end{theorem 2}

The proof of this theorem consists of two general parts: counting the number of geometric points of the 
stack $\XXX_{\Theta, \alpha}^B$ (Theorem \ref{orbital} and Proposition \ref{orbital II}) and calculating the
length of the local ring $\OOO^{\text{sh}}_{\XXX_{\Theta, \alpha}^B, x}$ (Theorem \ref{local ring}). As explained above, Theorem 2 can be seen as a generalization of the Gross-Zagier formula in \cite{Gross-Zagier}, but there is another generalization to Shimura curves in \cite{Daas} that is more directly analogous to factoring differences of $j$-values as in (\ref{GZ}). For $N \in \{6, 10, 22\}$, let $B_N$ be the indefinite quaternion division algebra over $\QQ$ with discriminant $N$, and let $R_N \subset B_N$ be a maximal order, which is unique up to conjugation by $B_N^\times$. Let $R_{N, 1}^{\times} \subset R_N^{\times}$ be the group of units with reduced norm $1$, fix an embedding $\iota : B_N \to B_N \otimes_{\QQ}\RR \cong \MM_2(\RR)$, and let $\Gamma_N = \iota(R_{N, 1}^{\times})/\{\pm 1\}$. If $\mathbb{H}$ is the complex upper half plane, then $X_N = \Gamma_N\backslash\mathbb{H}$ is a compact Riemann surface of genus $0$, which has a model over $\QQ$ such that $(X_N)_{/K}$ is a coarse moduli scheme for $\MMM^{B_N}_{/K}$.  Let $j_N : X_N \to \mathbb{P}^1$ be a generator of the function field of $X_N$ over $\QQ$. Then, after a suitable normalization, $j_N(P)$ is an algebraic number for any CM point $P \in X_N$, and \cite[Theorem 1.1]{Daas} gives a formula for the $\QQ$-norm of the cross ratio of $j_N(P_1)$, $j_N(P_2)$, and Atkin-Lehner twists of these two points, for any CM points $P_1, P_2 \in X_N$. The result has a very similar form to \cite[Theorem 1.3]{Gross-Zagier}. There are two proofs given; the first, in \cite[\S 2]{Daas}, uses Theorem 2 above by expressing the $\frakP$-adic valuation, for $\frakP \subset \Oo_K$ prime, of a difference of $j_N$-values at CM points $P_1, P_2 \in X_N$, corresponding to a geometric point $x = (\AAA_1, \AAA_2) \in \XXX_\Theta^B(\overline{\FF}_\frakP)$,  in terms of the length of the local ring $\OOO^{\text{sh}}_{\XXX_{\Theta}^B, x}$.

\subsection{Eisenstein series}
Theorem 1 is really only half of a larger story, one that gives a better explanation of the definition
of the arithmetic degree of $\XXX_\alpha$ and provides a surprising connection between 
arithmetic geometry and analysis. To explain this,
let $K_1$, $K_2$, $F$, and $K$ be as in Section 1.1, let $D = \disc(F)$, and let $\sigma_1$ and $\sigma_2$ be
the two real embeddings of $F$. For $\tau_1, \tau_2$ in the complex upper half plane and $s \in \CC$
define an Eisenstein series
\begin{align*}
E^\ast(\tau_1, \tau_2, s) &= D^{(s+1)/2}\left(\pi^{-(s+2)/2}\Gamma\left(\frac{s+2}{2}\right)\right)^2
\sum_{\fraka \in \Cl(\Oo_F)}\chi(\fraka)\NN(\fraka)^{1 + s} \\
&\times \sum_{(0, 0) \neq (m, n) \in \fraka \times \fraka/\Oo_F^\times}
\frac{(v_1v_2)^{s/2}}{[m, n](\tau_1, \tau_2)|[m, n](\tau_1, \tau_2)|^s},
\end{align*}
where $\Cl(\Oo_F)$ is the ideal class group of $F$, $v_i = \text{Im}(\tau_i)$, and
$$
[m, n](\tau_1, \tau_2) = (\sigma_1(m)\tau_1 + \sigma_1(n))(\sigma_2(m)\tau_2 + \sigma_2(n)).
$$
This series, which is convergent for $\text{Re}(s) \gg 0$, has meromorphic continuation to all $s \in \CC$
and defines a non-holomorphic Hilbert modular form of weight $1$ for $\text{SL}_2(\Oo_F)$ which is
holomorphic in $s$ in a neighborhood of $s = 0$. The derivative of $E^\ast(\tau_1, \tau_2, s)$
at $s = 0$ has a Fourier expansion
$$
(E^\ast)'(\tau_1, \tau_2, 0) = \sum_{\alpha \in \frakD^{-1}}a_\alpha(v_1, v_2)\cdot q^\alpha,
$$
where $e(x) = e^{2\pi ix}$ and $q^\alpha = e(\sigma_1(\alpha)\tau_1 + \sigma_2(\alpha)\tau_2)$.
The connection between this analytic theory and the moduli space $\XXX_\alpha$ lies in the
next theorem (\cite[Theorem B, Theorem C]{Howard}).

\begin{theorem*}[Howard-Yang]
Suppose $\alpha \in F^\times$ is totally positive. If $\alpha \in \frakD^{-1}$ then $a_\alpha = a_\alpha(v_1, v_2)$
is independent of $v_1, v_2$ and $a_\alpha = 4\cdot \deg(\XXX_\alpha)$.
\end{theorem*}

It seems likely that there is a theorem in the spirit of the one above for the moduli space $\XXX_{\Theta, \alpha}^B$, 
but we do not pursue that direction here. A reasonable next question to address is: can Theorem 2 be
extended to the case where $\YYY_j^B$ is defined to be the stack of QM abelian surfaces with CM by a fixed non-maximal
order in $K_j$? A result of this type would seemingly extend the results of Lauter and Viray in \cite{LV} to QM abelian surfaces.

\subsection{Notation}   
When we say ``stack" we mean algebraic stack in the sense
of \cite{Vistoli}, also called a Deligne-Mumford stack. We write $\QQ_{p^2}$
for the unique unramified quadratic extension of $\QQ_p$ and $\ZZ_{p^2} \subset \QQ_{p^2}$ for its ring of integers. If $\CCC$ is a category, we write $C \in \CCC$ to mean $C$ is an object of $\CCC$. We use $\Delta$ to denote the maximal order in the unique quaternion division algebra over $\QQ_p$ and $\overline{\FF}$ for an algebraic closure of a finite field $\FF$. For any number field $L$, we write $\widehat{L} = L \otimes_{\QQ} \widehat{\QQ}$ for the ring of finite adeles over $L$. We say a field $k$ is ``over $\Oo_L$" to mean $k$ is an $\Oo_L$-algebra. If $M$ is a $\ZZ$-module and $V$ a $\QQ$-vector space, let $\widehat{M} = M \otimes_{\ZZ} \widehat{\ZZ}$ and $\widehat{V} = 
V \otimes_{\QQ} \widehat{\QQ}$. If $A \to S$ is an abelian scheme, we write $\End^0_S(A)$ for
$\End_S(A) \otimes_{\ZZ} \QQ$. If $R$ is a commutative ring and $M$ an $R$-module, we write
$\text{lg}_R(M)$ for the length of $M$ and $\text{lg}(R)$ for $\text{lg}_R(R)$. For any ring $R$, let $\MM_n(R)$ be the ring of $n\times n$ matrices over $R$.

\subsection{Acknowledgment} This research forms part of my Boston College Ph.D. thesis. I would
like to thank my advisor Ben Howard for extensive discussions as these results developed. I would also like to thank the anonymous referee for many helpful corrections and suggestions.

\section{QM abelian surfaces}

In this section we give a brief review of the basic theory of QM abelian surfaces. 
For the remainder of this paper fix an indefinite quaternion algebra $B$ over $\QQ$ and a maximal order $\Oo_B$ of $B$. 
We do not exclude the case where $B$ is split, that is, where $B = \MM_2(\QQ)$.
As $B$ is split at $\infty$, all maximal orders of $B$ are conjugate by elements of $B^\times$ and hence isomorphic as rings (combine 17.4.10, 17.4.13, and 17.8.3 of \cite{Voight}). Let $d_B$ be the discriminant of $B$.

\begin{definition} 
Let $S$ be a scheme. A \textit{QM abelian surface} over $S$ is a pair $(A, i)$ where $A \to S$ is an abelian
scheme of relative dimension $2$ and $i : \Oo_B \hookrightarrow \End_S(A)$ is an injective ring homomorphism.
\end{definition}

\begin{definition}
Let $(A_1, i_1)$ and $(A_2, i_2)$ be two QM abelian surfaces over a scheme $S$. A \textit{homomorphism} $f : A_1 \to A_2$
of QM abelian surfaces is a homomorphism of abelian schemes over $S$ satisfying $i_2(x) \circ f = 
f \circ i_1(x)$ for all $x \in \Oo_B$. If in addition $f$ is an isogeny of abelian schemes, then $f$ is called
an \textit{isogeny} of QM abelian surfaces. We write $\Hom_{\Oo_B}(A_1, A_2)$ for the group of all
homomorphisms of QM abelian surfaces $A_1 \to A_2$.
\end{definition}

In fact, any nonzero homomorphism of QM abelian surfaces $A_1 \to A_2$ is necessarily an
isogeny (Corollary \ref{isog}), and any ring homomorphism $\Oo_B \to \End_S(A)$ is automatically injective.
For each place $v$ of $\QQ$ let $\inv_v : \Br_2(\QQ_v) \to \{\pm 1\}$ be the unique isomorphism. 

\begin{definition}
For each prime number $p$, define $B^{(p)}$ to be the quaternion division algebra over $\QQ$ determined
by 
$$
\inv_v(B^{(p)}) = \left\{\begin{array}{ll}
\inv_v(B) & \text{if $v \notin \{p, \infty\}$} \\
-\inv_v(B) & \text{if $v \in \{p, \infty\}$}.
\end{array} \right.
$$
\end{definition}

\begin{proposition}\label{char p}
Suppose $A$ is a QM abelian surface over a field $k$. \\
{\upshape(a)} The abelian variety $A$ is either simple or isogenous to $E^2$ for some elliptic curve $E$ over $k$. \\
{\upshape(b)} If $k = \overline{\FF}_p$ then 
$\End^0_{\Oo_B}(A) = \End_{\Oo_B}(A) \otimes_{\ZZ} \QQ$ is either
\begin{enumerate}
\item[{\upshape (1)}] an imaginary quadratic field which splits $B$, or
\item[{\upshape (2)}] the definite quaternion algebra $B^{(p)}$.
\end{enumerate}
Furthermore, $A$ is isogenous to $E^2$ with
$E$ ordinary in case {\upshape (1)} and supersingular in case {\upshape (2)}. \\
{\upshape(c)} If $k = \CC$ then $\End^0_{\Oo_B}(A)$ is either $\QQ$ or an imaginary quadratic field
which splits $B$.
\end{proposition}

\begin{proof}
(a) Suppose $A \sim E_1 \times E_2$ with $E_1 \not\sim E_2$. Since $B$ is a simple $\QQ$-algebra, there is an injective ring homomorphism
$$
B \to \End^0_k(A) \cong \End_k^0(E_1) \times \End^0_k(E_2) \to \End^0_k(E_1),
$$
so $B \cong \End^0_k(E_1)$ by counting $\QQ$-dimensions. This forces $\charr(k) = p$ and $\End^0_k(E_1)$ to be the quaternion division algebra over $\QQ$ ramified at $p$ and $\infty$ (\cite[Theorem 42.1.9]{Voight}), contradicting $B$ being split at $\infty$.

For (b) see \cite[Proposition 5.2]{Milne} and for (c) see \cite[Proposition 52]{Clark}.
\end{proof}

\begin{proposition}\label{lefschetz}
Suppose $A$ is a QM abelian surface over a field $k$ and $k' \subset k$ is a subfield.
Then $\End_k(A)$ embeds into $\End_{\widetilde{k}}(\widetilde{A})$ for some QM abelian surface $\widetilde{A}$ defined over a finite extension $\widetilde{k}/k'$. 
\end{proposition}

\begin{proof}
Write $k = \dlim k_j$ with $\{k_j\}$\vspace{.5mm} the directed system of all subfields of $k$ that are finitely generated over $k'$. Then $\End_k(A) = \dlim \End_{k_j}(A \otimes_{k'} k_j)$ by \cite[Theorem 10.63]{GW} and since $\End_k(A)$ is a
finitely generated $\ZZ$-module, there is an index $j_0$ such that $\End_k(A) = \End_{k_j}(A \otimes_{k'} k_j)$ for all $j \gqq j_0$, so we may assume the extension $k/k'$ is finitely generated.

Now we will use induction on the transcendence degree of $k$ over $k'$.
The result is trivial if $k$ has transcendence degree $0$ over $k'$, so assume the result
holds for any QM abelian surface defined over a field $k_0$ with a fixed transcendence degree over $k'$,
and suppose $A$ is a QM abelian surface defined over a field $k$ with transcendence degree $1$ over $k_0$.
Then $k$ is a finite extension of $k_0(x)$ for some $x \in k$ transcendental over $k_0$. Let $\Oo_k$ be the integral closure of $k_0[x]$ in $k$, which is a Dedekind domain, and
fix a prime $\frakp \subset \Oo_k$ of good reduction for $A$. This means that there is an abelian scheme
$\Aa$ over $\Spec(\Oo_{k, \frakp})$ whose generic fiber is $A$. Since $\Aa$ is an abelian scheme, it is the N\'eron model of its generic fiber $A$, so $\End_k(A) \cong \End_{\Oo_{k, \frakp}}(\Aa)$
by the universal property of the N\'eron model. This gives $\Aa$ the structure of a QM abelian surface.

Let $\widetilde{A} = \Aa \otimes_{\Oo_{k, \frakp}} \widetilde{k}$
be the reduction of $\Aa$ modulo $\frakp$, where $\widetilde{k} = \Oo_k/\frakp$, which is a QM abelian surface over $\widetilde{k}$.
By \cite[Theorem 2.1(2)]{Conrad} the natural map $\End_{\Oo_{k, \frakp}}(\Aa) \to \End_{\widetilde{k}}(\widetilde{A})$
is injective. Since $\widetilde{k}$ is a finite extension of $k_0$ and there is
an inclusion $\End_k(A) \hookrightarrow \End_{\widetilde{k}}(\widetilde{A})$, we are done by induction.
\end{proof}

\begin{corollary}\label{endo ring}
Let $A$ be a QM abelian surface over an algebraically closed field $k$. \\
{\upshape(a)} If $\charr(k) = 0$ then $\End^0_{\Oo_B}(A)$ is isomorphic to $\QQ$ or an imaginary quadratic field which splits $B$. \\
{\upshape(b)} If $\charr(k) = p$ then $\End^0_{\Oo_B}(A)$ is isomorphic to one of{\upshape \hspace{.5mm}:} $\QQ$, an imaginary quadratic field, or $B^{(p)}$.
\end{corollary}

\begin{proof}
All cases are an immediate consequence of Propositions \ref{char p} and \ref{lefschetz}, except when
$\charr(k) = p$ and $L = \End^0_{\Oo_B}(A)$ embeds as a quadratic subfield of $B^{(p)}$. In this case, let $D = \End^0_k(A)$, so $\dim_{\QQ} D = (\dim_{\QQ} B)(\dim_{\QQ} L) = 8$ by \cite[Proposition 7.7.8]{Voight}. By the classification of the possibilities for $D$ (\cite[Theorem 1.1]{Yu}), we must have $D \cong B \otimes_{\QQ} L'$
for some imaginary quadratic field $L'$. Therefore $L = C_D(B) = L'$ is imaginary quadratic. 
\end{proof}

Let $x \mapsto x^{\iota}$ be the main
involution of $B$ and fix $a \in \Oo_B$ satisfying $a^2 = -d_B$. 
Define another involution on $B$ by $x \mapsto x^{\ast} = a^{-1}x^{\iota}a$. The order $\Oo_B$ is 
stable under $x \mapsto x^\ast$. 
If $(A, i)$ is a QM abelian surface over $S$, then so is the dual abelian scheme $A^\vee$, with corresponding homomorphism
$i^\vee : \Oo_B \hookrightarrow \End_S(A^\vee)$ defined by $i^\vee(x) = i(x^\ast)^\vee$.
If $f : A_1 \to A_2$ is a homomorphism of QM abelian surfaces, then so is $f^\vee : A_2^\vee \to 
A_1^\vee$.

\begin{proposition} \label{polarization}
Let $(A, i)$ be a QM abelian surface over a scheme $S$. There is a 
unique principal polarization $\lambda : A \to A^{\vee}$ such that
the corresponding Rosati involution $\varphi \mapsto \varphi^{\dagger} = 
\lambda^{-1} \circ \varphi^\vee \circ \lambda$ on $\End^0_S(A)$ induces 
$x \mapsto x^\ast$ on $\Oo_B$, that is, $\lambda^{-1} \circ i(x)^\vee \circ \lambda = i(x^\ast)$ for all $x \in \Oo_B$.
\end{proposition}

\begin{proof}
See \cite[Proposition III.1.8]{BC} and \cite[Proposition III.3.5]{BC} for the cases where $S = \Spec(k)$ with $k$ an algebraically closed field of characteristic $0$ and $p$, respectively.
The general case is reduced to these by \cite[Proposition in \S 11]{B}.
\end{proof}

\begin{lemma} \label{isogenies 2}
Let $(A, i)$ be a QM abelian surface over a scheme $S$ and assume $B$ is a division algebra. 
If $x \in \Oo_B$ is nonzero then $i(x) \in \End_S(A)$
is an isogeny of degree $\Nrd(x)^2$, where $\Nrd : B^\times \to \QQ^\times$ is the reduced norm.
\end{lemma}

\begin{proof}
Any nonzero $x \in B$ is invertible, so $i(x)$ is an isogeny. To compute its degree we may assume $S = \Spec(k)$ for
$k$ an algebraically closed field. Applying the Skolem-Noether theorem to the
two maps $B \to \End^0_k(A)$ given by $b \mapsto i(b)$ and $b \mapsto i(b^{\iota})^\dagger$, where $b \mapsto b^{\iota}$
is the main involution on $B$ and $\varphi \mapsto \varphi^\dagger$ is the Rosati involution on $\End^0_k(A)$ as in Proposition \ref{polarization}, we find that there is a $u \in \End^0_k(A)^\times$ such that
$i(b) = u \circ i(b^{\iota})^\dagger \circ u^{-1}$ for all $b \in B$. Hence $\deg(i(x)) = \deg(i(x^{\iota})^\dagger) = \deg(i(x^{\iota}))$ and 
$$
\deg(i(x))^2 = \deg(i(xx^{\iota})) = \deg([\Nrd(x)]) = \Nrd(x)^4.
$$
Since $\deg(i(x))$ is a positive integer, $\deg(i(x)) = \Nrd(x)^2$.
\end{proof}

Let $A_1$ and $A_2$ be QM abelian surfaces over $S$ with corresponding principal polarizations
$\lambda_1 : A_1 \to A_1^\vee$ and $\lambda_2 : A_2 \to A_2^\vee$. Suppose $f : A_1 \to A_2$ is a nonzero homomorphism
of QM abelian surfaces. We obtain a map $f^t : A_2 \to A_1$ defined
as the composition 
$$
f^t = \lambda_1^{-1} \circ  f^\vee \circ \lambda_2 : A_2 \to A_1.
$$
This is a homomorphism of QM abelian surfaces (it is $\Oo_B$-linear by the $\ast$-property in Proposition \ref{polarization}), called the \textit{dual homomorphism} of $f$.
If $f, g \in \Hom_{\Oo_B}(A_1, A_2)$ then $(f^t)^t = f$ and $(f + g)^t = f^t + g^t$. These follow immediately from the corresponding properties of $f \mapsto f^\vee$. Similarly, if $f \in \Hom_{\Oo_B}(A_1, A_2)$ and $g \in \Hom_{\Oo_B}(A_2, A_3)$, then $(g \circ f)^t = f^t \circ g^t$.

\begin{proposition}
Let $f : A_1 \to A_2$ be a homomorphism of QM abelian surfaces over a scheme $S$. The homomorphism $f^t \circ f : A_1 \to A_1$
is locally on $S$ multiplication by an integer.
\end{proposition}

\begin{proof}
We may assume $S = \Spec(k)$ with $k$ an algebraically closed field. Viewing $f^t \circ f \in \End^0_{\Oo_B}(A_1)$, a calculation shows $f^t \circ f$ is fixed by the Rosati involution $\varphi \mapsto \varphi^\dagger = \lambda_1^{-1} \circ \varphi^\vee \circ \lambda_1$. By Corollary \ref{endo ring}, the $\QQ$-algebra
$E = \End_{\Oo_B}^0(A_1)$ is isomorphic to one of: (i) $\QQ$, (ii) an imaginary quadratic field, or (iii) the definite quaternion algebra $B^{(p)}$ for some prime $p$. In case (i), $E = \QQ$, so $f^t \circ f \in \QQ$. Using the classification of finite dimensional division algebras over $\QQ$ with a positive involution (\cite[\S 21]{Mumford}), there are the following possibilities for the involution $\dagger$ in the remaining cases. In case (ii), $E$ is an imaginary quadratic field and $\dagger$ is complex conjugation, so the set of fixed points is $\QQ$. In case (iii), $E$ is a definite quaternion algebra over $\QQ$ and $\dagger$ is the standard involution, so again the set of fixed points is $\QQ$. This shows
$f^t \circ f \in \QQ \cap \End_{\Oo_B}(A_1) = \ZZ$.
\end{proof}

\begin{definition}
If the integer in the previous proposition is constant on $S$, we call it the \textit{QM degree} of $f$, and it is denoted $\deg^\ast(f)$.
\end{definition}

\begin{corollary}
Let $A_1$ and $A_2$ be QM abelian surfaces over a connected scheme $S$ and suppose $f \in \Hom_{\Oo_B}(A_1, A_2)$
is an isogney. Then $\deg^\ast(f^t) = \deg^\ast(f)$ and $\deg(f) = \deg^\ast(f)^2$.
\end{corollary}

\begin{proof}
This can be checked on geometric fibers, so we may assume $S = \Spec(k)$ for $k$ an algebraically closed field.
Let $d = \deg^\ast(f)$. Note that $f^t \circ f = [d]_{A_1}$ implies $(f \circ f^t) \circ f = [d]_{A_2} \circ f$ and
hence $f \circ f^t = [d]_{A_2}$ by \cite[Proposition 27.178]{GW2}.
The first claim now follows from $(f^t)^t = f$.
For the second claim, since $f^t \circ f = [d]_{A_1}$, we have $\deg(f^t)\deg(f) = d^4$.
However, $\deg(f^t) = \deg(f^\vee) = \deg(f)$, so $\deg(f) = d^2$.
\end{proof}

\begin{proposition}\label{QM degree qf}
Let $A_1$ and $A_2$ be QM abelian surfaces over a connected scheme $S$. The map $\deg^\ast : \Hom_{\Oo_B}(A_1, A_2) \to \ZZ$ is a positive definite quadratic form. 
\end{proposition}

\begin{proof}
The only nontrivial part is showing $\deg^\ast(f) > 0$ if $f \in \Hom_{\Oo_B}(A_1, A_2)$ is nonzero. For this we may assume $S = \Spec(k)$ with $k$ an algebraically closed field. 
Define an isogeny of abelian varieties 
$$
\Phi : A_1 \times A_2 \to A_1 \times A_2
$$
by $\Phi(x, y) = (f^t(y), f(x))$ on points in $k$-schemes. Then $\Phi^\vee$ is given by $\Phi^\vee(u, v) = (f^\vee(v), (f^t)^\vee(u))$. If $\lambda_j : A_j \to A_j^\vee$, $j =1, 2$, are the usual principal polarizations, then we get a principal polarization
$$
\lambda_1 \times \lambda_2 : A_1 \times A_2 \to A_1^\vee \times A_2^\vee.
$$
A calculation shows that the corresponding Rosati involution on $\End^0_k(A_1 \times A_2)$ satisfies $\Phi^\dagger = \Phi$, so $\Phi\circ \Phi^\dagger = [\deg^\ast(f)]$. Since the Rosati involution is positive, $\deg^\ast(f) > 0$.
\end{proof}

\begin{corollary}\label{isog}
Let $A_1$ and $A_2$ be QM abelian surfaces over a scheme $S$. Any nonzero element of $\Hom_{\Oo_B}(A_1, A_2)$
is an isogeny.
\end{corollary}

\begin{proof}
Assume $f \in \Hom_{\Oo_B}(A_1, A_2)$ is nonzero. To show $f$ is an isogeny it suffices to 
check that the map on fibers $f_s$ is an isogeny for all $s \in S$, so we may assume $S = \Spec(k)$ with $k$ a field. Since $\deg^\ast(f) \neq 0$ by Proposition \ref{QM degree qf}, the homomorphism
$f^t \circ f = [\deg^\ast(f)]$ is an isogeny, so $f^t$ is surjective and thus is an isogeny since
$A_1$ and $A_2$ have the same dimension (\cite[Corollary 27.177]{GW2}). Therefore
$f = (f^t)^t = \lambda_2^{-1} \circ (f^t)^{\vee}\circ \lambda_1$ is an isogeny.
\end{proof}

\section{QM abelian surfaces with CM}

For this section let $\bm{k}$ be an imaginary quadratic field and let $\bm{K}$ be a finite extension of $\bm{k}$.
Assume any prime dividing $d_B$ is inert in $\bm{k}$. This implies $\bm{k}$ splits $B$ or, equivalently, $\bm{k}$ embeds in $B$ (\cite[Proposition 14.6.7]{Voight}).

\subsection{Definitions}

\begin{definition}
Let $S$ be an $\Oo_{\bm{K}}$-scheme. A \textit{QM abelian surface over $S$
with complex multiplication by} $\Oo_{\bm{k}}$, which we will abbreviate as a \textit{CMQM abelian surface}, is a triple $\textbf{A} = (A, i, \kappa)$, where $(A, i)$ is a QM abelian surface over $S$ and $\kappa : \Oo_{\bm{k}} \to \End_{\Oo_B}(A)$ is a ring homomorphism such that the diagram
$$
\xymatrix{
\Oo_{\bm{k}} \ar[rr]^<<<<<<<<<<<<<<<<{\kappa^{\Lie}} \ar[dr] & & \End_{\Oo_B}(\Lie(A)) \\
& \OOO_S(S) \ar[ur] }
$$
commutes, where $\Oo_{\bm{k}} \hookrightarrow \Oo_{\bm{K}} \to \OOO_S(S)$ is the structure map.
We call the commutativity of this diagram the \textit{CM normalization condition}.
\end{definition}

When we speak of a CMQM abelian surface over $\overline{\FF}_\frakP$ for some prime ideal $\frakP \subset \Oo_{\bm{K}}$, 
where $\FF_\frakP = \Oo_{\bm{K}}/\frakP$, it is understood that $\Spec(\overline{\FF}_\frakP)$ is an $\Oo_{\bm{K}}$-scheme through the reduction map $\Oo_{\bm{K}} \to \FF_\frakP \hookrightarrow \overline{\FF}_\frakP$.
Less precisely, when we speak of a CMQM abelian surface $A$ over $\overline{\FF}_p$ for some
prime number $p$, we really mean $A$ is a CMQM abelian surface over $\overline{\FF}_\frakP$ for some prime ideal $\frakP \subset \Oo_{\bm{K}}$ lying over $p$. 

\begin{remark}\label{CM remark}
Suppose $(A, i, \kappa)$ is a QM abelian surface over a field $k$ with CM by $\Oo_{\bm{k}}$. We claim $A$ is a CM abelian variety in the traditional sense, that is, $\End_k^0(A)$ contains an \'etale $\QQ$-subalgebra with dimension $4$ over $\QQ$. To see this, let $\bm{k}'$ be an imaginary quadratic field
with discriminant prime to $\disc(\bm{k})$ and such that each prime dividing $d_B$ is inert in $\bm{k}'$. Then $\bm{L} = \bm{k} \otimes_{\QQ} \bm{k}'$ is a quartic CM field with ring of integers $\Oo_{\bm{L}} = 
\Oo_{\bm{k}} \otimes_{\ZZ} \Oo_{\bm{k}'}$ and there is an embedding $\bm{k}' \hookrightarrow B$.
It follows that there is an embedding $\Oo_{\bm{k}'} \hookrightarrow \Oo_B$ since $\Oo_{\bm{k}'}$ is contained in some maximal order and all maximal orders are isomorphic. Therefore there is an embedding $\Oo_{\bm{L}} \hookrightarrow \End_k(A)$ given by 
$$
a_1 \otimes a_2 \mapsto \kappa(a_1) \circ i(a_2) = i(a_2) \circ \kappa(a_1),
$$
viewing $a_2 \in \Oo_{\bm{k}'} \subset \Oo_B$, so in particular, $\bm{L} \subset \End^0_k(A)$.
\end{remark}

\begin{lemma} \label{transpose}
Suppose $A$ is a QM abelian surface with complex
multiplication by $\Oo_{\bm{k}}$ via the homomorphism $\kappa : \Oo_{\bm{k}} \to \End_{\Oo_B}(A)$. 
If $a \in \Oo_{\bm{k}}$ is nonzero then $\kappa(a)$ is an isogeny and
$\kappa(a)^t = \kappa(\overline{a})$, where $a \mapsto \overline{a}$ is complex
conjugation on $\bm{k}$. In particular, $\deg^\ast(\kappa(a)) = \NN_{\bm{k}/\QQ}(a)$ for all $a \in \bm{k}$.
\end{lemma}

\begin{proof}
The homomorphism $\kappa(a)$ is an isogeny for the same reason that $i(b)$ is an isogeny for any nonzero $b \in 
\Oo_B$ (Lemma \ref{isogenies 2}). For the rest, it suffices to assume $A$ is defined over an algebraically closed field.
From the embedding $\kappa : \bm{k} \hookrightarrow \End^0_{\Oo_B}(A)$ and Corollary \ref{endo ring}, either $\End^0_{\Oo_B}(A)$ is $\bm{k}$ or a definite quaternion algebra over $\QQ$. It follows that
there is an embedding $\kappa' : \bm{k} \hookrightarrow \End^0_{\Oo_B}(A)$ such that the Rosati involution
on $\End^0_{\Oo_B}(A)$ corresponding to the principal polarization $\lambda : A \to A^\vee$ 
restricts to complex conjugation on $\kappa'(\bm{k})$. 
Then using the Skolem-Noether theorem, there is a $u \in \End^0_{\Oo_B}(A)^{\times}$
such that $\kappa(a) = u \circ \kappa'(a) \circ u^{-1}$ for all $a \in \bm{k}$, and hence $\deg^\ast(\kappa(a)) = 
\deg^\ast(\kappa'(a))$. 

The Rosati involution on $\End^0_{\Oo_B}(A)$ is given by $\varphi \mapsto \lambda^{-1}\circ \varphi^\vee
\circ \lambda = \varphi^t$, so by construction, $\kappa'(a)^t = \kappa'(\overline{a})$ for any $a \in \bm{k}$. Hence
$$
[\deg^\ast(\kappa'(a))] = \kappa'(a)^t \circ \kappa'(a) = \kappa'(\overline{a}) \circ \kappa'(a) = \kappa'(\overline{a}a)
= [\NN_{\bm{k}/\QQ}(a)],
$$
which means $\deg^\ast(\kappa(a)) = \NN_{\bm{k}/\QQ}(x)$ for all $a \in \bm{k}$. Then, for any $a \in \Oo_{\bm{k}}$,
$$
\kappa(\overline{a}) \circ \kappa(a) = [\NN_{\bm{k}/\QQ}(a)] = [\deg^\ast(\kappa(a))] = \kappa(a)^t\circ \kappa(a)
$$
and composing both sides on the right with $\kappa(a)^t$ gives
$$
[\deg^\ast(\kappa(a))] \circ \kappa(\overline{a}) = [\deg^\ast(\kappa(a))] \circ \kappa(a)^t,
$$
so $\kappa(a)^t = \kappa(\overline{a})$ (\cite[Proposition 27.178]{GW2}).
\end{proof}

\begin{corollary}\label{polarization2}
If $A$ is a QM abelian surface with complex multiplication $\kappa : \Oo_{\bm{k}} \to \End_{\Oo_B}(A)$, then the principal
polarization $\lambda : A \to A^\vee$ of Proposition {\upshape \ref{polarization}} is $\Oo_{\bm{k}}$-linear{\upshape \hspace{.5mm}:}
$\lambda \circ \kappa(a) = \kappa(\overline{a})^\vee \circ \lambda$ for all $a \in \Oo_{\bm{k}}$.
\end{corollary}

\begin{proof}
By definition,
$\kappa(a)^t = \lambda^{-1} \circ \kappa(a)^\vee \circ \lambda$ for any $a \in \Oo_{\bm{k}}$. But from Lemma \ref{transpose},
$\kappa(a)^t = \kappa(\overline{a})$, and hence $\lambda \circ \kappa(a) = \kappa(\overline{a})^\vee \circ \lambda$. 
\end{proof}

\begin{definition}\label{CM FELC}
Define $\YYY^B$ to be the category whose objects are triples $(A, i, \kappa)$, 
where $(A, i)$ is a QM abelian surface
over some $\Oo_{\bm{K}}$-scheme with complex multiplication $\kappa : \Oo_{\bm{k}} \to \End_{\Oo_B}(A)$. A morphism $(A', i', \kappa') \to (A, i, \kappa)$ between 
two such triples defined
over $\Oo_{\bm{K}}$-schemes $T$ and $S$, respectively, is a morphism of $\Oo_{\bm{K}}$-schemes $T \to S$ together with 
an $\Oo_{\bm{k}}$-linear isomorphism $A' \to A \times_S T$ of QM abelian surfaces.
\end{definition}

The category $\YYY^B$ is a stack of finite type over $\Spec(\Oo_{\bm{K}})$. In fact,  $\YYY^B \to \Spec(\Oo_{\bm{K}})$ is 
\'etale by Proposition \ref{reduction} below, proper by a proof similar to that of \cite[Proposition 3.3.5]{Howard3},
and quasi-finite by Corollary \ref{st2} below, so the morphism is finite \'etale.
Let $[\YYY^B(S)]$ denote the set of isomorphism classes of objects in $\YYY^B(S)$.

For each prime $p$ dividing $d_B$ there is a unique maximal ideal $\frakm_p \subset \Oo_B$
such that $\frakm_p \cap \ZZ = p\ZZ$ (\cite[Theorem 18.3.6]{Voight}), and $\Oo_B/\frakm_p$ is a finite field with $p^2$ elements (\cite[Theorem 13.3.11]{Voight}). Define 
$\frakm_B = \bigcap_{p \mid d_B}\frakm_p$. We have $\frakm_B = \prod_{p \mid d_B}\frakm_p$
because for any two primes $p$ and $q$ dividing $d_B$, $\frakm_p\frakm_q = \frakm_q\frakm_p$ (\cite[Theorem 18.3.4]{Voight}) and $\frakm_p + \frakm_q = \Oo_B$. Note that 
$$
\Oo_B/\frakm_B \cong \prod_{p \mid d_B} \FF_{p^2}
$$
as rings.
Let $(A, i)$ be a QM abelian surface over a scheme $S$ and suppose $B$ is a division algebra. The $d_B$-torsion $A[d_B]$ is a finite locally free commutative $S$-group scheme with a natural action of $\frakm_B/d_B\Oo_B$. 
Let $x_B$ be any element of $\frakm_B$ whose image 
generates the principal ideal $\frakm_B/d_B\Oo_B \subset \Oo_B/d_B\Oo_B$. Define the $\frakm_B$-torsion of $A$
as
$$
A[\frakm_B] = \ker(i(x_B) : A[d_B] \to A[d_B]),
$$
which is a finite locally free commutative $S$-group scheme of order $\deg(i(x_B)) = \Nrd(x_B)^2 = d_B^2$ by Lemma \ref{isogenies 2}. This definition does not depend on the choice of $x_B$. The group scheme $A[\frakm_B]$ has an action of 
$\Oo_B/\frakm_B$ given on points by $\widetilde{x}\cdot a = i(x)(a)$ for $\widetilde{x} \in 
\Oo_B/\frakm_B$ and $a \in A[\frakm_B](T)$ for any $S$-scheme $T$. Equivalently, we could define 
$$
A[\frakm_B] = \bigcap_{x \in \frakm_B}\ker(i(x)),
$$
the scheme-theoretic intersection. If $B$ is split, set $\frakm_B = \Oo_B$ and define $A[\frakm_B]$ to
be the trivial $S$-group scheme.

\begin{definition}\label{theta stack}
Let $\theta : \Oo_{\bm{k}} \to \Oo_B/\frakm_B$ be a ring homomorphism. Define
$\YYY^B(\theta)$ to be the category whose objects are objects $(A, i, \kappa)$ of $\YYY^B$ such that
the diagram
\begin{equation}\tag{D$(A, \theta)$}
\xymatrix{
\Oo_{\bm{k}} \ar[rr]^>>>>>>>>>>>>>>>>{\kappa^{\frakm_B}} \ar[dr]_{\theta} && \End_{\Oo_B/\frakm_B}(A[\frakm_B]) \\
& \Oo_B/\frakm_B \ar[ur]_{i^{\frakm_B}} & }
\end{equation}
commutes, where $\kappa^{\frakm_B}$ and $i^{\frakm_B}$ are the maps on $\frakm_B$-torsion induced by $\kappa$ and $i$.
Morphisms are defined in the same way as in the category $\YYY^B$.
\end{definition}

Note that $\YYY^B(\theta) = \YYY^B$ if $B$ is split.
Recall from the introduction that $\YYY$ is the stack over $\Spec(\Oo_{\bm{K}})$ with $\YYY(S)$ the category of
elliptic curves over the $\Oo_{\bm{K}}$-scheme $S$ with CM by $\Oo_{\bm{k}}$.
We will prove below there is an isomorphism of stacks over $\Spec(\Oo_{\bm{K}})$
\begin{equation}\label{decomposition}
\bigsqcup_{\theta : \Oo_{\bm{k}} \to \Oo_B/\frakm_B} \YYY \to \YYY^B
\end{equation}
inducing an isomorphism $\YYY \to \YYY^B(\theta)$ for any $\theta$ (Theorem
\ref{Serre type}). It follows that $\YYY^B(\theta)$ has the structure of a stack, finite \'etale over $\Spec(\Oo_{\bm{K}})$,
and $\YYY \cong \YYY^B$ in the case of $B$ split.

\subsection{Group actions}
Suppose $(A, i, \kappa)$ is a QM abelian surface over an $\Oo_{\bm{K}}$-scheme $S$ with complex multiplication by $\Oo_{\bm{k}}$, and let $\fraka$ be a fractional ideal of $\Oo_{\bm{k}}$. Since there is a ring homomorphism
$\kappa : \Oo_{\bm{k}} \to \End_S(A)$, we may view $A$ as an $\Oo_{\bm{k}}$-module scheme over $S$, so from $\fraka$ being a finitely generated projective $\Oo_{\bm{k}}$-module, locally free of rank $1$, there is an 
abelian scheme $\fraka \otimes_{\Oo_{\bm{k}}} A \to S$ of relative dimension $2$ satisfying 
$(\fraka \otimes_{\Oo_{\bm{k}}} A)(T) = \fraka \otimes_{\Oo_{\bm{k}}} A(T)$ for any $S$-scheme $T$ 
(see \cite[\S 7]{Conrad}, \cite[\S 1]{Zavosh}). 
There are commuting actions
$$
i_{\fraka} : \Oo_B \to \End_S(\fraka \otimes_{\Oo_{\bm{k}}} A), \quad \kappa_{\fraka} : \Oo_{\bm{k}} \to \End_S(\fraka \otimes_{\Oo_{\bm{k}}} A)
$$
defined in the obvious way.
Using the isomorphism $\Lie(\fraka \otimes_{\Oo_{\bm{k}}} A) \cong \fraka \otimes_{\Oo_{\bm{k}}} \Lie(A)$ of $\OOO_S$-modules, it follows that $\kappa_\fraka^{\Lie}$  inherits the CM normalization condition from $\kappa^{\Lie}$.
This shows $\fraka \otimes_{\Oo_{\bm{k}}} A$ is a QM abelian surface over $S$ with complex multiplication by $\Oo_{\bm{k}}$.
Therefore the ideal class group $\Cl(\Oo_{\bm{k}})$ acts on the set $[\YYY^B(S)]$. This action restricts to an action on the subset $[\YYY^B(\theta)(S)]$ since there is an isomorphism 
$$
(\fraka \otimes_{\Oo_{\bm{k}}} A)[\frakm_B] \cong \fraka \otimes_{\Oo_{\bm{k}}} A[\frakm_B]
$$
of $\Oo_{\bm{k}}$-module schemes over $S$.

\begin{lemma}\label{free action}
Let $S$ be a connected $\Oo_{\bm{K}}$-scheme. The action of $\Cl(\Oo_{\bm{k}})$ on $[\YYY^B(S)]$ is free.
\end{lemma}

\begin{proof}
Suppose there exists $\fraka \in \Cl(\Oo_{\bm{k}})$ and $A \in [\YYY^B(S)]$ satisfying
$A \cong \fraka \otimes_{\Oo_{\bm{k}}} A$ as CMQM abelian surfaces. Setting $\Oo = \Oo_B \otimes_{\ZZ} \Oo_{\bm{k}}$, there is an isomorphism
of $\Oo_{\bm{k}}$-modules 
$$
\End_{\Oo}(A) \cong \Hom_{\Oo}(A, \fraka \otimes_{\Oo_{\bm{k}}} A).
$$
By \cite[Proposition 2(b)]{Zavosh} there is an isomorphism of $\Oo_{\bm{k}}$-modules
$$
\Hom_{\Oo_{\bm{k}}}(A, \fraka \otimes_{\Oo_{\bm{k}}} A) \cong \fraka \otimes_{\Oo_{\bm{k}}}\End_{\Oo_{\bm{k}}}(A),
$$
so $\End_{\Oo}(A) \cong \fraka \otimes_{\Oo_{\bm{k}}}\End_{\Oo}(A)$. We claim that $\End_{\Oo}(A) \cong
\Oo_{\bm{k}}$ as an $\Oo_{\bm{k}}$-algebra. By definition, $\End_{\Oo}(A)$ is the centralizer of
$\Oo_{\bm{k}}$ in $\End_{\Oo_B}(A)$. Picking any geometric point $\overline{s}$ of $S$, there are inclusions 
$$
\Oo_{\bm{k}} \hookrightarrow \End_{\Oo_B}(A) \hookrightarrow \End_{\Oo_B}(A_{\overline{s}}),
$$
the second coming from the rigidity lemma for abelian schemes (\cite[Proposition 27.105]{GW2}). By Corollary \ref{endo ring}, either
$\End_{\Oo_B}(A_{\overline{s}}) \cong \Oo_{\bm{k}}$ or $\End_{\Oo_B}(A_{\overline{s}})$ is an order in a quaternion algebra,
so the same is true of $\End_{\Oo_B}(A)$.
The centralizer of $\Oo_{\bm{k}}$ in either such ring is $\Oo_{\bm{k}}$.
Hence $\fraka \cong \Oo_{\bm{k}}$ as an $\Oo_{\bm{k}}$-module, which means $\fraka$ is principal.
\end{proof}

The other important group action on $[\YYY^B(S)]$ comes from the Atkin-Lehner group $W_B$ of $\Oo_B$.
By definition, $W_B = \NN_{B^\times}(\Oo_B)/\QQ^\times\Oo_B^\times = \langle w_p : p \mid d_B\rangle$,
where $w_p \in \Oo_B$ has reduced norm $p$. As an abstract group, $W_B \cong \prod_{p \mid d_B}
\ZZ/2\ZZ$. The group $W_B$ acts on the set $[\YYY^B(S)]$ for any 
$\Oo_{\bm{K}}$-scheme $S$ as follows: for $w \in W_B$ and $\mathbf{A} = (A, i, \kappa) \in \YYY^B(S)$, define
$w\cdot \mathbf{A} = (A, i_w, \kappa)$, where $i_w : \Oo_B \to \End_S(A)$ is given by $i_w(x) = i(w^{-1}xw)$. The actions of
$W_B$ and $\Cl(\Oo_{\bm{k}})$ commute, so there is an induced action of $W_B \times \Cl(\Oo_{\bm{k}})$ on
$[\YYY^B(S)]$. 

\begin{proposition}\label{st1}
The group $W_B \times \Cl(\Oo_{\bm{k}})$ acts simply transitively on $[\YYY^B(\CC)]$.
\end{proposition}

\begin{proof}
It is shown in \cite{Jordan} that $W_B' \times \Cl(\Oo_{\bm{k}})$ acts simply transitively on $[\YYY^B(\CC)]$, where $W_B' \subset W_B$ is the subgroup generated by $\{w_p : \text{$p \mid d_B$, $p$ inert in $\bm{k}$}\}$. However, we are assuming 
each prime $p \mid d_B$ is inert in $\bm{k}$.
\end{proof}

\subsection{Structure of CMQM abelian surfaces}

The main result of this section states that any CMQM abelian surface arises from a CM
elliptic curve through the Serre tensor construction described in Section 3.2. We will use this in the next section to give
a description, in terms of certain coordinates, of the ring $\End_{\Oo_B}(A) \otimes_{\ZZ} \ZZ_p$ for $A$ a CMQM abelian surface over $\overline{\FF}_p$ for $p \mid d_B$. 
Fix a prime ideal $\frakP \subset \Oo_{\bm{K}}$ of residue characteristic $p$. Let $\WWW_{\bm{K}_\frakP}$ be the ring of integers in the completion
of the maximal unramified extension of $\bm{K}_{\frakP}$, so in particular $\WWW_{\bm{K}_\frakP}$ is an $\Oo_{\bm{K}}$-algebra. Let $\mathbf{CLN}_{\bm{K}_\frakP}$ be the category whose objects are complete local Noetherian $\WWW_{\bm{K}_\frakP}$-algebras with residue field $\overline{\FF}_\frakP$,
where $\FF_\frakP = \Oo_{\bm{K}}/\frakP$,
and morphisms $R \to R'$ are local $\WWW_{\bm{K}_\frakP}$-algebra homomorphisms inducing the identity $\overline{\FF}_\frakP \to \overline{\FF}_\frakP$ on residue fields.

\begin{definition}
Suppose $\widetilde{R} \to R$ is a surjection of $\Oo_{\bm{K}}$-algebras and $\mathbf{A} = (A, i, \kappa) \in \YYY^B(R)$. 
A \textit{deformation of} $\mathbf{A}$ (or just a \textit{deformation of} $A$) \textit{to} $\widetilde{R}$ is an object $(\widetilde{A}, \widetilde{i}, \widetilde{\kappa})
\in \YYY^B(\widetilde{R})$ together with an $\Oo_{\bm{k}}$-linear isomorphism $\widetilde{A} \otimes_{\widetilde{R}} R \to A$
of QM abelian surfaces.
\end{definition}

If $\widetilde{R} \to R$ is a surjection of $\Oo_{\bm{K}}$-algebras, $(A, i, \kappa) \in \YYY^B(R)$, and 
$(\widetilde{A}, \widetilde{i}, \widetilde{\kappa}) \in \YYY^B(\widetilde{R})$ is a deformation of $(A, i, \kappa)$, then
it is easy to check that the principal polarizations $\widetilde{\lambda} : \widetilde{A} \to (\widetilde{A})^\vee$
and $\lambda : A \to A^\vee$ defined in Proposition \ref{polarization} are compatible in the sense that $\lambda$ is
the reduction of $\widetilde{\lambda}$.
Let $\mathbf{A} = (A, i, \kappa) \in \YYY^B(\overline{\FF}_\frakP)$ and define a functor $\Def_{\Oo_B}(A, \Oo_{\bm{k}}) : \mathbf{CLN}_{\bm{K}_\frakP} \to \mathbf{Sets}$ that assigns to each object $R$ of $\mathbf{CLN}_{\bm{K}_\frakP}$ the set of isomorphism classes of deformations of $\mathbf{A}$
to $R$. 

\begin{proposition} \label{reduction}
The functor $\Def_{\Oo_B}(A, \Oo_{\bm{k}})$ is represented by $\WWW_{\bm{K}_\frakP}$, so there is a bijection
$$
\Def_{\Oo_B}(A, \Oo_{\bm{k}})(R) \cong \Hom_{\mathbf{CLN}_{\bm{K}_\frakP}}(\WWW_{\bm{K}_\frakP}, R),
$$
which is a one point set, for any object $R$ of $\mathbf{CLN}_{\bm{K}_\frakP}$. In particular, the reduction map
$[\YYY^B(R)] \to [\YYY^B(\overline{\FF}_\frakP)]$ is a bijection for any $R \in \mathbf{CLN}_{\bm{K}_\frakP}$.
\end{proposition}

\begin{proof}
Let $R$ be an Artinian object of $\mathbf{CLN}_{\bm{K}_\frakP}$, so the reduction map $R \to \overline{\FF}_\frakP$ is surjective with nilpotent kernel. Using Corollary \ref{polarization2}, by \cite[Proposition 2.1.2]{Howard2}, $A$ has a unique deformation $\widetilde{A}$, as an abelian scheme with
an action of $\Oo_{\bm{k}}$, to $R$. Also, the reduction map $\End_{\Oo_{\bm{k}}}(\widetilde{A}) \to \End_{\Oo_{\bm{k}}}(A)$ is an isomorphism by \cite[Theorem 2.2.1(2)]{Howard4}. Therefore we can lift the $\Oo_{\bm{k}}$-linear
action of $\Oo_B$ on $A$ to a unique such action on $\widetilde{A}$. This shows that each object of
$\YYY^B(\overline{\FF}_\frakP)$ has a unique deformation to an object of $\YYY^B(R)$ for any Artinian $R$
in $\mathbf{CLN}_{\bm{K}_\frakP}$. Now let $R$ be an arbitrary object of $\mathbf{CLN}_{\bm{K}_\frakP}$, so $R = \mil R/\frakm^n$, where $\frakm \subset R$
is the maximal ideal. The result now follows
from the Artinian case, the bijection 
$$
\Hom_{\mathbf{CLN}_{\bm{K}_\frakP}}(\WWW_{\bm{K}_\frakP}, R) \cong \mil \Hom_{\mathbf{CLN}_{\bm{K}_\frakP}}(\WWW_{\bm{K}_\frakP}, R/\frakm^n),
$$
and the fact that the natural map 
$$
\Def_{\Oo_B}(A, \Oo_{\bm{k}})(R) \to \mil \Def_{\Oo_B}(A, \Oo_{\bm{k}})(R/\frakm^n)
$$
is a bijection by Grothendieck's existence theorem (\cite[Theorem 3.4]{Conrad}).
\end{proof}

\begin{proposition} \label{st}
The group $W_B \times \Cl(\Oo_{\bm{k}})$ acts simply transitively on $[\YYY^B(\overline{\FF}_\frakP)]$.
\end{proposition}

\begin{proof}
Let $\CC_p$ be the field of complex $p$-adic numbers and fix a ring embedding $\WWW_{\bm{K}_\frakP} \to \CC_p$.
There is a $W_B\times\Cl(\Oo_{\bm{k}})$-equivariant bijection $[\YYY^B(\CC_p)] \to [\YYY^B(\overline{\FF}_{\frakP})]$ defined as follows. Let $A \in \YYY^B(\CC_p)$.
Since $A$ is a CM abelian variety over $\CC_p$ (Remark \ref{CM remark}), it descends to a number field, which means there is an abelian surface $A_0$
over $L$ with an action of $\Oo_{\bm{k}}$, for some number field $L$, and an isomorphism $A_0 \otimes_L \CC_p \cong A$
compatible with the actions of $\Oo_{\bm{k}}$ on each. By passing to a finite extension of $L$ if necessary,
we may assume $\End_{\overline{L}}(A_0) = \End_L(A_0)$, where $\overline{L}$ is an algebraic closure of $L$.
Now fix a prime $\frakp \subset \Oo_L$ lying over $p$.
Passing to a further finite extension of $L$, we may assume $A_0$ has good reduction
at $\frakp$, since $A_0$ is a CM abelian variety.
From \cite[Theorem 2.1(1)]{Conrad}, the natural map $\End_{\overline{L}}(A_0) \to
\End_{\CC_p}(A)$ is an isomorphism, so there is an isomorphism $\End_{\Oo_{\bm{k}}}(A_0) \to \End_{\Oo_{\bm{k}}}(A)$. Therefore the
$\Oo_{\bm{k}}$-linear $\Oo_B$-action on $A$ induces such an action on $A_0$, which means $A_0$ is a CMQM abelian surface over $L$.
 
Let $\Aa_0$ be the N\'eron model of $A_0$ at $\frakp$, 
so $\Aa_0$ is an abelian scheme over $\Oo_{L, \frakp}$ satisfying $\Aa_0 \otimes_{\Oo_{L, \frakp}}
\text{Frac}(\Oo_{L, \frakp}) \cong A_0$. Since $\End_L(A_0) \cong \End_{\Oo_{L, \frakp}}(\Aa_0)$, 
there are induced commuting actions of $\Oo_B$ and
$\Oo_{\bm{k}}$ on $\Aa_0$, making it into a CMQM abelian surface over $\Oo_{L, \frakp}$. Finally, let $\widetilde{A}_0 = \Aa_0 \otimes_{\Oo_{L, \frakp}} \Oo_L/\frakp$, so $\widetilde{A}_0$ is a CMQM abelian surface over $\Oo_L/\frakp$.
Define $[\YYY^B(\CC_p)] \to [\YYY^B(\overline{\FF}_\frakP)]$ by $A \mapsto \widetilde{A}_0 \otimes_{\Oo_L/\frakp} \overline{\FF}_\frakP$.
This map does not depend on the abelian surface $A_0$ or the number field $L$ such that $A_0 \otimes_L \CC_p \cong A$
since we are base extending to $\overline{\FF}_\frakP$ in the end. 

Next define a map $[\YYY^B(\overline{\FF}_\frakP)] \to [\YYY^B(\CC_p)]$ as the composition
$$
[\YYY^B(\overline{\FF}_\frakP)] \to [\YYY^B(\WWW_{\bm{K}_\frakP})] \to [\YYY^B(\CC_p)],
$$
where the first map is the inverse of the reduction map in Proposition \ref{reduction} and the second map is base extension to $\CC_p$.
This is the inverse to the map $[\YYY^B(\CC_p)] \to [\YYY^B(\overline{\FF}_\frakP)]$ defined above. The result now follows from
Proposition \ref{st1}.
\end{proof}

\begin{corollary}\label{st2}
The group $W_B \times \Cl(\Oo_{\bm{k}})$ acts simply transitively on $[\YYY^B(k)]$ for any algebraically closed field $k$ over $\Oo_{\bm{K}}$.
\end{corollary}

\begin{proof}
Let $A \in [\YYY^B(k)]$ and first suppose $k$ is of characteristic $0$, so there is an embedding $\overline{\QQ} \hookrightarrow k$. Then there is an abelian variety $A_0$ over $\overline{\QQ}$ such 
that $A_0 \otimes_{\overline{\QQ}}k \cong A$ and $\End_{\overline{\QQ}}(A_0) \cong \End_k(A)$ (\cite[Theorem 1.7.2.1]{CCO}). Hence $A_0$ has commuting actions of $\Oo_B$ and $\Oo_{\bm{k}}$, making it into a CMQM abelian surface over $\overline{\QQ}$. This shows that base extension
$[\YYY^B(\overline{\QQ})] \to [\YYY^B(k)]$ is surjective, and clearly it is $W_B\times\Cl(\Oo_{\bm{k}})$-equivariant. 

For any $A_1, A_2 \in [\YYY^B(\overline{\QQ})]$, the natural map 
$$
\Hom_{\overline{\QQ}}(A_1, A_2) \to \Hom_k(A_1 \otimes_{\overline{\QQ}} k, A_2 \otimes_{\overline{\QQ}} k)
$$
is an isomorphism by \cite[Theorem 2.1(2)]{Conrad}. Therefore, if $f : A_1 \otimes_{\overline{\QQ}} k \to A_2 \otimes_{\overline{\QQ}} k$ is an isomorphism in $\YYY^B(k)$, then $f = f_0 \otimes \id_k$ for some $f_0 \in \Hom_{\overline{\QQ}}(A_1, A_2)$, which is necessarily an isomorphism of abelian varieties by faithfully flat descent. Also, $f$ being $\Oo_B \otimes_{\ZZ} \Oo_{\bm{k}}$-linear implies $f_0$ is as well, which shows $f_0 : A_1 \to A_2$ is an isomorphism in $\YYY^B(\overline{\QQ})$ and hence base extension 
$[\YYY^B(\overline{\QQ})] \to [\YYY^B(k)]$ is a bijection. Now fix an embedding $\overline{\QQ} \hookrightarrow \CC$. It follows from \cite[Corollary 7.10]{Milne1} that base extension $[\YYY^B(\overline{\QQ})] \to [\YYY^B(\CC)]$ is a bijection, so the result in this case is a consequence of
Proposition \ref{st1}.

Next suppose $k$ is of characteristic $p$, so there is an embedding $\overline{\FF}_\frakP \hookrightarrow k$. Since there is a quartic CM field $L$ such that $\Oo_L \subset \End_k(A)$ (Remark \ref{CM remark}), by
\cite[Theorem 1.7.5.1]{CCO}, there is an abelian variety $A_0$ over $\overline{\FF}_\frakP$ such that
$A_0 \otimes_{\overline{\FF}_\frakP} k \cong A$. Since the natural map
$$
\Hom_{\overline{\FF}_\frakP}(A_1, A_2) \to \Hom_k(A_1 \otimes_{\overline{\FF}_\frakP} k, A_2 \otimes_{\overline{\FF}_\frakP} k)
$$
is an isomorphism for any $A_1, A_2 \in [\YYY^B(\overline{\FF}_\frakP)]$ by \cite[Theorem 2.1(2)]{Conrad}, the same argument as above shows that base extension $[\YYY^B(\overline{\FF}_\frakP)] \to [\YYY^B(k)]$ is a bijection, so the result in this case is a consequence of Proposition \ref{st}.
\end{proof}

\begin{remark}\label{st remark}
Similar proofs as in Proposition \ref{st} and Corollary \ref{st2} show that for any algebraically closed field $k$ over $\Oo_{\bm{K}}$, the group $\Cl(\Oo_{\bm{k}})$ acts simply transitively on the set $[\YYY(k)]$, using the special case of $k = \CC$ given in \cite[Proposition 1.2]{Silverman}.
\end{remark}

Our next goal is to prove there is an isomorphism as in (\ref{decomposition}). It will be a consequence of this isomorphism
that any $A \in \YYY^B(S)$ is of the form $M \otimes_{\Oo_{\bm{k}}} E$ for some $E \in \YYY(S)$ and
some $\Oo_B \otimes_{\ZZ}\Oo_{\bm{k}}$-module $M$, free of rank $4$ over $\ZZ$.
To prove this result, we will describe a bijection between the set of isomorphism classes of such modules $M$ and the set $[\YYY^B(\CC)]$.

For the remainder of this section
set $\Oo = \Oo_B \otimes_{\ZZ} \Oo_{\bm{k}}$ and $\Oo_p = \Oo \otimes_{\ZZ} \ZZ_p$ for any prime $p$. Define $\LLL$ to be the set of isomorphism
classes of $\Oo$-modules that are free of rank $4$ over $\ZZ$ and $\KKK$ to be the set of $\Oo_B^\times$-conjugacy
classes of ring embeddings $\Oo_{\bm{k}} \hookrightarrow \Oo_B$. Note that since $\bm{k}$ embeds in 
$B$ and all maximal orders of $B$ are isomorphic, the set $\KKK$ is nonempty. We begin by examining the local structure of modules
in $\LLL$.

\begin{lemma}\label{Eichler}
Fix a prime $p$, let $\Delta$ be the maximal order in the unique quaternion division algebra $\Delta_{\QQ}$
over $\QQ_p$, and fix an embedding $\ZZ_{p^2} \hookrightarrow \Delta$. There is an isomorphism of rings $\ZZ_{p^2} \otimes_{\ZZ_p} \Delta \cong R_1$, where
$$
R_1 = \begin{bmatrix} \ZZ_{p^2} & \ZZ_{p^2} \\ p\ZZ_{p^2} & \ZZ_{p^2} \end{bmatrix}
$$
is the standard Eichler order of level $1$ in $\MM_2(\QQ_{p^2})$.
\end{lemma}

\begin{proof}
There is a ring homomorphism $f : \ZZ_{p^2} \to \End_{\ZZ_{p^2}}(\Delta)$ given by $f(a)(\delta) = a\delta$, and
there is a ring homomorphism $g : \Delta \to \End_{\ZZ_{p^2}}(\Delta)$ given by $g(x)(\delta) = \delta x^{\iota}$,
where $x \mapsto x^{\iota}$ is the main involution on $\Delta_{\QQ}$. As $f$ and $g$ have commuting images, there is an induced ring homomorphism
$$
\Phi : \ZZ_{p^2} \otimes_{\ZZ_p} \Delta \to \End_{\ZZ_{p^2}}(\Delta) \cong \MM_2(\ZZ_{p^2})
$$
given by $\Phi(a \otimes x)(\delta) = a\delta x^{\iota}$. Tensoring this map with $\QQ_{p^2}$ induces the natural isomorphism 
$\QQ_{p^2} \otimes_{\QQ_p} \Delta_{\QQ} \cong \MM_2(\QQ_{p^2})$ (the maximal subfield $\QQ_{p^2} \subset 
\Delta_{\QQ}$ containing $\QQ_p$ splits $\Delta_{\QQ}$), so $\ker \Phi$ is a torsion $\ZZ_p$-module. However,
$\ZZ_{p^2} \otimes_{\ZZ_p} \Delta$ is a torsion-free $\ZZ_p$-module, which means $\Phi$ is injective.

Let $\frakm_{\Delta} \subset \Delta$ be the unique maximal ideal. Then $\End_{\ZZ_{p^2}}(\frakm_{\Delta})$ and
$\End_{\ZZ_{p^2}}(\Delta)$ are distinct maximal orders in $\End_{\QQ_{p^2}}(\Delta \otimes_{\ZZ_{p^2}} \QQ_{p^2})
\cong \MM_2(\QQ_{p^2})$ and 
$$
\imm\Phi \subset R' = \End_{\ZZ_{p^2}}(\Delta) \cap \End_{\ZZ_{p^2}}(\frakm_{\Delta}).
$$
Since $R'$ is an Eichler order in $\MM_2(\QQ_{p^2})$, it is conjugate to 
$$
R_n = \begin{bmatrix} \ZZ_{p^2} & \ZZ_{p^2} \\ p^n\ZZ_{p^2} & \ZZ_{p^2} \end{bmatrix}
$$
for some $n \gqq 1$ (\cite[Lemma A.9(2)]{Conrad}). To show $\imm\Phi = R_1$, we will consider the
discriminants of the orders $\imm\Phi \cong \ZZ_{p^2} \otimes_{\ZZ_p} \Delta$ and $R'$ in $\MM_2(\QQ_{p^2})$.
By \cite[Example A.13]{Conrad}, $\disc(\Delta) = p^2\ZZ_p$, and thus $\disc(\imm\Phi) = 
\disc(\Delta)\ZZ_{p^2} = p^2\ZZ_{p^2}$. By \cite[Example A.12]{Conrad}, $\disc(R') = \disc(R_n) = p^{2n}\ZZ_{p^2}$.
As $\imm\Phi \subset R'$, $p^{2n}\ZZ_{p^2} \mid p^2\ZZ_{p^2}$, so
we must have $n = 1$ and $\imm\Phi = R' \cong R_1$.
\end{proof}

\begin{lemma}\label{types}
Fix a prime $p$ and an embedding $\ZZ_{p^2} \hookrightarrow \Delta$ so that
there is a decomposition $\Delta = \ZZ_{p^2} \oplus \ZZ_{p^2}\Pi$, where $\Pi$ is a uniformizer satisfying
$\Pi^2 = p$ and $\Pi a = \overline{a}\Pi$ for all $a \in \ZZ_{p^2}$. 
Then any ring homomorphism $f : \Delta \to \MM_2(\ZZ_{p^2})$ is $\GL_2(\ZZ_{p^2})$-conjugate to exactly one of the following two maps{\upshape \hspace{.5mm}:}
$$
f_1 : a + b\Pi \mapsto \begin{bmatrix} a & b \\ p\overline{b} & \overline{a}\end{bmatrix}, \quad 
f_2 : a + b\Pi \mapsto \begin{bmatrix} a & pb \\ \overline{b} & \overline{a}\end{bmatrix}.
$$
\end{lemma}

The proof uses the general ideas of the proof of \cite[Theorem 1.4]{Ribet}.

\begin{proof}
Let $M = \ZZ_{p^2} \oplus \ZZ_{p^2}$. Then $M$ is a left $\ZZ_{p^2}$-module via componentwise multiplication, and
a right $\Delta$-module via matrix multiplication $\begin{bmatrix} a & b\end{bmatrix}f(x)$, viewing elements of $M$ as
row vectors. These actions commute, so $M$ is a $\Delta \otimes_{\ZZ_p} \ZZ_{p^2}$-module. By Lemma \ref{Eichler},
$\Delta \otimes_{\ZZ_p} \ZZ_{p^2} \cong R_1$ is the standard Eichler order of level $1$ in $\MM_2(\QQ_{p^2})$ and thus a 
hereditary order. Any $R_1$-module which is free of finite rank over $\ZZ_p$ is a direct sum of copies of $\Delta$ and $\frakm_{\Delta}$, where $\frakm_{\Delta} \subset \Delta$ is the unique maximal ideal (\cite[Chapter 9]{orders}).
By comparing $\ZZ_p$-ranks, we see that there is an isomorphism of $\Delta \otimes_{\ZZ_p} \ZZ_{p^2}$-modules
$\varphi : M \to \Delta$ or $\varphi : M \to \frakm_{\Delta}$.

First suppose $\varphi : M \to \Delta$ is an isomorphism of $\Delta \otimes_{\ZZ_p} \ZZ_{p^2}$-modules, where $\Delta$ is
a right $\Delta$-module through multiplication on the right, and a left $\ZZ_{p^2}$-module through multiplication on the left
via the inclusion $\ZZ_{p^2} \hookrightarrow \Delta$. Let $M'$ be the group $M$ with the same left $\ZZ_{p^2}$-action, but
now a right $\Delta$-action given by 
$$
(x, y) \cdot (a + b\Pi) = \begin{bmatrix} x & y \end{bmatrix}\begin{bmatrix} a & b \\ p\overline{b} & \overline{a}\end{bmatrix}.
$$
Then there is an isomorphism $\psi : \Delta \to M'$ of $\Delta \otimes_{\ZZ_p} \ZZ_{p^2}$-modules given by
$\psi(a + b\Pi) = \begin{bmatrix} a & b \end{bmatrix}$, and thus $\gamma = \psi \circ \varphi : M \to M'$ is a $\Delta \otimes_{\ZZ_p}
\ZZ_{p^2}$-linear isomorphism. Hence $\gamma \in \GL_2(\ZZ_{p^2})$ and since it is $\Delta$-linear, $\gamma(m\cdot x) = \gamma(m)\cdot x$ for all $x \in \Delta$ and $m \in M$. Therefore $f = \gamma\circ f_1 \circ \gamma^{-1}$.

Now suppose $\varphi : M \to \frakm_{\Delta}$ is an isomorphism of $\Delta \otimes_{\ZZ_p} \ZZ_{p^2}$-modules, where $\frakm_{\Delta}$ is a right
$\Delta$-module through multiplication on the right, and a left $\ZZ_{p^2}$-module through $\ZZ_{p^2} \hookrightarrow \Delta$.
Let $M'$ be the group $M$ with the same left $\ZZ_{p^2}$-action, but now a right $\Delta$-action given by
$$
(x, y) \cdot (a + b\Pi) = \begin{bmatrix} x & y \end{bmatrix}\begin{bmatrix} a & pb \\ \overline{b} & \overline{a} \end{bmatrix}.
$$
Writing $\frakm_{\Delta} = p\ZZ_{p^2} \oplus \ZZ_{p^2}\Pi$, there is an isomorphism $\psi : \frakm_{\Delta} \to M'$ of $\Delta \otimes_{\ZZ_p}\ZZ_{p^2}$-modules given by $\psi(pa + b\Pi) = \begin{bmatrix} a & b \end{bmatrix}$. Similarly to the first case, it follows that
$f = \gamma \circ f_2 \circ \gamma^{-1}$, where $\gamma = \psi \circ \varphi \in \GL_2(\ZZ_{p^2})$.

To show $f_1$ and $f_2$ are not $\GL_2(\ZZ_{p^2})$-conjugate, first note that $f_1 = Tf_2T^{-1}$, where
$$
T = \begin{bmatrix} 1 & 0 \\ 0 & p \end{bmatrix}.
$$
Suppose $f_1$ and $f_2$ are conjugate, so $f_1 = Xf_2X^{-1}$ for some $X \in \GL_2(\ZZ_{p^2})$. Then
conjugation by $T$ on $f_2(\Delta) \subset \MM_2(\ZZ_{p^2})$ is equal to conjugation by $X$, which means
$X = UT$ for some $U$ in the center of $f_2(\Delta)$. In particular, $U \in \MM_2(\ZZ_{p^2})$. We then have $0 = \ord_p(\det(X)) = \ord_p(\det(U)) + 1 \gqq 1$,
a contradiction.
\end{proof}

\begin{lemma}\label{local modules}
Let $p$ be a prime number. For $p \nmid d_B$ there is a unique isomorphism class of $\Oo_p$-modules free
of rank $4$ over $\ZZ_p$ and for $p \mid d_B$ there are two isomorphism classes.
\end{lemma}

\begin{proof}
First suppose $p \nmid d_B$. In this case,
$$
\Oo_p \cong \Oo_{B, p} \otimes_{\ZZ_p} \Oo_{\bm{k},p} \cong \MM_2(\Oo_{\bm{k}, p}),
$$
and any $\Oo_p$-module that is free of rank $4$ over $\ZZ_p$ is isomorphic to $\Oo_{\bm{k}, p} \oplus
\Oo_{\bm{k}, p}$, with the natural left action of $\MM_2(\Oo_{\bm{k}, p})$. Now suppose $p \mid d_B$, so $\Oo_p \cong 
\Delta \otimes_{\ZZ_p} \ZZ_{p^2}$. By the proof of Lemma \ref{types}
there are two isomorphism classes of modules over this ring that are free of rank $4$ over $\ZZ_p$.
\end{proof}

Now we will show that the three sets $\KKK$, $\LLL$, and $[\YYY^B(\CC)]$ are all in bijection. Define an action of the group $\Cl(\Oo_{\bm{k}})$ on the set $\LLL$ by 
$\fraka \cdot M = \fraka\otimes_{\Oo_{\bm{k}}} M$. 
To define an action of $\Cl(\Oo_{\bm{k}})$ on $\KKK$, let $\fraka \in \Cl(\Oo_{\bm{k}})$ and let
$\theta : \Oo_{\bm{k}} \to \Oo_B$ be a representative of an $\Oo_B^\times$-conjugacy class of embeddings. Then $\theta(\fraka)\Oo_B = \xi\Oo_B$ for some $\xi \in \Oo_B$ (\cite[Theorem 17.8.3]{Voight}) and we define $\fraka\cdot\theta = \xi^{-1}\theta\xi$. A different choice of element $\xi$ will conjugate the embedding by an element of $\Oo_B^\times$. That $\xi^{-1}\theta\xi$ has image in $\Oo_B$ is verified in \cite[\S 3.2.1]{S}. Next define an action of the Atkin-Lehner group $W_B$ on $\KKK$ by $w \cdot \theta = w\theta w^{-1}$. The next result will give an action of $W_B$ on $\LLL$. 

\begin{proposition}\label{bijection}
There is a $\Cl(\Oo_{\bm{k}})$-equivariant bijection $\KKK \to \LLL$.
\end{proposition}

\begin{proof}
Let $\theta : \Oo_{\bm{k}} \to \Oo_B$ be a representative of an $\Oo_B^\times$-conjugacy class of embeddings and define
$f : \KKK \to \LLL$ by sending $\theta$ to the $\ZZ$-module $M_\theta = \Oo_B$, viewed as a left $\Oo_{\bm{k}}$-module
through $\theta$ (and multiplication on the left) and a right $\Oo_B$-module through multiplication on the right.
The isomorphism class of this $\Oo$-module only depends on $\theta$ through its $\Oo_B^\times$-conjugacy class.
To show $f$ is injective, suppose $\theta, \theta' : \Oo_{\bm{k}} \to \Oo_B$ are two embeddings and suppose
$\varphi : M_{\theta} \to M_{\theta'}$ is an $\Oo$-module isomorphism. By $\Oo_{\bm{k}}$-linearity we have
$\varphi(\theta(a)x) = \theta'(a)\varphi(x)$ for all $a \in \Oo_{\bm{k}}$ and all $x \in \Oo_B$. 
By $\Oo_B$-linearity we have $\varphi \in \End_{\Oo_B}(\Oo_B)^\times$, viewing $\Oo_B$ as a right $\Oo_B$-module. There is an isomorphism of
rings $\Oo_B \to \End_{\Oo_B}(\Oo_B)$ defined by sending $x$ to the
endomorphism $y \mapsto xy$. Therefore there is a $u \in \Oo_B^\times$ such that $\varphi(x) = ux$ for all
$x \in \Oo_B$, so $\theta = u^{-1}\theta'u$.

Next we will show $f$ is $\Cl(\Oo_{\bm{k}})$-equivariant. Let $\fraka \in \Cl(\Oo_{\bm{k}})$ and suppose
$\theta(\fraka)\Oo_B = \xi\Oo_B$ as above, so $\fraka\cdot\theta = \xi^{-1}\theta\xi$. Then
$$
\fraka\otimes_{\Oo_{\bm{k}}} M_\theta \cong \theta(\fraka)\Oo_B = \xi\Oo_B.
$$
It is easily checked that the map $M_{\fraka\cdot\theta} \to \fraka\otimes_{\Oo_{\bm{k}}} M_\theta$
given by $x \mapsto \xi x$ is an isomorphism of $\Oo$-modules, which means $f(\fraka\cdot\theta) \cong
\fraka\cdot f(\theta)$.

To show $f$ is surjective, let $M \in \LLL$ and let $\theta : \Oo_{\bm{k}} \to \Oo_B$ be an embedding 
such that $(M_\theta)_\ell \cong M_\ell$ as $\Oo_{\ell}$-modules for all primes $\ell$. To see that such a $\theta$ exists, note that by Lemma \ref{local modules}, the completions of $M_\theta$ and $M$ are isomorphic at all primes except possibly those dividing $d_B$. If they are not isomorphic at $\ell \mid d_B$, replace $\theta$ with $w_\ell\cdot\theta$, where $w_\ell \in W_B$ is the Atkin-Lehner operator at $\ell$, which interchanges the two isomorphism classes of $\Oo_\ell$-modules. Then there is an $\Oo$-linear isomorphism 
$$
\Hom_{\Oo}(M_\theta, M)  \otimes_{\Oo_{\bm{k}}}   M_\theta \to M
$$
given by $\varphi \otimes x \mapsto \varphi(x)$, where the module on the left has the obvious right 
$\Oo_B$-action through its action on $M_\theta$, and $\Hom_{\Oo}(M_\theta, M)$ is an $\Oo_{\bm{k}}$-module
via the pointwise action on the images of the homomorphisms. That this map is an isomorphism can be checked
by proving it is an isomorphism after completion at each prime number, using that 
$\fraka = \Hom_{\Oo}(M_\theta, M)$ is a fractional $\Oo_{\bm{k}}$-ideal and the isomorphism of $\Oo_{\bm{k}, \ell}$-modules $\fraka \otimes_{\ZZ} \ZZ_\ell \cong \Oo_{\bm{k}, \ell}$ for each prime $\ell$.
Hence $M \cong \fraka \otimes_{\Oo_{\bm{k}}}  M_{\theta} = \fraka\cdot M_\theta$ is in the image of
$f$.
\end{proof}

Define an action of $W_B$ on $\LLL$ by $w\cdot M_\theta = M_{w\cdot\theta}$, so the above bijection is also $W_B$-equivariant by definition.

\begin{proposition}\label{bijection 2}
There is a $W_B \times \Cl(\Oo_{\bm{k}})$-equivariant bijection $\LLL \to [\YYY^B(\CC)]$.
\end{proposition}

\begin{proof}
Let $M \in \LLL$. Then $V = M \otimes_{\ZZ} \RR$
is a $4$-dimensional $\RR$-vector space with $M$ a $\ZZ$-lattice in $V$. The action of $\Oo_{\bm{k}}$ on
$M$ induces a map $\bm{k} \otimes_{\QQ} \RR \cong \CC \to \End_\RR(V)$, turning $V$ into a $\CC$-vector space.
Define $f : \LLL \to [\YYY^B(\CC)]$ by sending $M$ to the CMQM abelian surface with 
complex points $V/M$. Such an abelian surface exists by 
\cite[Lemma 43.6.16]{Voight}. The inverse of $f$ is given by $A \mapsto H_1(A(\CC), \ZZ)$. The map $f$ is $\Cl(\Oo_{\bm{k}})$-equivariant by \cite[Theorem 7.6]{Conrad}. If $w \in W_B$ and $A \in \YYY^B(\CC)$, where $A(\CC) = V/M_\theta$, then $w\cdot A$ has $\CC$-points $V/M'_\theta$, where $M'_\theta$ is the same $\Oo_{\bm{k}}$-module as $M_\theta$, but 
with right $\Oo_B$-action given by $x \cdot z = xw^{-1}zw$. It is easily checked that the map $M_\theta' \to M_{w\theta w^{-1}} = w\cdot M_\theta$ given by $x \mapsto wxw^{-1}$ is an isomorphism of $\Oo$-modules, which implies $f(w\cdot M_\theta) \cong w\cdot f(M_\theta)$.
\end{proof}

Define an equivalence relation on the set $\KKK$ according to $\theta \sim \theta'$ if and only if the 
induced maps $\widetilde{\theta}, \widetilde{\theta}' : \Oo_{\bm{k}} \to \Oo_B/\frakm_B$ are equal. 
Let $\KKK'$ be the set of equivalence classes under this relation.
Under the bijection $\KKK \to \LLL$, this equivalence relation corresponds to the following equivalence relation
on $\LLL$: $M \sim M'$ if and only if $M_\ell \cong M_\ell'$ as $\Oo_{\ell}$-modules for all primes $\ell$
(note by Lemma \ref{local modules} that this is really only a condition at each prime dividing $d_B$).
Let $\LLL'$ be the set of equivalence classes under this relation. 
We know that the group $W_B \times \Cl(\Oo_{\bm{k}})$ acts simply transitively on the set $[\YYY^B(\CC)]$, so its
natural actions on $\KKK$ and $\LLL$ are also simply transitive.

The elements of $\LLL'$ can be thought of as collections of $\Oo_{\ell}$-modules $\{M_\ell\}_\ell$ indexed by the prime
numbers. The action of $W_B$ on $\LLL$ induces an action on $\LLL'$. Explicitly, for $\ell \mid d_B$, the
Atkin-Lehner operator $w_\ell \in W_B$ interchanges the two isomorphism classes of modules $M_\ell$
over $\Oo_{\ell}$. 
It follows that under the action of $W_B \times \Cl(\Oo_{\bm{k}})$ on $\LLL$, the group 
$\Cl(\Oo_{\bm{k}})$ acts simply transitively on each equivalence class under $\sim$ and the group 
$W_B$ acts simply transitively on the set of equivalence classes $\LLL'$. The corresponding results hold
for the set $\KKK$, so in particular $\#\KKK' = |W_B| = 2^r$, where $r$ is the number of primes dividing $d_B$.
Since there are
$2^r$ ring homomorphisms $\Oo_{\bm{k}} \to \Oo_B/\frakm_B$, each such homomorphism arises
as the reduction of a homomorphism $\Oo_{\bm{k}} \to \Oo_B$.

Using the bijections of Corollary \ref{st2}, the equivalence relation $\sim$ on $\KKK$ induces an equivalence relation
on the set $[\YYY^B(k)]$, for any algebraically closed field $k$ over $\Oo_{\bm{K}}$, determined by the following: if $[\theta]$ is the equivalence class of $\theta \in \KKK$,
then $[\theta]$ is in bijection with $[\YYY^B(\widetilde{\theta})(k)]$. It follows that the natural action of $\Cl(\Oo_{\bm{k}})$ on $[\YYY^B(\widetilde{\theta})(k)]$ is simply transitive. 

Suppose $(E, \kappa)$ is an elliptic curve over an $\Oo_{\bm{K}}$-scheme $S$ with CM by $\Oo_{\bm{k}}$ and let
$M \in \LLL$. From
$M$ being a finitely generated projective $\Oo_{\bm{k}}$-module, locally free of rank $2$, there is an abelian scheme
$M \otimes_{\Oo_{\bm{k}}} E \to S$ of relative dimension $2$ with 
$(M \otimes_{\Oo_{\bm{k}}} E)(T) = M \otimes_{\Oo_{\bm{k}}} E(T)$ for any $S$-scheme $T$.
There are commuting actions
$$
i_M : \Oo_B \to \End_S(M \otimes_{\Oo_{\bm{k}}} E), \quad \kappa_M : \Oo_{\bm{k}} \to \End_S(M \otimes_{\Oo_{\bm{k}}} E),
$$
given on points by
$$
i_M(x)(m \otimes z) = x\cdot m \otimes z, \quad \kappa_M(a) (m \otimes z) = m \otimes \kappa(a)(z),
$$
so $M \otimes_{\Oo_{\bm{k}}} E$ is a QM abelian surface over $S$ with complex multiplication by $\Oo_{\bm{k}}$. Note that the Serre construction used here requires $M$ to be a right $\Oo_{\bm{k}}$-module and, as seen in the proof of Proposition \ref{bijection}, we are viewing $M$ as a left $\Oo_{\bm{k}}$-module, but there is no issue since $\Oo_{\bm{k}}$ is commutative.

If $\theta : \Oo_{\bm{k}} \to \Oo_B$ is 
a ring homomorphism, we will sometimes write $\YYY^B([\theta])$ for $\YYY^B(\widetilde{\theta})$.
Recall that $\YYY$ is the stack of all elliptic curves over $\Oo_{\bm{K}}$-schemes with CM by $\Oo_{\bm{k}}$.

\begin{theorem}\label{Serre type}
Fix representatives $\theta_1, \ldots, \theta_m \in \KKK$ of the $m = 2^r$ classes in $\KKK'$.
There is an isomorphism of stacks over $\Spec(\Oo_{\bm{K}})$
$$
\Phi : \bigsqcup_{d = 1}^m\YYY \to \YYY^B
$$
defined by $(E, d) \mapsto M_{\theta_d} \otimes_{\Oo_{\bm{k}}} E$, which induces an isomorphism
$\YYY \to \YYY^B([\theta])$ for any $[\theta] \in \KKK'$.
\end{theorem}

The notation $(E, d)$ means $E$ is an object of the $d$-th copy of $\YYY$ in the disjoint union, and $M_\theta$ is as in the proof of Proposition \ref{bijection}.
Therefore we obtain an isomorphism
$$
\bigsqcup_{\theta : \Oo_{\bm{k}} \to \Oo_B/\frakm_B}\YYY^B(\theta) \to \YYY^B.
$$
In particular, any $A \in \YYY^B(S)$ is isomorphic to $M_\theta \otimes_{\Oo_{\bm{k}}} E$ for some $\theta : \Oo_{\bm{k}}
\to \Oo_B$ and some $E \in \YYY(S)$. Note that if $S = \Spec(\overline{\FF}_\frakP)$, then $A = M_\theta \otimes_{\Oo_{\bm{k}}} E \sim (E')^2$ for some elliptic curve $E'$ over $\overline{\FF}_\frakP$
with $E'$ supersingular if and only if $E$ is supersingular. The proof of Theorem \ref{Serre type} will use the following lemmas.

\begin{lemma}\label{geometric points}
The morphism $\Phi : \bigsqcup_{d = 1}^m\YYY \to \YYY^B$ induces a bijection on geometric points
of $\Spec(\Oo_{\bm{K}})$.
\end{lemma}

\begin{proof}
Let $k$ be an algebraically closed field over $\Oo_{\bm{K}}$ and
$X \subset [\YYY^B(k)]$ be the image of the map
$$
\Phi(k) : \bigsqcup_{d=1}^m[\YYY(k)] \to [\YYY^B(k)]
$$
on $k$-points determined by $\Phi$. The group $W_B \times \Cl(\Oo_{\bm{k}})$
acts simply transitively on $[\YYY^B(k)]$ by Corollary \ref{st2} and this action preserves the subset $X$:
this is true for the action of $\Cl(\Oo_{\bm{k}})$ since this group acts on $[\YYY(k)]$. If $w \in W_B$ then 
$w\cdot (M_{\theta_d} \otimes_{\Oo_{\bm{k}}} E) \cong M_{w\theta_dw^{-1}} \otimes_{\Oo_{\bm{k}}} E$ by the proof of Proposition \ref{bijection 2}. This implies $\Phi(k)$ is surjective. Since
$\Cl(\Oo_{\bm{k}})$ acts simply transitively on $[\YYY(k)]$ (Remark \ref{st remark}), $\Phi(k)$ is a bijection as
\begin{equation*}
\#\bigsqcup_{d=1}^m[\YYY(k)] = m\cdot\#[\YYY(k)] = |W_B|\cdot|\Cl(\Oo_{\bm{k}})| = 
\#[\YYY^B(k)]. \qedhere
\end{equation*}
\end{proof}

\begin{lemma}\label{reduction 2}
The reduction map $\mathscr{R} : [\YYY^B(\widetilde{\theta}_d)(R)] \to [\YYY^B(\widetilde{\theta}_d)(\overline{\FF}_\frakP)]$ is a bijection for any $R \in \mathbf{CLN}_{\bm{K}_\frakP}$ and $1 \lqq d \lqq m$.
\end{lemma}

\begin{proof}
It follows from Proposition \ref{reduction} that $\mathscr{R}$ is injective since $\YYY^B(\widetilde{\theta}_d)(S)$ is a full subcategory of $\YYY^B(S)$ for any $\Oo_{\bm{K}}$-scheme $S$. Now let $X$ be the image of $\mathscr{R}$. The group $\Cl(\Oo_{\bm{k}})$ acts simply transitively on $[\YYY^B(\widetilde{\theta}_d)(\overline{\FF}_\frakP)]$
and the action preserves $X$. To see this, let $\fraka \in \Cl(\Oo_{\bm{k}})$ and
$\mathbf{A} \in X$. Then $\mathbf{A} \cong \mathbf{A}_0 \times_{\Spec(R)} \Spec(\overline{\FF}_\frakP)$ for some
$\mathbf{A}_0 \in [\YYY^B(\widetilde{\theta}_d)(R)]$ and
$$
\fraka \otimes_{\Oo_{\bm{k}}} \mathbf{A} \cong (\fraka \otimes_{\Oo_{\bm{k}}} \mathbf{A}_0) \times_{\Spec(R)} \Spec(\overline{\FF}_\frakP).
$$
Since $\fraka \otimes_{\Oo_{\bm{k}}} \mathbf{A}_0 \in [\YYY^B(\widetilde{\theta}_d)(R)]$, this proves the claim and that $\mathscr{R}$ is surjective.
\end{proof}

\begin{lemma}\label{decomposition 2}
We have $\YYY^B = \bigsqcup_{d=1}^m\YYY^B(\widetilde{\theta}_d)$.
\end{lemma}

\begin{proof}
By Lemma \ref{reduction 2}, the structure morphism $\YYY^B(\widetilde{\theta}_d) \to \Spec(\Oo_{\bm{K}})$ is \'etale. It is also quasi-finite by Corollary \ref{st2} and proper by a proof similar to that of \cite[Proposition 3.3.5]{Howard3}, so the structure morphism $\bigsqcup_{d=1}^m\YYY^B(\widetilde{\theta}_d) \to \Spec(\Oo_{\bm{K}})$ is finite \'etale. As $\YYY^B \to \Spec(\Oo_{\bm{K}})$ is also finite \'etate, the inclusion functor $I : \bigsqcup_{d=1}^m\YYY^B(\widetilde{\theta}_d) \to \YYY^B$, which is an immersion of stacks over $\Spec(\Oo_{\bm{K}})$, is finite \'etale, and in particular, an open immersion. By Corollary \ref{max ideal} below, whose proof only uses Lemma \ref{geometric points}, the morphism
$I$ induces a bijection on geometric points of $\Spec(\Oo_{\bm{K}})$ and hence is a surjective morphism of stacks by \cite[Tag 06D7]{Stacks}. As $I$ is a surjective open immersion, it is an isomorphism.
\end{proof}

\begin{proof}[Proof of Theorem {\upshape \ref{Serre type}}]
The idea of the proof is to introduce level structure to the stacks $\YYY$ and $\YYY^B$, show that these
new spaces are schemes, and then show $\Phi$ induces an isomorphism between these schemes. 
Fix an integer $n \gqq 1$ and set $S = \Spec(\Oo_{\bm{K}})$ and $S_n = \Spec(\Oo_{\bm{K}}[n^{-1}])$. 
For $\theta \in \KKK$, let $(\Oo_B/(n))_T^\theta$ be the constant
group scheme over the $S_n$-scheme $T$ associated with $\Oo_B/(n)$. Here we are viewing $\Oo_B/(n)$ as a right $\Oo_B$-module through multiplication on the right and a left $\Oo_{\bm{k}}$-module through $\theta: \Oo_{\bm{k}} \to \Oo_B$ and multiplication on the left. 
For $n$ prime to $d_B$ and $\theta \in \KKK$, define $\YYY^B(\widetilde{\theta}, n)$ to be the category fibered in groupoids over $S_n$ with $\YYY^B(\widetilde{\theta}, n)(T)$ the category of quadruples
$(A, i, \kappa, \nu)$ where $(A, i, \kappa) \in \YYY^B(\widetilde{\theta})(T)$ and 
$$
\nu : (\Oo_B/(n))_T^\theta \to A[n]
$$
is an $\Oo$-linear isomorphism of group schemes, and define
$$
\YYY^B(n) = \bigsqcup_{d=1}^m\YYY^B(\widetilde{\theta}_d, n).
$$
Forgetting $\nu$ defines a
finite \'etale representable morphism $\YYY^B(n) \to \YYY^B \times_S S_n$, so
$\YYY^B(n)$ is a stack, finite \'etale over $S_n$. 

We claim that for $n \gqq 3$ prime to $d_B$, any object of $\YYY^B(n)$ over a connected base has no nontrivial automorphisms. Let 
$T$ be a connected $S_n$-scheme, let $(A, i, \kappa, \nu) \in \YYY^B(n)(T)$, and suppose $g \in \Aut(A, i, \kappa, \nu)$.
Set $g' = g - \id_A$ and suppose $g' \neq 0$. Since $g = \kappa(a)$ for some $a \in \Oo_{\bm{k}}^\times$ (see the proof of Lemma \ref{free action}), the morphism
$g' = \kappa(a - 1)$ is an isogeny of QM abelian surfaces. Then
$$
[\deg^\ast(g')] = (g')^t \circ g' = g^t \circ g - g^t - g + \id_A = 2\cdot\id_A - (g^t + g),
$$
so $g^t + g = [c]$, where $c = 2 - \deg^\ast(g')$. Hence $g$ is a root of the polynomial $x^2 - cx + 1$
in $\End_T(A)[x]$. But $g$ is a root of unity, which means $|c| \lqq 2$ and thus $\deg^\ast(g') \lqq 4$.
By definition of $g$ being an automorphism of $(A, i, \kappa, \nu)$, the endomorphism 
$g'$ kills $A[n]$, so $g' = g'' \circ [n]$ for some $g'' \in \End_{\Oo_B}(A)$. Then $|n^2\deg^\ast(g'')| \lqq 4$ and
since $n \gqq 3$, we must have $\deg^\ast(g'') = 0$, a contradiction. It follows from this fact, as in
\cite[proof of Corollary 2.3]{Buzzard}, that $\YYY^B(n)$ is a scheme. 

For any $n \gqq 1$ define $\YYY(n)$ to be the category fibered in groupoids over $S_n$ with $\YYY(n)(T)$
the category of triples $(E, \kappa, \nu)$ where $(E, \kappa) \in \YYY(T)$ and 
$$
\nu : (\Oo_{\bm{k}}/(n))_T \to E[n]
$$
is an $\Oo_{\bm{k}}$-linear isomorphism of group schemes. As above, $\YYY(n)$ is a 
scheme, finite \'etale over $S_n$. Let $G_n = \Aut_{\Oo_{\bm{k}}}(\Oo_{\bm{k}}/(n)) \cong (\Oo_{\bm{k}}/(n))^\times$. There is an action of the finite
group scheme $(G_n)_{S_n}$ on the scheme $\YYY(n)$, defined on $T$-points, for any connected $S_n$-scheme $T$, by
$$
g \cdot (E, \kappa, \nu) = (E, \kappa, \nu \circ g^{-1}).
$$
There is an associated quotient stack $\YYY(n)/(G_n)_{S_n} \to S_n$, defined in \cite[Example 7.17]{Vistoli}, and
there is an isomorphism of stacks $\YYY(n)/(G_n)_{S_n} \to \YYY \times_S S_n$ such that the composition
$$
\YYY(n) \to \YYY(n)/(G_n)_{S_n} \map{\cong} \YYY \times_S S_n
$$
is the morphism defined by forgetting the level structure. 

Note that there is an isomorphism of groups 
$\Aut_{\Oo}(\Oo_B/(n)) \cong (\Oo_{\bm{k}}/(n))^\times$, viewing $\Oo_B/(n)$ as a left $\Oo_{\bm{k}}$-module through some $\theta \in \KKK$,
so $(G_n)_{S_n}$ also acts on $\YYY^B(\widetilde{\theta}, n)$, the action defined in the same way as above. As before there is an isomorphism
of stacks $\YYY^B(\widetilde{\theta}, n)/(G_n)_{S_n} \to \YYY^B(\widetilde{\theta}) \times_S S_n$ such that the composition
$$
\YYY^B(\widetilde{\theta}, n) \to \YYY^B(\widetilde{\theta}, n)/(G_n)_{S_n} \map{\cong} \YYY^B(\widetilde{\theta}) \times_S S_n
$$
is the forgetful morphism. The base change
$$
\Phi_n = \Phi \times \id : \bigsqcup_{d=1}^m\YYY \times_S S_n \to \YYY^B \times_S S_n
$$
induces a morphism of schemes over $S_n$
$$
\Phi_n' : \bigsqcup_{d=1}^m \YYY(n) \to \YYY^B(n)
$$
given on $T$-points by $(E, \nu, d) \mapsto (M_{\theta_d} \otimes_{\Oo_{\bm{k}}} E, M_{\theta_d}\otimes\nu)$, where $M_{\theta_d}\otimes\nu$ is the composition
\begin{equation}\label{level structure}
(\Oo_B/(n))_T^{\theta_d} \cong M_{\theta_d} \otimes_{\Oo_{\bm{k}}} (\Oo_{\bm{k}}/(n))_T \map{\id \otimes \nu}
M_{\theta_d} \otimes_{\Oo_{\bm{k}}} E[n] \cong (M_{\theta_d} \otimes_{\Oo_{\bm{k}}} E)[n].
\end{equation}

Let $k$ be an algebraically closed field over $\Oo_{\bm{K}}[n^{-1}]$ and fix a triple $(A, i, \kappa) \in \YYY^B(k)$, so $A \cong
M_{\theta_d} \otimes_{\Oo_{\bm{k}}} E$ for some $d$ and some $E \in \YYY(k)$, by Lemma \ref{geometric points}. Let $X$ be the set of all $\Oo$-linear isomorphisms of group schemes
$\mu : (\Oo_B/(n))_k^{\theta_d} \to A[n]$, where two such isomorphisms $\mu$ and $\mu'$ are considered equal in $X$
if the objects $(A, i, \kappa, \mu)$ and $(A, i, \kappa, \mu')$ are isomorphic in $\YYY^B(\widetilde{\theta}_d, n)(k)$. The group
$G_n$ acts simply transitively on $X$, the action as above, and this action preserves the subset of $X$ consisting of all $M_{\theta_d}\otimes\nu$ as in (\ref{level structure}) for some
$\Oo_{\bm{k}}$-linear isomorphism $\nu : (\Oo_{\bm{k}}/(n))_k \to E[n]$ since $a \cdot (\id \otimes \nu) = \id \otimes (\nu \circ m_{a^{-1}})$ for any $a \in (\Oo_{\bm{k}}/(n))^\times$, where $m_{a^{-1}}$ is  multiplication by $a^{-1}$. Combining this with
Lemma \ref{geometric points}, it follows that
$\Phi_n'$ determines a bijection $[\YYY(n)(k)] \to [\YYY^B(\widetilde{\theta}_d, n)(k)]$ for any $1 \lqq d \lqq m$. Hence $\Phi_n'$ determines a bijection
$$
(\Phi'_n)(k) : \bigsqcup_{d=1}^m[\YYY(n)(k)] \to \bigsqcup_{d=1}^m[\YYY^B(\widetilde{\theta}_d, n)] = [\YYY^B(n)(k)].
$$
The morphism $\Phi'_n$ is $(G_n)_{S_n}$-equivariant, so there is a morphism of stacks
$$
\bigsqcup_{d=1}^m\YYY(n)/(G_n)_{S_n} \to \bigsqcup_{d=1}^m\YYY^B(\widetilde{\theta}_d, n)/(G_n)_{S_n}
$$
inducing $\Phi_n$ under the isomorphisms described above, using Lemma \ref{decomposition 2}.
It follows that to show $\Phi_n$ is an isomorphism, it suffices to show $\Phi_n'$ is an isomorphism.
As $\Phi_n'$ is a finite \'etale morphism of
$S_n$-schemes inducing a bijection on geometric points of $S_n$, it is an isomorphism.
Choosing relatively prime integers $n, n' \gqq 3$ prime to $d_B$, the morphisms $\Phi_n$ and $\Phi_{n'}$ being isomorphisms implies $\Phi$ is an isomorphism.

For the final statement of the theorem, let $S$ be any $\Oo_{\bm{K}}$-scheme and fix an integer $1 \lqq d \lqq m$. 
It follows directly from the definitions that any CMQM abelian surface of the form $M_{\theta_d} \otimes_{\Oo_{\bm{k}}} E$ for some $E \in \YYY(S)$ lies in $\YYY^B([\theta_d])(S)$. Conversely, 
suppose $(A, i, \kappa) \in \YYY^B([\theta_d])(S)$. Since $\Phi$ is an isomorphism, $A \cong M_{\theta_{e}} \otimes_{\Oo_{\bm{k}}} E$ for some 
$E \in \YYY(S)$ and a unique $1 \lqq e \lqq m$, so the diagram (D$(A, \eta)$) of Definition \ref{theta stack}
commutes for $\eta = \widetilde{\theta}_d$ and $\eta = \widetilde{\theta}_e$. 
Picking any geometric point $\overline{s}$ of $S$, the same diagram still commutes with $A$ replaced with $A_{\overline{s}}$. But the induced map $i_{\overline{s}}^{\frakm_B}$
is an isomorphism by Corollary \ref{max ideal}. Therefore $d = e$, which shows $\Phi$ defines an isomorphism $\YYY \to \YYY^B([\theta_d])$.
\end{proof}

\begin{corollary}\label{Kottwitz 2}
Suppose $S$ is an $\Oo_{\bm{K}}$-scheme and let $(A, i, \kappa) \in \YYY^B(S)$. The trace of $i(x)$
acting on $\Lie(A)$ is equal to $\Trd(x)$ for any $x \in \Oo_B$.
\end{corollary}

What this means is that each point of $S$ has an affine open neighborhood $U$ such that
the trace of $i(x)$ acting on the free $\OOO_U$-module $\Lie(A \times_S U)$ is equal to $\Trd(x)$ for any $x \in \Oo_B$.

\begin{proof}
By Theorem \ref{Serre type}, $A \cong \Oo_B \otimes_{\Oo_{\bm{k}}} E$ for some $E \in \YYY(S)$, so
$\Lie(A) \cong \Oo_B \otimes_{\Oo_{\bm{k}}} \Lie(E)$ as $\Oo \otimes_{\Oo_{\bm{k}}} \OOO_S$-modules, with $\Oo_B$ acting on $\Oo_B \otimes_{\Oo_{\bm{k}}} \Lie(E)$ through right multiplication on $\Oo_B$. Let $s \in S$ and $U \subset S$ be an affine open neighborhood of $s$ such that $\Lie(E \times_S U)$ is a free $\OOO_U$-module of rank $1$.
Then the natural map $\Oo_B \otimes_{\Oo_{\bm{k}}} \OOO_U \to \mathscr{E}nd_{\OOO_U}(\Lie(A \times_S U))$ is an isomorphism of $\OOO_U$-modules and we may view $\Oo_B \otimes_{\Oo_{\bm{k}}} \OOO_U$ as $\MM_2(\OOO_U)$ acting on $\Lie(A\times_SU) \cong \OOO_U \oplus \OOO_U$. Hence $\Tr_{\Lie(A\times_SU)/\OOO_U}(i(x)) = \Tr(i(x)) = \Trd(x)$, where the middle term is matrix trace.
\end{proof}

\begin{corollary}\label{decomp lift}
Suppose $\widetilde{R} \to R$ is a surjection in $\mathbf{CLN}_{\bm{K}_\frakP}$, $\mathbf{A}  \in \YYY^B(R)$, and 
$\widetilde{\mathbf{A}} \in \YYY^B(\widetilde{R})$ is a deformation of $\mathbf{A}$.
Let $\theta : \Oo_{\bm{k}} \to \Oo_B/\frakm_B$ be a ring homomorphism.
Then $\mathbf{A} \in \YYY^B(\theta)(R)$ if and only if $\widetilde{\mathbf{A}} \in \YYY^B(\theta)(\widetilde{R})$.
\end{corollary}

\begin{proof}
First suppose $\mathbf{A} \in \YYY^B(\theta)(R)$.
By Theorem \ref{Serre type}, $\widetilde{\mathbf{A}} \in \YYY^B(\eta)(\widetilde{R})$ for a unique $\eta : \Oo_{\bm{k}} \to \Oo_B/\frakm_B$, so the
diagram (D$(\widetilde{A}, \eta)$) of Definition \ref{theta stack} commutes. 
Since $\widetilde{\mathbf{A}}$ is a deformation of $\mathbf{A}$, the diagram (D$(A, \eta)$) then commutes. But $\mathbf{A} \in \YYY^B(\theta)(R)$ implies the diagram (D$(A, \theta)$) also commutes, so
$\eta = \theta$ as in the last part of the proof of Theorem \ref{Serre type}. Conversely, if $\widetilde{\mathbf{A}} \in 
\YYY^B(\theta)(\widetilde{R})$ then (D$(\widetilde{A}, \theta)$) commutes, so 
as before it follows that (D$(A, \theta)$) commutes and hence $\mathbf{A} \in \YYY^B(\theta)(R)$.
\end{proof}

\subsection{Tate and Dieudonn\'e modules}

In this section we collect some basic results on Tate and Dieudonn\'e modules of QM abelian surfaces that will be used later. 

\begin{lemma}\label{Tate modules}
Let $\ell$ be a prime number and suppose $A_1$ and $A_2$ are QM abelian surfaces
over a field $k$ with $\charr(k) \neq \ell$. \\
{\upshape (a)} Suppose $\ell \nmid d_B$ and set
$$
\varepsilon = \begin{bmatrix} 1 & 0 \\ 0 & 0 \end{bmatrix}, \quad \varepsilon' = \begin{bmatrix} 0 & 0 \\ 0 & 1
\end{bmatrix}
$$
in $\MM_2(\ZZ_{\ell}) \cong \Oo_B \otimes_{\ZZ} \ZZ_{\ell}$. There are isomorphisms of $\ZZ_{\ell}$-modules
$$
\Hom_{\Oo_{B, \ell}}(T_{\ell}(A_1), T_{\ell}(A_2)) \cong \Hom_{\ZZ_{\ell}}(\varepsilon T_{\ell}(A_1), \varepsilon T_{\ell}(A_2))
\cong \MM_2(\ZZ_{\ell}).
$$
If $A_1 = A_2$ then these are
isomorphisms of rings.  \\
{\upshape (b)} If $\ell \mid d_B$ then there is an isomorphism of $\ZZ_{\ell}$-modules
$$
\Hom_{\Oo_{B, \ell}}(T_{\ell}(A_1), T_{\ell}(A_2)) \cong \Oo_{B, \ell},
$$
which is an isomorphism of rings if $A_1 = A_2$.
\end{lemma}

\begin{proof}
(a) For $j = 1,2$ write $T_j$ for $T_{\ell}(A_j)$. From $T_j = \varepsilon T_j \oplus \varepsilon' T_j$, there is an
inclusion
$$
\Hom_{\Oo_{B, \ell}}(T_1, T_2) \hookrightarrow \Hom_{\ZZ_{\ell}}(\varepsilon T_1, \varepsilon T_2) \oplus
\Hom_{\ZZ_{\ell}}(\varepsilon' T_1, \varepsilon' T_2).
$$
Denote this map
by $f \mapsto (f_{\varepsilon}, f_{\varepsilon'})$ and let 
$$
w = \begin{bmatrix} 0 & 1 \\ 1 & 0 \end{bmatrix} \in \MM_2(\ZZ_{\ell}).
$$
Then for any $x \in T_1$, $f_{\varepsilon'}(\varepsilon'x) = wf_{\varepsilon}(\varepsilon wx)$,
which means $f_{\varepsilon}$ determines $f_{\varepsilon'}$. Therefore the above map is really an inclusion $\Hom_{\Oo_{B, \ell}}(T_1, T_2) \hookrightarrow \Hom_{\ZZ_{\ell}}(\varepsilon T_1, \varepsilon T_2)$.
To show this map is surjective, let $f_{\varepsilon} \in \Hom_{\ZZ_{\ell}}(\varepsilon T_1, \varepsilon T_2)$.
Define $f_{\varepsilon'} \in \Hom_{\ZZ_{\ell}}(\varepsilon' T_1, \varepsilon' T_2)$ by $f_{\varepsilon'}(\varepsilon'x)
= wf_{\varepsilon}(\varepsilon wx)$, and define $f : T_1 \to T_2$ by $f = f_{\varepsilon} \oplus f_{\varepsilon'}$.
It is simple to check that $f$ is $\MM_2(\ZZ_{\ell})$-linear, and by construction, $f \mapsto f_{\varepsilon}$.

(b) Since $\ell \mid d_B$, $B_{\ell}$ is a quaternion division algebra over $\QQ_{\ell}$, and from $T_j$
being free of rank $4$ as a $\ZZ_{\ell}$-module, $T_j \otimes_{\ZZ_{\ell}} \QQ_{\ell}$ is a free $B_{\ell}$-module
of rank $1$. Choosing a generator, we obtain an isomorphism of $B_{\ell}$-modules $T_j \otimes_{\ZZ_{\ell}}\QQ_{\ell}
\cong B_{\ell}$, which identifies $T_j$ with a finitely generated $\Oo_{B, \ell}$-submodule of $B_{\ell}$.
Multiplying $T_j$ by a suitably large power of $\ell$ gives an isomorphism of $T_j$ with a finitely generated
$\Oo_{B, \ell}$-submodule of $\Oo_{B, \ell}$, that is, a left ideal in $\Oo_{B, \ell}$. Since all ideals of
$\Oo_{B, \ell}$ are principal, $T_j \cong \Oo_{B, \ell}$ as a left $\Oo_{B, \ell}$-module. Hence
$$
\Hom_{\Oo_{B, \ell}}(T_1, T_2) \cong \End_{\Oo_{B, \ell}}(\Oo_{B, \ell}) \cong \Oo_{B, \ell}^{\text{op}} \cong \Oo_{B, \ell}
$$
as $\ZZ_{\ell}$-modules, where the isomorphism $\Oo_{B, \ell} \to \Oo_{B, \ell}^{\text{op}}$ is given by the main
involution.
\end{proof}

A QM abelian surface $(A, i)$ over a field of characteristic $p$ is \textit{supersingular} if the underlying abelian variety $A$ is
supersingular, that is, $A \sim E^2$ for some supersingular elliptic curve $E$.

\begin{lemma} \label{Tate modules 2}
Let $A_1$ and $A_2$ be supersingular QM abelian surfaces over $\overline{\FF}_p$. 
For any prime $\ell \neq p$ the natural map
$$
\Phi : \Hom_{\Oo_B}(A_1, A_2) \otimes_{\ZZ} \ZZ_{\ell} \to \Hom_{\Oo_{B. \ell}}(T_{\ell}(A_1), T_{\ell}(A_2))
$$
is an isomorphism of $\ZZ_{\ell}$-modules, and is an isomorphism of rings if $A_1 = A_2$.
\end{lemma}

\begin{proof}
The map $\Phi$ is injective for any abelian varieties by \cite[\S 19, Theorem 3]{Mumford}. For $j = 1, 2$ write $T_j$ for $T_{\ell}(A_j)$, and let $M = \imm(\Phi)$. We claim 
$\Hom_{\Oo_{B, \ell}}(T_1, T_2)/M$ is a torsion-free $\ZZ_{\ell}$-module.
Suppose $f \in \Hom_{\Oo_{B, \ell}}(T_1, T_2)$ satisfies $\ell f \in M$. Then $\ell f = \Phi(\varphi)$
for some $\varphi \in \Hom_{\Oo_B}(A_1, A_2) \otimes_{\ZZ} \ZZ_{\ell}$, which means $\varphi$ vanishes on $A_1[\ell](\overline{\FF}_p)$. Hence $(\ker\varphi)(\overline{\FF}_p) \supset (\ker\hspace{.5mm}[\ell])(\overline{\FF}_p)$, 
and thus there is a $\varphi' \in \Hom_{\overline{\FF}_p}(A_1, A_2)$ such that 
$\varphi = \varphi' \circ [\ell] = \ell\varphi'$. 
Note that $\varphi$ being $\Oo_B$-linear implies $\varphi'$ is also 
$\Oo_B$-linear. Then
$\ell\Phi(\varphi') = \Phi(\ell\varphi') = \Phi(\varphi) = \ell f$, so $f = \Phi(\varphi') \in M$ since
$\Hom_{\Oo_{B, \ell}}(T_1, T_2)$ is a torsion-free $\ZZ_{\ell}$-module. 

As $A_1$ and $A_2$ are supersingular, $\Hom_{\Oo_B}(A_1, A_2) \otimes_{\ZZ} \ZZ_{\ell}$ is a free 
$\ZZ_{\ell}$-module of rank $4$ since it is a lattice in $\Hom_{\Oo_B}(A_1, A_2) \otimes_{\ZZ} \QQ_\ell \cong \End_{\Oo_B}(A_1) \otimes_{\ZZ}
\QQ_\ell \cong B^{(p)}_\ell$, where the first isomorphism comes from choosing an isogeny $A_1 \to A_2$ of QM abelian surfaces (there is an isogeny of abelian varieties $A_1 \to A_2$ since they are supersingular and hence there is an isogeny of QM abelian surfaces by \cite[Remark 5.3]{Milne}).
By Lemma \ref{Tate modules}, $\Hom_{\Oo_{B, \ell}}(T_1, T_2)$ is also a free $\ZZ_\ell$-module of rank $4$, so $\Phi$ is an isomorphism.
\end{proof}

Fix a prime number $p$ and let $W = W(\overline{\FF}_p)$ be the ring of Witt vectors over $\overline{\FF}_p$, so
$W$ is the ring of integers in the completion of the maximal unramified extension of $\QQ_p$.
If $A$ is a QM abelian surface over $\overline{\FF}_p$, we write $D(A)$ for the covariant Dieudonn\'e module of
$A$ (that is, the Dieudonn\'e module of $A[p^{\infty}]$), which is a module over the Dieudonn\'e ring $\DDD$, free of rank $4$ over $W$. Recall that there is a unique continuous ring automorphism $\sigma$ of $W$ inducing the absolute Frobenius $x \mapsto x^p$ on $W/pW \cong \overline{\FF}_p$, and $\DDD = W\{\FFF, \VVV\}/(\FFF\VVV - p)$ where $W\{\FFF, \VVV\}$ is the non-commutative
polynomial ring in two commuting variables $\FFF$ and $\VVV$ satisfying $\FFF x = \sigma(x)\FFF$ and 
$\VVV x = \sigma^{-1}(x)\VVV$
for all $x \in W$. 

If $E$ is a supersingular elliptic curve over $\overline{\FF}_p$, then the natural map
$$
\End_{\overline{\FF}_p}(E) \otimes_{\ZZ} \ZZ_p \to \End_{\DDD}(D(E)) \cong \Delta
$$
is an isomorphism of $\ZZ_p$-algebras, where $\Delta$ is the unique maximal order in the quaternion division algebra over $\QQ_p$.

\begin{lemma} \label{Dieudonne 1}
Let $A_1$ and $A_2$ be QM abelian surfaces over $\overline{\FF}_p$. Suppose $p \nmid d_B$ and set
$$
\varepsilon = \begin{bmatrix} 1 & 0 \\ 0 & 0 \end{bmatrix}
$$
in $\MM_2(W) \cong \Oo_B \otimes_{\ZZ} W$. There is an isomorphism
of $W$-modules
$$
\Hom_{\Oo_B \otimes_{\ZZ} W}(D(A_1), D(A_2)) \cong \Hom_W(\varepsilon D(A_1), \varepsilon D(A_2)) \cong \MM_2(W),
$$
which is an isomorphism of rings if $A_1 = A_2$. In particular, if $A_1$ and $A_2$ are supersingular, then
$$
\Hom_{\Oo_B \otimes_{\ZZ} \DDD}(D(A_1), D(A_2)) \cong \Delta.
$$
\end{lemma}

\begin{proof}
The proof of the first part is identical to that of Lemma \ref{Tate modules}(a), replacing $\ZZ_{\ell}$-linearity with $W$-linearity.
For the in particular statement, note that
$$
\Hom_{\Oo_B \otimes_{\ZZ} \DDD}(D(A_1), D(A_2)) = \{\varphi \in \MM_2(W) : \FFF\varphi = \varphi^{\sigma}\FFF, \hspace{1mm} \VVV\varphi = \varphi^{\sigma^{-1}}\VVV\},
$$
where $\varphi^{\sigma}$ is the matrix obtained by applying $\sigma$ to all of the entries.
Since $A_j$ is supersingular, $\varepsilon D(A_j)$ is free of rank of $2$ over $W$ with basis $\{e_1, e_2\}$
satisfying $\FFF(e_1) = \VVV(e_1) = e_2$ and $\FFF(e_2) = \VVV(e_2) = pe_1$.
A computation in coordinates then shows
\begin{equation*}
\{\varphi \in \MM_2(W) : \FFF\varphi = \varphi^{\sigma}\FFF, \hspace{1mm} \VVV\varphi = \varphi^{\sigma^{-1}}\VVV\}
= \left\{\begin{bmatrix} a & pb \\ \overline{b} & \overline{a} \end{bmatrix} : a, b \in \ZZ_{p^2} \right\} \cong \Delta.  \qedhere
\end{equation*}
\end{proof}

\begin{lemma} \label{Dieudonne 2}
If $A_1$ and $A_2$ are supersingular QM abelian surfaces over $\overline{\FF}_p$, then the natural map
$$
\Hom_{\Oo_B}(A_1, A_2) \otimes_{\ZZ} \ZZ_p \to \Hom_{\Oo_B \otimes_{\ZZ}\DDD}(D(A_1), D(A_2))
$$ 
is an isomorphism of $\ZZ_p$-modules, and is an isomorphism of rings if $A_1 = A_2$.
\end{lemma}

\begin{proof}
The map is injective for any abelian varieties by \cite[Proposition 1.2.5.1]{CCO}.
The proof of surjectivity is very similar to that of Lemma \ref{Tate modules 2}, using the following fact: the group
$H = \Hom_{\Oo_B \otimes_{\ZZ} \DDD}(D(A_1), D(A_2))$ is free $\ZZ_p$-module of rank $4$. To see this, consider
the $\QQ_p$-vector space $H \otimes_{\ZZ_p} \QQ_p$. Since $A_1 \sim A_2$ as QM abelian surfaces, $D(A_1) \otimes_{\ZZ_p} \QQ_p \cong
D(A_2) \otimes_{\ZZ_p} \QQ_p$ as $\Oo_B \otimes_{\ZZ} \DDD$-modules, so
$$
H \otimes_{\ZZ_p} \QQ_p \cong \End_{\Oo_B \otimes_{\ZZ}\DDD}(D(A_1)) \otimes_{\ZZ_p} \QQ_p.
$$
As $A_1 \sim E_1^2$ for some supersingular elliptic curve $E_1$ over $\overline{\FF}_p$, we have
$D(A_1) \otimes_{\ZZ_p} \QQ_p \cong D(E_1)^2 \otimes_{\ZZ_p} \QQ_p$ as $\DDD$-modules and thus there are isomorphisms
\begin{align*}
\End_{\DDD}(D(A_1)) \otimes_{\ZZ_p} \QQ_p &\cong \MM_2(\End_{\DDD}(D(E_1))) \otimes_{\ZZ_p} \QQ_p
\cong \MM_2(\End_{\overline{\FF}_p}(E_1)) \otimes_{\ZZ} \QQ_p \\
&\cong \End_{\overline{\FF}_p}(A_1) \otimes_{\ZZ} \QQ_p.
\end{align*}
Taking the centralizer of $\Oo_B$ in each ring shows $H \otimes_{\ZZ_p} \QQ_p \cong \End_{\Oo_B}(A_1) \otimes_{\ZZ}
\QQ_p$ has $\QQ_p$-dimension $4$.
\end{proof}

Let $A \in \YYY^B(\overline{\FF}_\frakP)$,
so $A \cong M \otimes_{\Oo_{\bm{k}}} E$ for some $E \in \YYY(\overline{\FF}_\frakP)$ and some module
$M$ over $\Oo = \Oo_B \otimes_{\ZZ} \Oo_{\bm{k}}$, free of rank $4$ over $\ZZ$. Let $p$ be the rational prime below $\frakP$. 
There is an isomorphism of 
$W \otimes_{\ZZ_p} \Oo_p$-modules
$$
D(A) \cong M_p \otimes_{\Oo_{\bm{k}, p}} D(E).
$$
However, $M_p \cong \Oo_{\bm{k}, p} \oplus \Oo_{\bm{k}, p}$ as $\Oo_{\bm{k}, p}$-modules and thus $D(A) \cong D(E) \oplus
D(E)$ as modules over $W \otimes_{\ZZ_p} \Oo_{\bm{k}, p}$, where $\Oo_{\bm{k}, p}$ acts on $D(E) \oplus D(E)$ diagonally through its action on $D(E)$. We still have to determine the possibilities for the actions of $\Oo_{B, p}$ and $\DDD$ on $D(A)$.

\begin{lemma}\label{supersingular2}
If $p$ is a prime dividing $d_B$, or more generally, a prime nonsplit in $\bm{k}$, then any $A \in \YYY^B(\overline{\FF}_\frakP)$ is supersingular.
\end{lemma}

\begin{proof}
By Proposition \ref{char p} there are two possibilities for $A$ up to isogeny. Suppose $A \sim E^2$ for some ordinary
elliptic curve $E$ over $\overline{\FF}_\frakP$. Then $\End^0_{\overline{\FF}_\frakP}(E) \cong L$ for some imaginary quadratic
field $L$ and $\End^0_{\Oo_B}(A) \cong L$. But $\bm{k} \hookrightarrow \End^0_{\Oo_B}(A)$, so $L \cong \bm{k}$.
Tensoring the $p$-adic representation $\End_{\overline{\FF}_\frakP}(E) \to \End_{\ZZ_p}(T_p(E))$ with $\QQ_p$ gives a $\QQ_p$-algebra
homomorphism $\bm{k}_p = \bm{k} \otimes_{\QQ} \QQ_p \to \QQ_p$.
This map cannot be injective by counting dimensions, so $\bm{k}_p$ is not a field, which means $p$ is split in $\bm{k}$.
\end{proof}

\begin{proposition}\label{endomorphisms}
Suppose $A \in \YYY^B(\overline{\FF}_\frakP)$ for $p \mid d_B$, with $A \cong 
M \otimes_{\Oo_{\bm{k}}} E$ for some supersingular $E$. 
Fix an isomorphism $\Oo_{B, p} \cong \Delta$ and a uniformizer $\Pi \in \Delta$
satisfying $\Pi^2 = p$ and $\Pi a = \overline{a}\Pi$ for all $a \in \ZZ_{p^2}$, where we are viewing $\ZZ_{p^2} \hookrightarrow \Delta$ through the CM action $\Oo_{\bm{k}, p} \to \End_{\overline{\FF}_\frakP}(E) \otimes_{\ZZ} \ZZ_p \cong \Delta$. Then there is an isomorphism of rings
$\End_{\Oo_B}(A) \otimes_{\ZZ} \ZZ_p \cong R_{11}$, where
$$
R_{11} = \left\{\begin{bmatrix} x & y\Pi \\ py\Pi & x\end{bmatrix} : x, y \in \ZZ_{p^2}\right\} \subset \MM_2(\Delta).
$$
\end{proposition}

\begin{proof}
There is the $\Delta$-action on $D(A)$
$$
D(i) : \Delta \to \End_{\Oo_{\bm{k}} \otimes_{\ZZ} \DDD}(D(A)) \cong \MM_2(\End_{\Oo_{\bm{k}} \otimes_{\ZZ} \DDD}(D(E))) \cong \MM_2(\ZZ_{p^2}).
$$
By Lemma \ref{types} there are two possibilities for $D(i)$ up to $\GL_2(\ZZ_{p^2})$-conjugacy, $f_1$ and $f_2$, and we may assume $D(i)$ is equal to $f_1$ or $f_2$ in computing 
$$
\End_{\Oo_B}(A) \otimes_{\ZZ} \ZZ_p \cong \End_{\Oo_B \otimes_{\ZZ} \DDD}(D(A)) \cong C_{\MM_2(\Delta)}(\Delta).
$$
If $D(i) = f_1$ then a computation shows $C_{\MM_2(\Delta)}(\Delta) = R_{11}$.
In the case of $D(i) = f_2$ we have
$C_{\MM_2(\Delta)}(\Delta) = R_{22}$, where
\begin{equation*}
R_{22} = \left\{\begin{bmatrix} x & py\Pi \\ y\Pi & x\end{bmatrix} : x, y \in \ZZ_{p^2}\right\} \cong R_{11}. \qedhere
\end{equation*}
\end{proof}

For $p \mid d_B$ there are two isomorphism classes of modules over $W \otimes_{\ZZ_p} \Oo_p$ that are free
of rank $4$ over $W$ (Lemma \ref{local modules}), and the proof of the previous proposition gives us explicit coordinates for each of these modules,
which we will use for the $W \otimes_{\ZZ_p} \Oo_p$-module $D(A)$. To describe this, identify
$\Delta$ with a subring of $\MM_2(\ZZ_{p^2}) \subset \MM_2(W)$ by
\begin{equation}\label{matrix rep}
a + b\Pi \mapsto \begin{bmatrix} a & pb \\ \overline{b} & \overline{a} \end{bmatrix},
\end{equation}
and use this to view $\ZZ_{p^2} \subset \Delta$ inside $\MM_2(\ZZ_{p^2})$.
Then there is a basis $\{e_n\}$ for the rank 4 free $W$-module $D(A) \cong D(E) \oplus D(E)$ relative to which the $\Delta$-action on $D(A)$ is given by one of the two maps $f_1, f_2 : \Delta \to \End_W(D(A)) \cong \MM_4(W)$ of Lemma \ref{types}:
\begin{equation}\label{OB action}
f_1(a + b\Pi) = \begin{bmatrix} a & 0 & b & 0 \\ 0 & \overline{a} & 0 & \overline{b} \\ p\overline{b} & 0 & \overline{a} & 0 \\
0 & pb & 0 & a \end{bmatrix}, \quad
f_2(a + b\Pi) = \begin{bmatrix} a & 0 & pb & 0 \\ 0 & \overline{a} & 0 & p\overline{b} \\ \overline{b} & 0 & \overline{a} & 0 \\
0 & b & 0 & a \end{bmatrix}.
\end{equation}
The action of $\Oo_{\bm{k}, p} \cong \ZZ_{p^2}$ on $D(A)$ is necessarily given in this basis by 
\begin{equation}\label{OK action}
a \mapsto\text{diag}(a, \overline{a}, a, \overline{a}). 
\end{equation}
Furthermore, using the basis $\{e_n\}$ to view
$R_{11} \cong \End_{\Oo_B \otimes_{\ZZ} \DDD}(D(A)) \subset \MM_4(W)$,
we can express any
$$
f = \begin{bmatrix} x & y\Pi \\ py\Pi & x \end{bmatrix} \in R_{11}
$$
as an element of $\MM_4(W)$ by
\begin{equation}\label{endo matrix}
f = 
\begin{bmatrix}
x & 0 & 0 & py \\
0 & \overline{x} & \overline{y} & 0 \\
0 & p^2y & x & 0 \\
p\overline{y} & 0 & 0 & \overline{x}
\end{bmatrix}.
\end{equation}
Note that (\ref{matrix rep}) comes from choosing a basis $\{v_1, v_2\}$ of $D(E)$ 
satisfying $\FFF(v_1) = \VVV(v_1) = v_2$ and $\FFF(v_2) = \VVV(v_2) = pv_1$, so we have proved the following.

\begin{proposition}\label{basis}
With notation as above, there is a $W$-basis $\{e_1, e_2, e_3, e_4\}$ for $D(A)$ relative to which the action of $\Delta$ on $D(A)$ is given by one of the matrices {\upshape(\ref{OB action})}, the action of $\Oo_{\bm{k}, p}$ is given by {\upshape(\ref{OK action})}, the action of $\FFF$ is determined by
\begin{equation}\label{frob action}
\FFF(e_1) = e_2, \quad \FFF(e_2) = pe_1, \quad \FFF(e_3) = e_4, \quad \FFF(e_4) = pe_3,
\end{equation}
the action of $\VVV$ is the same, and any $f \in \End_{\Oo_B \otimes_{\ZZ}\DDD}(D(A))$ is given by a matrix of the form 
{\upshape(\ref{endo matrix})}.
\end{proposition}

\begin{corollary}\label{basis 2}
Let $k$ be an algebraically closed extension of $\overline{\FF}_\frakP$ and let $W(k)$ be its ring of Witt vectors. For any $A \in \YYY^B(k)$, there is a $W(k)$-basis $\{e_1, e_2, e_3, e_4\}$ of $D(A)$ relative to which the action of $\Delta$ on $D(A)$ is given by one of the matrices {\upshape(\ref{OB action})}, the action of $\Oo_{\bm{k}, p}$ is given by {\upshape(\ref{OK action})}, and the actions of $\FFF$ and
$\VVV$ are given by {\upshape(\ref{frob action})}.
\end{corollary}

\begin{proof}
As in the proof of Corollary \ref{st2}, there exists $A_0 \in \YYY^B(\overline{\FF}_\frakP)$ such that $A_0 \otimes_{\overline{\FF}_\frakP} k \cong A$ and $\End_{\overline{\FF}_\frakP}(A_0) \cong \End_k(A)$. Since the functor $D$ is compatible with any extension of perfect fields, $D(A) \cong D(A_0) \otimes_W W(k)$ as $\DDD_k$-modules, where $\DDD_k$ is the Dieudonn\'e ring of $k$ (\cite[Proposition 7.2.6]{BC2}), and hence $\End_{W(k)}(D(A)) \cong \End_W(D(A_0)) \otimes_W W(k)$. Therefore the $W(k)$-basis
$\{e_n \otimes 1\}$ for $D(A)$, with $\{e_n\}$ a $W$-basis of $D(A_0)$ as in Proposition \ref{basis}, satisfies the stated properties.
\end{proof}

Proposition \ref{endomorphisms} gives a description of $\End_{\Oo_B}(A) \otimes_{\ZZ} \ZZ_p$ in terms of coordinates, which is
best suited for computations. The next result gives the abstract structure of this ring.

\begin{proposition} 
There is an isomorphism of rings $R_{11} \cong R_2$, where
$$
R_2 = \begin{bmatrix} \ZZ_p & \ZZ_p \\ p^2\ZZ_p & \ZZ_p \end{bmatrix}
$$
is the standard Eichler order of level $2$ in $\MM_2(\QQ_p)$.
\end{proposition}

\begin{proof}
The proof is identical to a calculation carried out in \cite[pp.\hspace{-.5mm} 26-27]{Goren}. 
\end{proof}

\section{Moduli spaces}

We continue with the same notation of $K_1, K_2, F$, and $K$ as in Section 1.1.
Recall that we assume any prime dividing $d_B$ is inert
in $K_1$ and $K_2$. In particular, each $p \mid d_B$ is nonsplit in $K_1$ and $K_2$, which implies $K_1$ and $K_2$ embed into $B$, or equivalently, they split $B$.
If a prime number $p$ is inert in both $K_1$ and $K_2$, then $p$ is split in $F$ and each prime of $F$ lying over $p$
is inert in $K$. If $p$ is ramified in one of $K_1$ or $K_2$, then $p$ is ramified in $F$ and the unique prime
of $F$ lying over $p$ is inert in $K$.

\begin{definition}
A \textit{CM pair} over an $\Oo_K$-scheme $S$ is a pair $(\AAA_1, \AAA_2)$ where $\AAA_1$ and $\AAA_2$
are QM abelian surfaces over $S$ with complex multiplication by $\Oo_{K_1}$ and $\Oo_{K_2}$, respectively. An \textit{isomorphism}
between CM pairs $(\AAA_1', \AAA_2') \to (\AAA_1, \AAA_2)$ over $S$ is a pair $(f_1, f_2)$ where each $f_j : A_j' \to A_j$ is an 
$\Oo_{K_j}$-linear isomorphism of QM abelian surfaces.
\end{definition}

Given a CM pair $(\AAA_1, \AAA_2)$ over an $\Oo_K$-scheme $S$ and a morphism of $\Oo_K$-schemes $T \to S$, 
there is a CM pair $(\AAA_1, \AAA_2)_{/T}$ over $T$ defined as the base change to $T$.
For every CM pair $(\AAA_1, \AAA_2)$ over an $\Oo_K$-scheme $S$, set 
$$
L(\AAA_1, \AAA_2) = \Hom_{\Oo_B}(A_1, A_2), \quad V(\AAA_1, \AAA_2) = L(\AAA_1, \AAA_2) \otimes_{\ZZ}\QQ.
$$
If $S$ is connected we have the quadratic form $\deg^\ast$
on $L(\AAA_1, \AAA_2)$. Let $[f, g] = f^t \circ g + g^t \circ f$ be the associated bilinear form. Then
$\Oo_K = \Oo_{K_1} \otimes_{\ZZ} \Oo_{K_2}$ acts on the $\ZZ$-module $L(\AAA_1, \AAA_2)$ by
$$
(x_1 \otimes x_2) \bullet f = \kappa_2(x_2) \circ f \circ \kappa_1(\overline{x}_1),
$$
where $\mathbf{A}_j = (A_j, i_j, \kappa_j)$.

\begin{proposition}\label{Hermitian}
Let $(\AAA_1, \AAA_2)$ be a CM pair over a connected $\Oo_K$-scheme.
There is a unique $F$-bilinear form $[\cdot\hspace{.5mm}, \cdot]_{\CM}$ on $V(\AAA_1, \AAA_2)$ satisfying
$[f, g] = \Tr_{F/\QQ}([f, g]_{\CM})$. Under this pairing,
$$[L(\AAA_1, \AAA_2), L(\AAA_1, \AAA_2)]_{\CM} \subset \frakD^{-1}.$$
The quadratic form $\deg_{\CM}(f) = \frac{1}{2}[f, f]_{\CM}$ is the unique $F$-quadratic form
on $V(\AAA_1, \AAA_2)$ satisfying $\deg^\ast(f) = \Tr_{F/\QQ}(\deg_{\CM}(f))$. Also, there is a unique $K$-Hermitian form $\langle \cdot\hspace{.5mm}, \cdot \rangle_{\CM}$ on $V(\AAA_1, \AAA_2)$
satisfying $[f, g]_{\CM} = \Tr_{K/F}(\langle f, g\rangle_{\CM})$.
\end{proposition}

\begin{proof}
This is the same as the proof of \cite[Proposition 2.2]{Howard}.
\end{proof}

\begin{definition}
For $j \in \{1, 2\}$ define $\YYY_j^B$ to be the stack $\YYY^B$ with $\bm{k} = K_j$ and $\bm{K} = K$.
For any ring homomorphism $\theta_j : \Oo_{K_j} \to \Oo_B/\frakm_B$, define $\YYY_j^B(\theta_j)$
to be the stack $\YYY^B(\theta_j)$ with $\bm{k} = K_j$ and $\bm{K} = K$.
\end{definition}

From now on, we write $\YYY^B$ to mean the category defined in Definition \ref{CM FELC} for
some fixed imaginary quadratic field $\bm{k}$ and finite extension $\bm{K}$.

\begin{definition}
Let $\Theta : \Oo_K \to \Oo_B/\frakm_B$ be a ring homomorphism. Define $\XXX_\Theta^B$ to be the category
whose objects are CM pairs $(\AAA_1, \AAA_2)$ over $\Oo_K$-schemes such that $\AAA_j$ is an object
of $\YYY_j^B(\theta_j)$ for $j = 1, 2$, where $\theta_j = \Theta|_{\Oo_{K_j}}$.
A morphism $(\AAA_1', \AAA_2') \to
(\AAA_1, \AAA_2)$ between two such pairs defined over $\Oo_K$-schemes $T$ and $S$, respectively, is a morphism of 
$\Oo_K$-schemes $T \to S$ together with an isomorphism of CM pairs 
$(\AAA_1', \AAA_2') \cong (\AAA_1, \AAA_2)_{/T}$ over $T$.
\end{definition}

\begin{definition}
Let $\Theta : \Oo_K \to \Oo_B/\frakm_B$ be a ring homomorphism. For any $\alpha \in F^\times$ define 
$\XXX_{\Theta, \alpha}^B$ to be the category whose objects are triples $(\AAA_1, \AAA_2, f)$ where 
$(\AAA_1, \AAA_2) \in \XXX_\Theta^B(S)$ for some $\Oo_K$-scheme $S$ and $f \in L(\AAA_1, \AAA_2)$
satisfies $\deg_{\CM}(f) = \alpha$ on every connected component of $S$. A morphism
$$
(\AAA_1', \AAA_2', f') \to (\AAA_1, \AAA_2, f)
$$
between two such triples, with $(\AAA_1', \AAA_2')$ and $(\AAA_1, \AAA_2)$ CM pairs over $\Oo_K$-schemes $T$ and $S$,
respectively, is a morphism of $\Oo_K$-schemes $T \to S$ together with an isomorphism 
$$
(\AAA_1', \AAA_2') \to (\AAA_1, \AAA_2)_{/T}
$$
of CM pairs over $T$ compatible with $f$ and $f'$.
\end{definition}
 
The categories $\XXX_\Theta^B$ and $\XXX_{\Theta, \alpha}^B$ are stacks of finite type over $\Spec(\Oo_K)$. 
For each positive integer $m$ define $\TTT_m^B$ to be the stack over $\Spec(\Oo_K)$ with $\TTT_m^B(S)$ the
category of triples $(\AAA_1, \AAA_2, f)$ where $\AAA_j \in \YYY_j^B(S)$ and $f \in L(\AAA_1, \AAA_2)$
satisfies $\deg^\ast(f) = m$ on every connected component of $S$.
It follows from Theorem \ref{Serre type} that there is a decomposition
\begin{equation}\label{TM decomp}
\TTT_m^B = \bigsqcup_{\substack{\alpha \in F^\times \\ \Tr_{F/\QQ}(\alpha) = m}}
\bigsqcup_{\Theta : \Oo_K \to \Oo_B/\frakm_B} \XXX_{\Theta, \alpha}^B.
\end{equation}

\begin{lemma}\label{tp}
If $S$ is a connected $\Oo_K$-scheme and $\XXX_{\Theta, \alpha}^B(S)$ is nonempty, then $\alpha$ is totally positive.
\end{lemma}

\begin{proof}
Let $(\AAA_1, \AAA_2, f) \in \XXX_{\Theta, \alpha}^B(S)$.
Suppose $\alpha$ is not totally positive and let $\tau_1, \tau_2$ be the embeddings of $F$ into $\RR$. Then we can choose $\gamma_1, \gamma_2 \in F$ such that 
$$
\tau_1(\gamma_1)^2\tau_1(\alpha) + \tau_2(\gamma_2)^2\tau_2(\alpha) < 0.
$$
By the weak approximation theorem, there is a $\gamma \in F$ such that $\tau_j(\gamma) < 2\tau_j(\gamma_j)$ for $j = 1, 2$. Hence
$$
\deg^\ast(\gamma f) = \Tr_{F/\QQ}(\deg_{\CM}(\gamma f)) = \Tr_{F/\QQ}(\gamma^2\deg_{\CM}(f)) = \sum_{j=1}^2\tau_j(\gamma)^2\tau_j(\alpha) < 4\sum_{j=1}^2\tau_j(\gamma_j)^2\tau_j(\alpha) < 0,
$$
which contradicts Proposition \ref{QM degree qf}, where we are extending $\deg^\ast$ to $V(\AAA_1, \AAA_2)$ in the obvious way.
\end{proof}

A CM pair $(\AAA_1, \AAA_2)$ is \textit{supersingular} if the underlying 
abelian varieties $A_1$ and $A_2$ are supersingular.

\begin{proposition} \label{quadratic}
Let $k$ be an algebraically closed field of characteristic $p \gqq 0$ and let $\Theta : \Oo_K \to \Oo_B/\frakm_B$
be a ring homomorphism. Let $\alpha \in F^\times$
and suppose $(\AAA_1, \AAA_2, f) \in \XXX_{\Theta, \alpha}^B(k)$. \\
{\upshape (a)} We have $p > 0$ and if $k = \overline{\FF}_p$, then $(\AAA_1, \AAA_2)$ is a supersingular CM pair. \\
{\upshape (b)} There is an isomorphism of $F$-quadratic spaces
$$
(V(\AAA_1, \AAA_2), \deg_{\CM}) \cong (K, \beta\cdot \NN_{K/F})
$$
for some totally positive $\beta \in F^\times$, determined up to multiplication by a norm from $K^\times$. \\
{\upshape (c)} There is an isomorphism of $\QQ$-quadratic spaces
$$
(V(\AAA_1, \AAA_2), \deg^\ast) \cong (B^{(p)}, \Nrd),
$$
where $\Nrd$ is the reduced norm on $B^{(p)}$. \\
{\upshape (d)} If $p$ does not divide $d_B$ then it is nonsplit in $K_1$ and $K_2$. 
\end{proposition}

\begin{proof}
(a) Suppose $p = 0$. By the proof of Corollary \ref{st2}, there is an $A^0_j \in \YYY^B(\overline{\QQ})$
such that $A^0_j \otimes_{\overline{\QQ}} k \cong A_j$ and $\End_k(A_j) \cong \End_{\overline{\QQ}}(A^0_j) \cong \End_{\CC}(A^0_j \otimes_{\overline{\QQ}} \CC)$.
Since $\deg^\ast(f) \neq 0$ by assumption,
$f : A_1 \to A_2$ is an isogeny, so it induces an isomorphism $\End^0_{\Oo_B}(A_1) \cong \End^0_{\Oo_B}(A_2)$
of $\QQ$-algebras. Also by assumption we have embeddings $\kappa_1 : K_1 \hookrightarrow \End^0_{\Oo_B}(A_1)$
and $\kappa_2 : K_2 \hookrightarrow \End^0_{\Oo_B}(A_2)$. By Proposition \ref{char p}, there are
two possibilities for $\End^0_{\Oo_B}(A_j) \cong \End^0_{\Oo_B}(A^0_j \otimes_{\overline{\QQ}}\CC)$.
The first case, where $\End^0_{\Oo_B}(A_1) \cong \End^0_{\Oo_B}(A_2) \cong \QQ$, is impossible. The second case is $\End^0_{\Oo_B}(A_1) \cong \End^0_{\Oo_B}(A_2)
\cong L$ for some imaginary quadratic field $L$. But then $\kappa_1$ and $\kappa_2$ induce isomorphisms
$K_1 \cong L \cong K_2$, contrary to our assumption about $K_1$ and $K_2$. Therefore $p > 0$. 
By Corollary \ref{endo ring},  $\End^0_{\Oo_B}(A_1)$ is
an imaginary quadratic field or $B^{(p)}$. It follows that $\End^0_{\Oo_B}(A_1) \cong \End^0_{\Oo_B}(A_2) \cong B^{(p)}$ because
otherwise $K_1 \cong K_2$ as above.

(b) Since $f : A_1 \to A_2$ is an isogeny, it induces an isomorphism $V(\AAA_1, \AAA_2) \to \End^0_{\Oo_B}(A_1)$
of $\QQ$-vector spaces defined by $\varphi \mapsto f^t \circ \varphi$. As $\End^0_{\Oo_B}(A_1) \cong B^{(p)}$
has dimension $4$ as a $\QQ$-vector space, $V(\AAA_1, \AAA_2)$ has dimension $1$ over $K$. Therefore
$(V(\AAA_1, \AAA_2), \langle\cdot\hspace{.5mm}, \cdot\rangle_{\CM})$ is a Hermitian $K$-module of dimension $1$.
This means that there is a $\gamma \in K^\times$ such that $\langle v, w\rangle_{\CM} = v\gamma\overline{w}$ for
all $v, w \in K$, so
$$
\deg_{\CM}(v) = \frac{1}{2}[v, v]_{\CM} = \frac{1}{2}\Tr_{K/F}(\langle v, v\rangle_{\CM}) = \frac{1}{2}\Tr_{K/F}(v\overline{v}
\gamma) = \beta\cdot\NN_{K/F}(v),
$$
where $\beta = \frac{1}{2}\Tr_{K/F}(\gamma) \in F^\times$. This proves the existence of the isomorphism of $F$-quadratic spaces.

Now suppose $\gamma' \in K^\times$ is another element satisfying $\langle v, w\rangle_{\CM} = v\gamma'\overline{w}$
for all $v, w \in K$. Then $\gamma = u\gamma'\overline{u} = \gamma'\cdot\NN_{K/F}(u)$ for some $u \in K^\times$,
so 
$$
\beta = \frac{1}{2}\Tr_{K/F}(\gamma) = \frac{1}{2}\NN_{K/F}(u)\Tr_{K/F}(\gamma') = \beta'\cdot\NN_{K/F}(u),
$$
where $\beta'$ is the element of $F^\times$ corresponding to $\gamma'$. 

To show $\beta$ is totally positive, note that if $g \in V(\AAA_1, \AAA_2)$ corresponds to $x \in K$ under the isomorphism $V(\AAA_1, \AAA_2) \cong K$, then $\deg_{\CM}(g) = \beta\NN_{K/F}(x)$, so $\deg^\ast(g) = \Tr_{F/\QQ}(\beta\NN_{K/F}(x))$. Since this element is positive for all $x \in K$, the elements $\Tr_{F/\QQ}(\beta)$ and $\NN_{F/\QQ}(\beta)$ are positive, which implies $\beta$ is totally positive.

(c) Under the isomorphism $V(\AAA_1, \AAA_2) \to \End^0_{\Oo_B}(A_1)$ defined above, the quadratic form $\deg^\ast$
on $V(\AAA_1, \AAA_2)$ corresponds to the quadratic form $d^{-1}\deg^\ast$ on $\End^0_{\Oo_B}(A_1)$, 
where $d = \deg^{\ast}(f)$.
We claim that under the isomorphism $\End^0_{\Oo_B}(A_1) \to B^{(p)}$, the quadratic form $\deg^\ast$ on 
$\End^0_{\Oo_B}(A_1)$
corresponds to the quadratic form $\Nrd$ on $B^{(p)}$. The Rosati involution $\varphi \mapsto \varphi^{\dagger} = 
\varphi^t$ on $\End^0_{\Oo_B}(A_1)$ corresponds to a positive involution on the definite quaternion algebra
$B^{(p)}$, which must be the main involution $x \mapsto x^{\iota}$ by \cite[Theorem 2 in \S 21]{Mumford}. Since
$\Nrd(x) = xx^{\iota}$ and $\deg^\ast(\varphi) = \varphi \circ \varphi^t$, this proves the claim. Therefore there
is an isomorphism of $\QQ$-quadratic spaces
$$
(V(\AAA_1, \AAA_2), \deg^\ast) \cong (B^{(p)}, d^{-1}\Nrd).
$$
However, since $d > 0$, it is in the image of $\Nrd$, so there is an isomorphism of $\QQ$-quadratic spaces
$$
(B^{(p)}, d^{-1}\Nrd) \cong (B^{(p)}, \Nrd).
$$

(d) Suppose $p \nmid d_B$, so $p$ ramifies in $B^{(p)}$.
If $p$ splits in $K_j$ then, since $K_j$ embeds in $B^{(p)}$,
we have $B^{(p)} \otimes_{\QQ} \QQ_p \supset K_j \otimes_{\QQ} \QQ_p \cong \QQ_p \times \QQ_p$.
This is impossible because $B^{(p)} \otimes_{\QQ} \QQ_p$ is a division algebra.
\end{proof}

For any $\Oo_K$-scheme $S$ and ring homomorphism $\Theta : \Oo_K \to \Oo_B/\frakm_B$,
the group $\Gamma = \Cl(\Oo_{K_1}) \times \Cl(\Oo_{K_2})$ acts on the set $[\XXX_\Theta^B(S)]$ of isomorphism classes of objects of $\XXX^B_\Theta(S)$ by
$$
(\fraka_1, \fraka_2)\cdot(\AAA_1, \AAA_2) = (\fraka_1 \otimes_{\Oo_{K_1}} A_1, \fraka_2 \otimes_{\Oo_{K_2}} A_2).
$$

\begin{lemma}\label{automorphism group}
Let $S$ be a connected $\Oo_K$-scheme and for $j \in\{1, 2\}$ set $w_j = |\Oo_{K_j}^\times|$. Every $x \in \XXX_\Theta^B(S)$,
viewed as an element of the set $[\XXX_\Theta^B(S)]$,
has trivial stabilizer in $\Gamma$ and satisfies $|\Aut_{\XXX_\Theta^B(S)}(x)| = w_1w_2$.
\end{lemma}

\begin{proof}
The first claim is immediate from Lemma \ref{free action}.
Next, by definition, an automorphism of $x$ in $\XXX_\Theta^B(S)$ is a pair $(a_1, a_2$) with 
$a_j \in \Aut_{\Oo_j}(A_j) \cong \Oo^\times_{K_j}$, where $\Oo_j = \Oo_B \otimes_{\ZZ} \Oo_{K_j}$.
\end{proof}

\section{Local quadratic spaces}

This section and the next form the technical core of this paper. In this section we (essentially) count
the number of geometric points of $\XXX_{\Theta, \alpha}^B$. This comes from a careful examination
of the quadratic spaces $(V_\ell(\AAA_1, \AAA_2), \deg_{\CM})$ for each prime $\ell$, where
$$
L_{\ell}(\AAA_1, \AAA_2) = L(\AAA_1, \AAA_2) \otimes_{\ZZ} \ZZ_{\ell}, \quad
V_{\ell}(\AAA_1, \AAA_2) = V(\AAA_1, \AAA_2) \otimes_{\QQ} \QQ_{\ell}.
$$ 
The methods of the
proofs follow \cite{Howard} quite closely.
Suppose $\ell$ is a prime dividing $d_B$, let $k$ be an algebraically closed field over $\Oo_{\bm{K}}$, possibly of characteristic $\ell$, and let $A \in \YYY^B(k)$.
Define the $\frakm_{\ell}$-torsion of $A$ as 
$$
A[\frakm_{\ell}] = \ker(i(x_\ell) : A[\ell] \to A[\ell]),
$$
where $x_\ell$ is any element of $\frakm_{\ell}$ whose image generates the principal ideal $\frakm_{\ell}/\ell\Oo_B \subset \Oo_B/\ell\Oo_B$. This is a finite flat commutative group scheme over $\Spec(k)$ of order $\ell^2$.

\begin{lemma}\label{ml torsion}
Suppose $A \in \YYY^B(k)$ with $k$ an algebraically closed field over $\Oo_{\bm{K}}$ and $\ell \neq \charr(k)$ is a prime dividing $d_B$. There is an isomorphism of $\Oo_B/\frakm_\ell$-algebras 
$\End_{\Oo_B/\frakm_{\ell}}(A[\frakm_{\ell}]) \cong \Oo_B/\frakm_\ell$.
\end{lemma}

\begin{proof}
Since $\ell$ is invertible in $k$, the group scheme $A[\ell]$ is finite \'etale over $k$,
so $A[\frakm_\ell]$ is finite \'etale over $k$ and thus constant. It follows that the natural map
$$
\End_{\Oo_B/\frakm_{\ell}}(A[\frakm_{\ell}]) \to \End_{\Oo_B/\frakm_{\ell}}(A[\frakm_{\ell}](k))
$$
is an isomorphism. The group $A[\frakm_\ell](k)$ is a vector space of dimension $1$ over $\Oo_B/\frakm_\ell$ since $|A[\frakm_\ell](k)| = \ell^2 = |\Oo_B/\frakm_\ell|$,
which proves the result.
\end{proof}

\subsection{The case of $\ell \neq p$}

Fix a prime ideal $\frakP \subset \Oo_K$ of residue characteristic $p$, where $p$ is
nonsplit in $K_1$ and $K_2$, a ring homomorphism $\Theta : \Oo_K \to \Oo_B/\frakm_B$, and a CM pair 
$(\AAA_1, \AAA_2) \in \XXX_\Theta^B(\overline{\FF}_{\frakP})$ (necessarily supersingular by Lemma \ref{supersingular2}). Recall that $\frakD$ is the different ideal of $F/\QQ$.

\begin{proposition}\label{quadratic I}
Let $\ell \neq p$ be a prime. There is a $K_{\ell}$-linear isomorphism of $F_{\ell}$-quadratic spaces
$$
(V_{\ell}(\AAA_1, \AAA_2), \deg_{\CM}) \cong (K_{\ell}, \beta_\ell\cdot \NN_{K_{\ell}/F_{\ell}})
$$
for some $\beta_\ell \in F_{\ell}^\times$ satisfying $\beta_\ell\Oo_{F,\ell} = \frakD_{\ell}^{-1} = \frakD^{-1}\Oo_{F, \ell}$
if $\ell \nmid d_B$ and $\beta_\ell\Oo_{F,\ell} = \frakl\frakD_{\ell}^{-1}$ if $\ell \mid d_B$, where $\frakl$ is
the prime over $\ell$ dividing $\ker(\Theta) \cap \Oo_F$. This map takes $L_{\ell}(\AAA_1, \AAA_2)$ 
isomorphically to $\Oo_{K, \ell}$.
\end{proposition}

\begin{proof}
We will write $L_{\ell}$ and $V_{\ell}$ for $L_{\ell}(\AAA_1, \AAA_2)$ and $V_{\ell}(\AAA_1, \AAA_2)$.
The existence of an isomorphism of quadratic spaces for some $\beta_\ell \in F_{\ell}^\times$ follows from Proposition \ref{quadratic}(b).
Under this isomorphism, $L_{\ell}$ is sent to a finitely generated $\Oo_{K, \ell}$-submodule
of $K_{\ell}$, that is, a fractional $\Oo_{K, \ell}$-ideal. Then since every ideal of $\Oo_{K, \ell}$ is principal, 
there is an isomorphism $V_\ell \cong K_\ell$ inducing an isomorphism $L_{\ell} \cong \Oo_{K, \ell}$. The $\Oo_{F, \ell}$-bilinear form
$$
[\cdot\hspace{.5mm}, \cdot]_{\CM} : L_{\ell} \times L_{\ell} \to \frakD_{\ell}^{-1} 
$$
induces an $\Oo_{F, \ell}$-bilinear form
$
\Oo_{K, \ell} \times \Oo_{K, \ell} \to \frakD_{\ell}^{-1}
$
given by $(x, y) \mapsto \beta_\ell\Tr_{K_{\ell}/F_{\ell}}(x\overline{y})$.
The dual lattice of $\Oo_{K, \ell}$ with respect to this pairing satisfies
\begin{align*}
\beta_\ell\Oo_{K,\ell}^{\vee} &= \{x \in K_{\ell} : \text{$\Tr_{K_{\ell}/F_{\ell}}(xy) \in \frakD_{\ell}^{-1}$
for all $y \in \Oo_{K, \ell}$}\} \\
&= \{x \in K_{\ell} : \text{$\Tr_{K_{\ell}/\QQ_{\ell}}(xy) \in \ZZ_{\ell}$ for all $y \in \Oo_{K, \ell}$}\} \\
&= \frakD^{-1}_{K/\QQ}\Oo_{K, \ell}.
\end{align*}
Since $K/F$ is unramified at any prime of $F$ over $\ell$, $\frakD_{K/\QQ}\Oo_{K, \ell} = \frakD\Oo_{K, \ell}$, which shows $L_{\ell}^\vee \cong \Oo_{K,\ell}^\vee
= \beta_\ell^{-1}\frakD^{-1}\Oo_{K,\ell}$.

First suppose $\ell \nmid d_B$. By Lemmas \ref{Tate modules} and \ref{Tate
modules 2} there are isomorphisms of $\ZZ_{\ell}$-modules
$$
L_{\ell} \cong \Hom_{\Oo_{B, \ell}}(T_\ell(A_1), T_\ell(A_2)) \cong \MM_2(\ZZ_\ell).
$$
Under this isomorphism the quadratic form $\deg^\ast$ on $L_{\ell}$ is identified with the quadratic
form $u\cdot\det$ on $\MM_2(\ZZ_{\ell})$ for some $u \in \ZZ_{\ell}^\times$. 
The lattice $\MM_2(\ZZ_{\ell}) \subset \MM_2(\QQ_{\ell})$ is self dual relative to $\det$, and the dual of $L_\ell$ with respect to $\deg_{\CM}$ is equal to the dual of $L_\ell$ with respect to $\deg^\ast$,
so from the isomorphism
$$
L_{\ell}^\vee/L_{\ell} \cong \beta_\ell^{-1}\frakD^{-1}\Oo_{K, \ell}/\Oo_{K,\ell},
$$
we find $\beta_\ell\Oo_{K, \ell} = \frakD^{-1}\Oo_{K, \ell}$, and thus $\beta_\ell\Oo_{F, \ell} = \frakD_{\ell}^{-1}$
as $K/F$ is unramified over $\ell$.

Now suppose $\ell \mid d_B$. In the proof of Lemma \ref{Tate modules}(b) we showed that
$T_{\ell}(A_j) \cong \Oo_{B, \ell}$ as $\Oo_{B, \ell}$-modules, so $T_{\ell}(A_1) \cong T_{\ell}(A_2)$
as $\Oo_{B, \ell}$-modules and thus by Lemma \ref{Tate modules 2} there are isomorphisms of $\ZZ_{\ell}$-modules
$$
\Hom_{\Oo_B}(A_1, A_2) \otimes_{\ZZ} \ZZ_{\ell}  \cong \End_{\Oo_{B, \ell}}(T_\ell(A_1)) \cong \End_{\Oo_B}(A_1) \otimes_{\ZZ} \ZZ_{\ell}.
$$
Therefore we may reduce to the case where $\AAA_1$ and
$\AAA_2$ have the same underlying QM abelian surface $A$. By Lemmas \ref{Tate modules} and \ref{Tate
modules 2} there are isomorphisms of $\ZZ_{\ell}$-algebras
$$
L_{\ell} \cong \End_{\Oo_{B, \ell}}(T_{\ell}(A)) \cong \Oo_{B, \ell},
$$
and by Proposition \ref{quadratic}(c), this isomorphism identifies the 
quadratic form $\deg^\ast$ on $L_{\ell}$ with the quadratic form $\Nrd$ on $\Oo_{B, \ell}$. If
$\frakn_{\ell} \subset \Oo_{B, \ell}$ is the unique maximal ideal, so $\frakm_\ell\Oo_{B, \ell} = \frakn_\ell$, 
then a calculation shows that the dual
lattice of $\Oo_{B, \ell}$ relative to $\Nrd$ is $\frakn_{\ell}^{-1}$. Hence we have $\Oo_{K, \ell}$-linear
isomorphisms
$$
\beta_\ell^{-1}\frakD^{-1}\Oo_{K, \ell}/\Oo_{K, \ell} \cong L_{\ell}^\vee/L_{\ell} \cong \frakn_{\ell}^{-1}/\Oo_{B, \ell}.
$$
As a group, $\frakn_{\ell}^{-1}/\Oo_{B, \ell} \cong \Oo_{B, \ell}/\frakn_{\ell} \cong \FF_{\ell^2}$, so 
$[\Oo_{K, \ell} : \beta_\ell\frakD\Oo_{K, \ell}] = \ell^2$.

Recall that $\Oo_K$ acts on $L_{\ell} \cong \Oo_{B, \ell}$ by
$$
(t_1 \otimes t_2) \bullet f = \kappa_2(t_2) \circ f \circ \kappa_1(\overline{t}_1).
$$
Fixing a uniformizer $\Pi \in \Oo_{B, \ell}$ satisfying $\kappa_1(\overline{t}_1)\Pi = \Pi\kappa_1(t_1)$
for all $t_1 \in \Oo_{K_1}$, for any $u \in \kappa_1(\Oo_{K_1}) \subset \Oo_{B, \ell}$ we have
$$
(t_1 \otimes t_2) \bullet u\Pi^{-1} = \kappa_2(t_2)u\Pi^{-1}\kappa_1(\overline{t}_1) = \kappa_2(t_2)\kappa_1(t_1)
u\Pi^{-1}.
$$
Since $\frakn_{\ell}^{-1}/\Oo_{B, \ell}$ is generated by such elements $u\Pi^{-1}$, $\Oo_K$ acts on 
$\frakn_{\ell}^{-1}/\Oo_{B, \ell}$ through left multiplication by the image of the composition 
$\Oo_K \to \Oo_{B, \ell} \to \Oo_{B, \ell}/\frakn_{\ell} \cong \frakn_\ell^{-1}/\Oo_{B, \ell}$, where the first map is given by 
$t_1 \otimes t_2 \mapsto \kappa_2(t_2)\kappa_1(t_1)$. 
Next, under the isomorphism $L_{\ell} \cong \Oo_{B, \ell}$, the action 
$$
\Oo_{B, \ell} \to \End_{\Oo_B/\frakm_\ell}(A[\frakm_\ell]) \cong \Oo_B/\frakm_\ell
$$
determines an isomorphism $\gamma : \Oo_{B, \ell}/\frakn_{\ell} \to \Oo_B/\frakm_\ell$, which allows us to
identify 
$$
\kappa_j^{\frakm_\ell} : \Oo_{K_j} \to \End_{\Oo_B/\frakm_\ell}(A[\frakm_\ell])
$$
with the composition
$$
\Oo_{K_j} \mapp{\kappa_j} \Oo_{B, \ell} \to \Oo_{B, \ell}/\frakn_{\ell} \mapp{\gamma} \Oo_B/\frakm_\ell.
$$
However, the map $\Oo_K \to \FF_{\ell^2}$ defined by $t_1 \otimes t_2 \mapsto \kappa_1^{\frakm_\ell}(t_1)\kappa_2^{\frakm_\ell}(t_2)$
is equal to the map 
$$
\Oo_K \mapp{\Theta} \Oo_B/\frakm_B \to \Oo_B/\frakm_\ell,
$$
by definition of $(\AAA_1, \AAA_2)$ being in $\XXX^B_\Theta(\overline{\FF}_\frakP)$,
and the kernel of this map is the prime $\frakL$ of $K$ over $\frakl$. It then follows from the factorization of $\kappa_j^{\frakm_\ell}$ above that any element of $\frakL$ acts trivially on $\frakn_{\ell}^{-1}/\Oo_{B, \ell}$ and therefore
there is an $\Oo_{K, \ell}$-linear map $\Oo_{K, \ell}/\frakL\Oo_{K, \ell} \hookrightarrow \frakn_{\ell}^{-1}/\Oo_{B, \ell}$
given by $t \mapsto t \bullet \Pi^{-1}$. But $\frakL$ has norm $\ell^2$, which means
$$
\Oo_{K, \ell}/\frakL\Oo_{K, \ell} \cong \frakn_{\ell}^{-1}/\Oo_{B, \ell} \cong \beta_\ell^{-1}\frakD^{-1}\Oo_{K, \ell}/\Oo_{K, \ell}
$$
as $\Oo_{K, \ell}$-modules. This shows $\beta_\ell\frakD\Oo_{K, \ell} = \frakL\Oo_{K, \ell}$ and thus
$\beta_\ell\Oo_{F, \ell} = \frakl\frakD_{\ell}^{-1}$.
\end{proof}
 
\subsection{The case of $\ell = p$}
In order to prove a similar result for $\ell = p$ we need a few preliminary results.

\begin{lemma}\label{Lie algebra}
If $A \in \YYY^B(\overline{\FF}_p)$ then $\End_{\Oo_B \otimes_{\ZZ} \overline{\FF}_p}(\Lie(A)) \cong \overline{\FF}_p$
as $\overline{\FF}_p$-algebras.
\end{lemma}

\begin{proof}
For $p \nmid d_B$, we have $\Oo_B \otimes_{\ZZ} \overline{\FF}_p \cong \MM_2(\overline{\FF}_p)$, so $\End_{\Oo_B \otimes_{\ZZ} \overline{\FF}_p}(\Lie(A)) \cong 
\End_{\MM_2(\overline{\FF}_p)}(\overline{\FF}_p^2) \cong \overline{\FF}_p$. Now suppose $p \mid d_B$
and fix an isomorphism $\Oo_{B, p} \cong \Delta$ and a uniformizer $\Pi \in \Delta$. There are isomorphisms of $\overline{\FF}_p$-vector spaces 
$$
\Lie(A) \cong \Lie(D(A)) \cong D(A)/\VVV D(A).
$$
Let $\{e_n\}$ be a $W$-basis for $D(A)$ as in
Proposition \ref{basis}, so the images $\widetilde{e}_1, \widetilde{e}_3$ of $e_1, e_3$ in $D(A)/\VVV D(A)$
form a basis for this $2$-dimensional vector space over $W/pW \cong \overline{\FF}_p$. 
In the notation of (\ref{OB action}), if $D(i) = f_1$ then the action of $\Delta$ on $D(A)/\VVV D(A)$ is given, in the basis $\{\widetilde{e}_1, \widetilde{e}_3\}$, by the matrix
$$
D(i)(a + b\Pi) = \begin{bmatrix} \widetilde{a} & \widetilde{b} \\ 0 & \widetilde{\overline{a}} \end{bmatrix} \in \MM_2(\overline{\FF}_p),
$$
where $\widetilde{a}$ is the image of $a$ in $W/pW \cong \overline{\FF}_p$. A computation shows that a matrix in $\MM_2(\overline{\FF}_p)$
commutes with $D(i)(a + b\Pi)$ for all $a + b\Pi \in \Delta$ if and only if it is a scalar matrix, and therefore
$$
\End_{\Delta\otimes_{\ZZ_p}\overline{\FF}_p}(\Lie(A)) \cong \End_{\Delta\otimes_{\ZZ_p}\overline{\FF}_p}(D(A)/\VVV D(A)) \cong \overline{\FF}_p.
$$
A similar computation gives the same result if $D(i) = f_2$.
\end{proof}

\begin{proposition}\label{quadratic form}
Suppose $(A, i) \in \YYY^B(\overline{\FF}_p)$ with $p \mid d_B$. Under the isomorphism 
$$\End_{\Oo_B}(A) \otimes_{\ZZ} \ZZ_p \to R_{11}$$
in Proposition {\upshape \ref{endomorphisms}},
the $\ZZ_p$-quadratic form $\deg^\ast$ on $\End_{\Oo_B}(A) \otimes_{\ZZ} \ZZ_p$ is identified with the $\ZZ_p$-quadratic form
$Q$ on $R_{11}$ given by
$$
Q\begin{bmatrix} x & y\Pi \\ py\Pi & x \end{bmatrix} = x\overline{x} - p^2y\overline{y}.
$$
\end{proposition}

\begin{proof}
Recall that $f^t = \lambda^{-1} \circ f^\vee \circ \lambda$, where $\lambda : A \to A^\vee$
is the unique principal polarization satisfying $\lambda^{-1} \circ i(x)^\vee \circ \lambda = i(x^\ast)$ for all $x \in \Oo_B$.
The polarization $\lambda$ induces a map 
$$
\Lambda = D(\lambda) : D(A) \to D(A^\vee) \cong D(A)^\vee,
$$
where $D(A)^\vee$ is a $\DDD$-module via
$$
(\FFF \cdot f)(x) = \sigma(f(\VVV x)), \quad (\VVV \cdot f)(x) = \sigma^{-1}(f(\FFF x)).
$$
Then $\Lambda$ determines 
a nondegenerate, alternating, bilinear pairing $\langle \cdot\hspace{.5mm}, \cdot
\rangle : D(A) \times D(A) \to W$ satisfying
$\langle \FFF x, y \rangle = \sigma(\langle x, \VVV y\rangle)$ for all $x, y \in D(A)$. 

Recall that $x^\ast = a^{-1}x^\iota a$, where $a \in \Oo_B$ is an element satisfying $a^2 = -d_B$. In the local case
of $a \in \Oo_B \otimes_\ZZ \ZZ_p \cong \Delta$, we can explicitly choose $a$: since 
the norm map $\NN_{\QQ_{p^2}/\QQ_p} : \ZZ_{p^2}^\times \to \ZZ_p^\times$ is surjective, there is an $a_0 \in \ZZ_{p^2}^\times$
such that $a_0\overline{a}_0 = -p^{-1}d_B$. Let $a = a_0\Pi \in \Delta$, so
$a^2 = pa_0\overline{a}_0 = -d_B$. 

Let $\{e_n\}$ be a $W$-basis for $D(A)$ as in Proposition \ref{basis}.
First suppose $D(i) = f_1$, in the notation of
(\ref{OB action}). Then for $x = u + v\Pi \in \Delta$, one can compute, using $x^\iota = \overline{u} - v\Pi$,
$$
D(i(x^\ast)) = D(i(a))^{-1}D(i(x^\iota))D(i(a)) = \begin{bmatrix} u & 0 & -a_0\overline{a}_0^{-1}\overline{v} & 0 \\
0 & \overline{u} & 0 & -a_0^{-1}\overline{a}_0v \\ -pa_0^{-1}\overline{a}_0v & 0 & \overline{u} & 0 \\
0 & -pa_0\overline{a}_0^{-1}\overline{v} & 0 & u \end{bmatrix}.
$$
Viewing $\Lambda$ as an element of $\Hom_W(D(A), D(A)^\vee) \cong \MM_4(W)$, using that
$\Lambda$ is invertible and alternating, and that $\Lambda \circ D(i(x^\ast)) = D(i(x))^\vee\circ \Lambda$
for all $x \in \Delta$, where $D(i(x))^\vee$ is the dual linear map, a computation shows $\Lambda$ must be of the form
$$
\Lambda = \begin{bmatrix}
0 & 0 & 0 & ba_0^{-1}\overline{a}_0 \\
0 & 0 & b & 0 \\
0 & -b & 0 & 0 \\
-ba_0^{-1}\overline{a}_0 & 0 & 0 & 0
\end{bmatrix}
$$
for some $b \in W^\times$. The equality $\langle\FFF e_1, e_3\rangle = \sigma\langle e_1, \VVV e_3\rangle$
implies $b = \sigma(b)a_0\overline{a}_0^{-1}$, so $b \in \ZZ_{p^2}^\times$ and
$$
\Lambda = \begin{bmatrix}
0 & 0 & 0 & \overline{b}\\
0 & 0 & b & 0 \\
0 & -b & 0 & 0 \\
-\overline{b} & 0 & 0 & 0
\end{bmatrix}.
$$

The involution $\varphi \mapsto \varphi^\dagger$ on $\End_W(D(A)) \cong \MM_4(W)$ corresponding
to the Rosati involution $f \mapsto \lambda^{-1} \circ f^\vee \circ \lambda$ on $\End^0_{\overline{\FF}_p}(A)$, which restricts to
$f \mapsto f^t$ on $\End_{\Oo_B}(A) \otimes_{\ZZ} \ZZ_p$, is then given by $\varphi^\dagger = \Lambda^{-1}
\varphi^T\Lambda$, where $\varphi^T$ is the transpose of the matrix $\varphi$. A computation shows that for $\varphi = [\varphi_{ij}] \in \MM_4(W)$,
$$
\varphi^\dagger = \begin{bmatrix} 
\varphi_{44} & b\overline{b}^{-1}\varphi_{34} & -b\overline{b}^{-1}\varphi_{24} & -\varphi_{14} \\ 
b^{-1}\overline{b}\varphi_{43} & \varphi_{33} & -\varphi_{23} & -b^{-1}\overline{b}\varphi_{13} \\
-b^{-1}\overline{b}\varphi_{42} & -\varphi_{32} & \varphi_{22} & b^{-1}\overline{b}\varphi_{12} \\
-\varphi_{41} & -b\overline{b}^{-1}\varphi_{31} & b\overline{b}^{-1}\varphi_{21} & \varphi_{11}
\end{bmatrix}.
$$
If 
$$
\varphi = \begin{bmatrix} x & y\Pi \\ py\Pi & x \end{bmatrix} \in R_{11},
$$
then viewing it as an element of $\MM_4(W)$ as in (\ref{endo matrix}), applying the involution $\dagger$, as described explicitly above, and then viewing it again in $R_{11}$, gives
$$
\varphi^\dagger = \begin{bmatrix} \overline{x} & -y\Pi \\ -py\Pi & \overline{x} \end{bmatrix}.
$$
Therefore $$\varphi^\dagger\varphi = \begin{bmatrix} x\overline{x} - p^2y\overline{y} & 0 \\ 0 & x\overline{x} - p^2y\overline{y}
\end{bmatrix}, $$
so after identifying $\ZZ_p$ with its diagonal image in $\MM_2(\ZZ_{p^2})$, we obtain $Q(\varphi) = x\overline{x}-p^2y\overline{y}$.
A similar computation gives the same result if $D(i) = f_2$. 
\end{proof}

For $j = 1, 2$ let $\theta_j : \Oo_{K_j} \to \Oo_B/\frakm_B$ be a ring homomorphism and let $A_j \in \YYY_j^B(\theta_j)
(\overline{\FF}_\frakP)$ for $p \mid d_B$. There is a unique ring isomorphism $\Oo_{K_1, p} \to \Oo_{K_2, p}$
making the diagram
\begin{equation}\label{identification}
\xymatrix{ 
\Oo_{K_1, p} \ar[rr] \ar[dr]_{\theta_1} && \Oo_{K_2, p} \ar[dl]^{\theta_2} \\
& \Oo_B/\frakm_p & }
\end{equation}
commute. We use this to identify the rings $\Oo_{K_1, p}$ and $\Oo_{K_2, p}$, and call this ring $\Oo^p$.

\begin{definition}\label{type}
With notation as above, if $D(A_1)$ and $D(A_2)$ are isomorphic as $\Delta \otimes_{\ZZ_p} \Oo^p$-modules, we
say that $A_1$ and $A_2$ are of the \textit{same type}. 
\end{definition}

Note that there are two isomorphism classes of $\Delta \otimes_{\ZZ_p} \Oo^p$-modules free of rank $4$ over
$\ZZ_p$, and $A_1$ and $A_2$ being of the same type just means $D(A_1)$ and $D(A_2)$ lie in the same isomorphism
class, and not being of the same type means they lie in the two separate classes. 
This definition is a bit misleading because we will see below that $A_1$ and $A_2$ are of the same type
if and only if $\frakP$ divides $\ker(\Theta)$, where $\Theta : \Oo_K \to \Oo_B/\frakm_B$ is the map induced
by $\theta_1$ and $\theta_2$, so this ``type" is really a property between $\frakP$ and $\Theta$, independent of
$A_1$ and $A_2$. However, the above definition is the easier one to start with in proving the next few results.

\begin{proposition}\label{mixed types}
Suppose $(A_j, i_j) \in \YYY_j^B(\theta_j)(\overline{\FF}_\frakP)$ for $j = 1, 2$, where $p \mid d_B$, and
$A_1$ and $A_2$ are not of the same type. There are isomorphisms of $\ZZ_p$-modules
$$
\Hom_{\Oo_B \otimes_{\ZZ} \DDD}(D(A_1), D(A_2)) \cong \Hom_{\Oo_B \otimes_{\ZZ} \DDD}(D(A_2), D(A_1)) \cong R_{12}, 
$$
where 
$$
R_{12} = \left\{\begin{bmatrix} px & y\Pi \\ y\Pi & x \end{bmatrix} : x, y \in \ZZ_{p^2}\right\} \subset \MM_2(\Delta)
$$
and we have fixed an embedding $\ZZ_{p^2} \hookrightarrow \Delta$ so that $\Delta = \ZZ_{p^2} \oplus \ZZ_{p^2}\Pi$.
Under the isomorphism
$$
\Hom_{\Oo_B}(A_1, A_2) \otimes_{\ZZ} \ZZ_p \mapp{D} \Hom_{\Oo_B \otimes_{\ZZ} \DDD}(D(A_1), D(A_2)) \cong R_{12},
$$
the $\ZZ_p$-quadratic form $\deg^\ast$ on $\Hom_{\Oo_B}(A_1, A_2) \otimes_{\ZZ} \ZZ_p$ is identified with the $\ZZ_p$-quadratic form $u\cdot Q'$ on $R_{12}$, where $u \in \ZZ_p^\times$ and
$$
Q'\begin{bmatrix} px & y\Pi \\ y\Pi & x \end{bmatrix} = p(x\overline{x} - y\overline{y}).
$$
Under the isomorphism
$$
\Hom_{\Oo_B}(A_2, A_1) \otimes_{\ZZ} \ZZ_p \mapp{D} \Hom_{\Oo_B \otimes_{\ZZ} \DDD}(D(A_2), D(A_1)) \cong R_{12},
$$
the quadratic form $\deg^\ast$ is identified with the quadratic form $u^{-1}\cdot Q'$.
\end{proposition}

\begin{proof}
The first claim follows from a computation in coordinates using Proposition \ref{basis}.
Now let $\lambda_j : A_j \to A_j^\vee$ be the unique principal polarization satisfying $i_j(x^\ast) = 
\lambda_j^{-1} \circ i_j(x)^\vee \circ \lambda_j$ for all $x \in \Oo_B$. In the proof of Proposition \ref{quadratic form} we
showed  
$$
\Lambda_j = D(\lambda_j) = \begin{bmatrix} 0 & 0 & 0 & \overline{b}_j \\ 0 & 0 & b_j & 0 \\ 0 & -b_j & 0 & 0 \\ -\overline{b}_j & 0 & 0 & 0
\end{bmatrix} \in \MM_4(W)
$$
for some $b_j \in \ZZ_{p^2}^\times$ satisfying $b_1^{-1}b_2 \in \ZZ_p^\times$. 
We have $D(f^t) = \Lambda_1^{-1}D(f)^\vee \Lambda_2$,
where $D(f)^\vee$ is the dual linear map in $\Hom_{\Oo_B \otimes_{\ZZ} \DDD}(D(A_2)^\vee, D(A_1)^\vee)$. Therefore, through the map $D$, the assignment $f \mapsto f^t$ corresponds to the assignment $\varphi \mapsto \varphi^\dagger = \Lambda_1^{-1}\varphi^T\Lambda_2$. If 
$$
\varphi = \begin{bmatrix} px & y\Pi \\ y\Pi & x \end{bmatrix} \in R_{12}
$$
then a computation shows
$$
\varphi^\dagger\varphi = \begin{bmatrix} p(x\overline{x} - y\overline{y})u & 0
\\ 0 & p(x\overline{x} - y\overline{y})u \end{bmatrix},
$$  
where $u = b_1^{-1}b_2$. The last statement is proved similarly.
\end{proof}

Recall that $(\AAA_1, \AAA_2) \in \XXX_\Theta^B(\overline{\FF}_\frakP)$ and for $p \mid d_B$ 
we are using $\Theta$ to identify $\Oo_{K_1, p}$ and $\Oo_{K_2, p}$ as in (\ref{identification}).

\begin{proposition}\label{quadratic II}
There is a $K_p$-linear isomorphism of $F_p$-quadratic spaces
$$
(V_p(\AAA_1, \AAA_2), \deg_{\CM}) \cong (K_p, \beta_p \cdot \NN_{K_p/F_p})
$$
for some $\beta_p \in F_p^\times$ satisfying 
$$
\beta_p\Oo_{F, p} = \left\{\begin{array}{ll}
\frakp\frakD_p^{-1} & \text{if $p \nmid d_B$} \\
\frakp^2\frakD_p^{-1} & \text{if $p \mid d_B$ and $A_1, A_2$ are of the same type} \\
\frakp\overline{\frakp}\frakD_p^{-1} & \text{if $p \mid d_B$ and $A_1, A_2$ are not of the same type},
\end{array} \right.
$$
where $\frakD_p = \frakD\Oo_{F, p}$,  $\frakp = \frakP \cap \Oo_F$, and $\overline{\frakp}$ is the other prime ideal of $\Oo_F$ lying over $p$. This map takes $L_p(\AAA_1, \AAA_2)$ isomorphically to $\Oo_{K, p}$.
\end{proposition}

\begin{proof}
First suppose $p \nmid d_B$. We will write $L_p$ for
$L_p(\AAA_1, \AAA_2)$. The proof of the existence of the isomorphism taking $L_p$ to $\Oo_{K, p}$ is the same as
for $\ell \neq p$. We may reduce to the case where
$\AAA_1$ and $\AAA_2$ have the same underlying QM abelian surface $A$ because the
idempotents $\varepsilon, \varepsilon' \in \MM_2(W) \cong \Oo_B \otimes_{\ZZ} W$ provide a splitting
$D(A_j) \cong \varepsilon D(A_j) \oplus \varepsilon' D(A_j)$, which means $D(A_1) \cong D(A_2)$ as
$\Oo_B \otimes_{\ZZ}\DDD$-modules and thus, by Lemma \ref{Dieudonne 2},
$$
\Hom_{\Oo_B}(A_1, A_2) \otimes_{\ZZ} \ZZ_p \cong
\End_{\Oo_B \otimes_{\ZZ} \DDD}(D(A_1)) \cong \End_{\Oo_B}(A_1) \otimes_{\ZZ} \ZZ_p.
$$
By Lemmas \ref{Dieudonne 1} and \ref{Dieudonne 2} there are isomorphisms of $\ZZ_p$-algebras
$$
L_p \cong \End_{\Oo_B \otimes_{\ZZ} \DDD}(D(A)) \cong \Delta.
$$
Similarly to the proof of Proposition \ref{quadratic I}, this isomorphism identifies the quadratic form
$\deg^\ast$ on $L_p$ with the quadratic form $\Nrd$ on $\Delta$. If $\frakm_{\Delta} \subset \Delta$ is the unique maximal ideal then 
the dual lattice of $\Delta$ relative to $\Nrd$ is $\frakm_{\Delta}^{-1}$, and there are $\Oo_{K, p}$-linear isomorphisms
$$
\beta_p^{-1}\frakD^{-1}\Oo_{K, p}/\Oo_{K, p} \cong L_p^\vee/L_p \cong \frakm_{\Delta}^{-1}/\Delta,
$$
the first obtained as in Proposition \ref{quadratic I}, so $[\Oo_{K, p} : \beta_p\frakD\Oo_{K, p}] = p^2$.

If $p$ is ramified in $K_1$ or $K_2$ then it is ramified in $F$, and the unique prime of $F$ above $p$ is inert in $K$. From
$p\Oo_F = \frakp^2$ and $[\Oo_{K, p} : \beta_p\frakD\Oo_{K, p}] = p^2$, we must have
$\beta_p\frakD\Oo_{K, p} = \frakP\Oo_{K, p}$, so $\beta_p\Oo_{F, p} = \frakp\frakD_p^{-1}$. 

Now suppose $p$ is inert in $K_1$ and $K_2$. Similarly to the proof of Proposition \ref{quadratic I}, $\Oo_K$ acts on $\frakm_{\Delta}^{-1}/\Delta$ through left multiplication
by the image of $\Oo_K \to \Delta \to \Delta/\frakm_{\Delta} \cong \frakm_{\Delta}^{-1}/\Delta$, where the first map is $t_1 \otimes t_2 \mapsto 
\kappa_2(t_2)\kappa_1(t_1)$. Under the isomorphism $L_p \cong \Delta$, the action $\Delta \to \End_{\Oo_B\otimes_{\ZZ}\overline{\FF}_{\frakP}}(\Lie(A)) \cong \overline{\FF}_{\frakP}$ determines an isomorphism $\gamma : \Delta/\frakm_{\Delta} \to \FF_{p^2}$, allowing us to identify $\kappa_j^{\Lie} :
\Oo_{K_j} \to \End_{\Oo_B\otimes_{\ZZ}\overline{\FF}_{\frakP}}(\Lie(A))$ with the composition
$$
\Oo_{K_j} \mapp{\kappa_j} \Delta \to \Delta/\frakm_{\Delta} \mapp{\gamma} \FF_{p^2}.
$$
However, the map $\Oo_K \to \overline{\FF}_{\frakP}$ defined by $t_1 \otimes t_2 \mapsto \kappa_1^{\Lie}(t_1)\kappa_2^{\Lie}(t_2)$
is precisely the structure map $\Oo_K \to \FF_{\frakP} \hookrightarrow \overline{\FF}_{\frakP}$ by the CM normalization condition,
whose kernel is $\frakP$. It then follows from the factorization of $\kappa_j^{\Lie}$ above that any element of $\frakP$ acts trivially
on $\frakm_{\Delta}^{-1}/\Delta$. Therefore there are isomorphisms of $\Oo_{K, p}$-modules
$$
\Oo_{K, p}/\frakP\Oo_{K, p} \cong \frakm_{\Delta}^{-1}/\Delta \cong \beta_p^{-1}\frakD^{-1}\Oo_{K, p}/\Oo_{K, p},
$$
which shows $\beta_p\frakD\Oo_{K, p} = \frakP\Oo_{K, p}$ and thus $\beta_p\Oo_{F, p} = \frakp\frakD_p^{-1}$.

Next suppose $p \mid d_B$, and first assume $A_1$ and $A_2$ are of the same type, where
$A_j \cong M_j \otimes_{\Oo_{K_j}} E_j$ for some supersingular CM elliptic curve $E_j$ over $\overline{\FF}_\frakP$. 
As mentioned above we identify $\Oo_{K_1, p}$ and $\Oo_{K_2, p}$, and call this ring $\Oo^p$.
By assumption there is a $\Delta \otimes_{\ZZ_p} \Oo^p$-linear isomorphism $f : D(A_1) \to D(A_2)$.
Then there is an isomorphism of $\ZZ_p$-modules
$$
G : \Hom_{\Oo_B \otimes_{\ZZ} \DDD}(D(A_1), D(A_2)) \to \End_{\Oo_B \otimes_{\ZZ} \DDD}(D(A_1))
$$
given by $\varphi \mapsto f^{-1} \circ \varphi$. Also, there are two maps $\Oo^p \rightrightarrows
\End_{\Oo_B \otimes_{\ZZ} \DDD}(D(A_1))$, the first being $\kappa_1$ and the second $(f^{-1})_\ast(\kappa_2)$. However,
since $f$ is $\Oo^p$-linear,
$$
(f^{-1})_\ast(\kappa_2(x)) = f^{-1} \circ \kappa_2(x) \circ f = f^{-1} \circ f \circ \kappa_1(x) = \kappa_1(x),
$$
so under the isomorphism $G$, the two CM actions $\kappa_1$ and $\kappa_2$ are identified with a 
single action $\Oo^p \to \End_{\Oo_B \otimes_{\ZZ} \DDD}(D(A_1))$. 

It follows from the above discussion that we may reduce to the
case where $\AAA_1$ and $\AAA_2$ have the same underlying QM abelian surface $A \cong M_1 \otimes_{\Oo_{K_1}} E_1$ and
$\kappa_1$ and $\kappa_2$ induce the same map $\kappa: \Oo^p \to \End_{\Oo_B \otimes_{\ZZ} \DDD}(D(A))$. If we fix the embedding $\Oo^p \cong \ZZ_{p^2} \hookrightarrow \Delta
\cong \End_{\DDD}(D(E_1))$, the CM action on $E_1$, then there is an isomorphism
$$
L_p = \End_{\Oo_B}(A) \otimes_{\ZZ} \ZZ_p \cong R_{11},
$$
in the notation of Proposition \ref{quadratic form}, with $\kappa : \Oo^p \to R_{11}$ given by $\kappa(x) = \text{diag}(x, x)$, and the quadratic form $\deg^\ast$ on $L_p$  is identified with the quadratic form $Q$ on $R_{11}$ given by
$$
Q\begin{bmatrix} x & y\Pi \\ py\Pi & x \end{bmatrix} = x\overline{x} - p^2y\overline{y}.
$$
The dual lattice of $R_{11}$ relative to $Q$ is
$$
R_{11}^\vee = \left\{\begin{bmatrix} x & p^{-2}y\Pi \\ p^{-1}y\Pi & x \end{bmatrix} : x, y \in \ZZ_{p^2}\right\},
$$
so $[R_{11}^\vee : R_{11}] = p^4$. Since there are isomorphisms of $\Oo_{K, p}$-modules
$$
\beta_p^{-1}\frakD^{-1}\Oo_{K, p}/\Oo_{K, p} \cong L_p^\vee/L_p \cong R_{11}^\vee/R_{11},
$$
we obtain $[\Oo_{K, p} : \beta_p\frakD\Oo_{K, p}] = p^4$.

Under the isomorphism $L_p \cong R_{11}$ there is an action $R_{11} \to \End_{\Delta \otimes_{\ZZ_p}\overline{\FF}_{\frakP}}(\Lie(A)) \cong \overline{\FF}_{\frakP}$, and a computation in coordinates shows that any element of
\begin{equation}\label{proof}
\frakM = \left\{\begin{bmatrix} px & y\Pi \\ py\Pi & px \end{bmatrix} : x, y \in \ZZ_{p^2} \right\} \subset R_{11},
\end{equation}
a maximal ideal of $R_{11}$, acts trivially on $D(A)/\VVV D(A) \cong \Lie(A)$, which shows $\frakM = \ker(R_{11}
\to \overline{\FF}_{\frakP})$. Hence, the action $R_{11} \to \End_{\Delta\otimes_{\ZZ_p}\overline{\FF}_{\frakP}}(\Lie(A))$ determines an isomorphism
$\gamma : R_{11}/\frakM \to \FF_{p^2}$, which allows us to identify $\kappa^{\Lie} : \Oo^p \to \End_{\Delta\otimes_{\ZZ_p}\overline{\FF}_{\frakP}}(\Lie(A))$
with the composition
$$
\Oo^p \mapp{\kappa} R_{11} \to R_{11}/\frakM \mapp{\gamma} \FF_{p^2}.
$$
However, the map $\Oo_{K, p} \cong \Oo^p \otimes_{\ZZ_p} \Oo^p \to \overline{\FF}_{\frakP}$ defined by $t_1 \otimes t_2 \mapsto \kappa^{\Lie}(t_1)\kappa^{\Lie}(t_2)$ is the structure map $\Oo_{K, p} \to \FF_\frakP \hookrightarrow \overline{\FF}_{\frakP}$ by the CM normalization
condition, so its kernel is $\frakP\Oo_{K, p}$. It follows from the factorization of $\kappa^{\Lie}$ above that if
$t_1 \otimes t_2 \in \frakP^2\Oo_{K, p}$ then $\kappa(t_2)\kappa(t_1) \in \frakM^2$. But $\kappa(t_j) = \text{diag}(t_j, t_j)$, so
$t_2t_1 \in p^2\ZZ_{p^2}$ for $t_1 \otimes t_2 \in \frakP^2\Oo_{K, p}$. Then for 
$$
\varphi = \begin{bmatrix} x & p^{-2}y\Pi \\ p^{-1}y\Pi & x \end{bmatrix}  \in R_{11}^\vee
$$
and $t_1 \otimes t_2 \in \frakP^2\Oo_{K, p}$, under the action of $\Oo_{K, p}$ on
$L_p$,
$$
(t_1 \otimes t_2) \bullet \varphi = \kappa(t_2)\varphi\kappa(\overline{t}_1) = \begin{bmatrix}
t_2\overline{t}_1 x & p^{-2}t_2t_1y\Pi \\ p^{-1}t_2t_1y\Pi & t_2\overline{t}_1x \end{bmatrix} \in R_{11}.
$$
This shows $\frakP^2\Oo_{K, p}$ acts trivially on $R_{11}^\vee/R_{11}$, and conversely, reversing this argument
shows that any element of $\Oo_{K, p}$ acting trivially on $R_{11}^\vee/R_{11}$ is in $\frakP^2\Oo_{K, p}$.
Hence there is
an $\Oo_{K, p}$-linear map $\Oo_{K, p}/\frakP^2\Oo_{K, p} \hookrightarrow R_{11}^\vee/R_{11}$ given by
$x \mapsto x \bullet P$ for any nonzero $P \in R_{11}^\vee/R_{11}$. But $\frakP^2$ has norm $p^4 = [R_{11}^\vee : R_{11}]$, so there are isomorphisms
of $\Oo_{K, p}$-modules
$$
\Oo_{K, p}/\frakP^2\Oo_{K, p} \cong R_{11}^\vee/R_{11} \cong \beta_p^{-1}\frakD^{-1}\Oo_{K, p}/\Oo_{K, p}.
$$
It follows that $\beta_p\frakD\Oo_{K, p} = \frakP^2\Oo_{K, p}$ and thus $\beta_p\Oo_{F, p} = \frakp^2\frakD_p^{-1}$.

Next assume $A_1$ and $A_2$ are not of the same type, with $A_j \cong M_j \otimes_{\Oo_{K_j}} E_j$. As before we identify $\Oo_{K_1, p}$ with $\Oo_{K_2, p}$ and call this ring $\Oo^p$.
Let $\frakg$ be the connected $p$-divisible group of height $2$ and dimension $1$ over $\overline{\FF}_\frakP$. Isomorphisms $E_j[p^\infty] \cong \frakg$ may be chosen in such a way
that the CM actions $g_1 : \Oo^p \to \End(E_1[p^\infty]) \cong \Delta$ and $g_2 : \Oo^p \to \End(E_2[p^\infty]) \cong \Delta$
have the same image in $\Delta$. 
Fix an embedding $\ZZ_{p^2} \hookrightarrow \Delta$ and a uniformizer $\Pi \in
\Delta$ satisfying $\Pi g_1(x) = g_1(\overline{x})\Pi$ for all $x \in \Oo^p$. 
By Lemma \ref{Dieudonne 2} and Proposition \ref{mixed types} there are isomorphisms of $\ZZ_p$-modules
$$
L_p \cong \Hom_{\Oo_B \otimes_{\ZZ} \DDD}(D(A_1), D(A_2)) \cong R_{12},
$$
and the quadratic form $\deg^\ast$ on $L_p$ is identified with the quadratic form $uQ'$ on $R_{12}$, where 
$u \in \ZZ_p^\times$ and
$$
Q'\begin{bmatrix} px & y\Pi \\ y\Pi & x \end{bmatrix} = p(x\overline{x} - y\overline{y}).
$$
The dual lattice of $R_{12}$ relative to $uQ'$ is 
$$
R_{12}^\vee = u^{-1}\cdot \left\{\begin{bmatrix} x & p^{-1}y\Pi \\ p^{-1}y\Pi & p^{-1}x \end{bmatrix} : x, y \in \ZZ_{p^2} \right\},
$$
so $[R_{12}^\vee : R_{12}] = p^4$. As before this gives $[\Oo_{K, p} : \beta_p\frakD\Oo_{K, p}] = p^4$. Fixing ring isomorphisms
$$
\End_{\Oo_B}(A_1) \otimes_{\ZZ} \ZZ_p \cong R_{11} \cong \End_{\Oo_B}(A_2) \otimes_{\ZZ} \ZZ_p,
$$
it makes sense to take the product $\kappa_2(t_2)\kappa_1(t_1)$ in $R_{11}$ for $t_1, t_2 \in \Oo^p$, and
$$
\kappa_2(t_2)\kappa_1(t_1) = \text{diag}(g_2(t_2)g_1(t_1), g_2(t_2)g_1(t_1)),
$$
where $g_2(t_2)g_1(t_1)$ is the product in the common image of $g_1$ and $g_2$ in $\Delta$. As in the case of $A_1$ and $A_2$
having the same type, the action $R_{11} \to \End_{\Delta\otimes_{\ZZ_p}\overline{\FF}_{\frakP}}(\Lie(A_j)) \cong \overline{\FF}_{\frakP}$, for $j = 1, 2$, determines
an isomorphism $\gamma_j : R_{11}/\frakM \to \FF_{p^2}$, which allows us to identify $\kappa_j^{\Lie} : \Oo^p \to \End_{\Delta\otimes_{\ZZ_p}\overline{\FF}_{\frakP}}(\Lie(A_j))$ with the composition
$$
\Oo^p \mapp{\kappa_j} R_{11} \to R_{11}/\frakM \mapp{\gamma_j} \FF_{p^2}.
$$
As above, the map $\Oo_{K, p} \to \overline{\FF}_{\frakP}$ defined by $t_1 \otimes t_2 \mapsto \kappa_1^{\Lie}(t_1)\kappa_2^{\Lie}(t_2)$ is the structure map $\Oo_{K, p} \to \FF_{\frakP} \hookrightarrow \overline{\FF}_{\frakP}$. Therefore $t_1 \otimes t_2 \in \frakP\Oo_{K, p}$ if and only if $\kappa_1^{\Lie}(t_1)\kappa_2^{\Lie}(t_2) = 0$, if and only if $\kappa_2(t_2)\kappa_1(t_1) \in \frakM$. 

Let $\overline{\frakP}$ be the other prime ideal of $\Oo_K$ lying over $p$, so $\overline{\frakP} \cap \Oo_F = \overline{\frakp}$. For $t_1 \otimes t_2 \in \Oo_K$, we have 
$$
t_1 \otimes t_2 \in \overline{\frakP} \iff \overline{t}_1 \otimes t_2 \in \frakP \iff t_1 \otimes \overline{t}_2 \in \frakP.
$$
For
$$
\varphi = u^{-1}\cdot\begin{bmatrix} x & p^{-1}y\Pi \\ p^{-1}y\Pi & p^{-1}x \end{bmatrix} \in R_{12}^\vee
$$
and $t_1 \otimes t_2 \in \Oo_{K, p}$,
$$
(t_1 \otimes t_2) \bullet \varphi = \kappa_2(t_2)\varphi\kappa_1(\overline{t}_1) = u^{-1}\cdot\begin{bmatrix} g_2(t_2)g_1(\overline{t}_1)x & 
p^{-1}g_2(t_2)g_1(t_1)y\Pi \\ p^{-1}g_2(t_2)g_1(t_1)y\Pi & p^{-1}g_2(t_2)g_1(\overline{t}_1)x \end{bmatrix}.
$$
Therefore
\begin{align*}
\text{$(t_1 \otimes t_2) \bullet \varphi \in R_{12}$ for all $\varphi$}  &\iff \text{$g_2(t_2)g_1(t_1) \in p\ZZ_{p^2}$ and $g_2(t_2)g_1(\overline{t}_1) \in p\ZZ_{p^2}$} \\
&\iff \text{$\kappa_2(t_2)\kappa_1(t_1) \in \frakM$ and $\kappa_2(t_2)\kappa_1(\overline{t}_1) \in \frakM$} \\
&\iff t_1 \otimes t_2 \in \frakP\Oo_{K, p} \cap \overline{\frakP}\Oo_{K, p} = \frakP\overline{\frakP}\Oo_{K, p}.
\end{align*}
This shows an element of $\Oo_{K, p}$ acts trivially 
on $R_{12}^\vee/R_{12}$ if and only if it is in $\frakP\overline{\frakP}\Oo_{K, p}$. Since $\frakP\overline{\frakP}$ has norm $p^4 = [R_{12}^\vee : R_{12}]$, similarly to above we obtain $\beta_p\Oo_{F, p} = \frakp\overline{\frakp}\frakD_p^{-1}$.
\end{proof}
 
 The next result gives an analogous conclusion as that of Lemma \ref{ml torsion} in the case when $\ell = \charr(k)$.

\begin{lemma}\label{mp torsion}
Suppose $A \in \YYY^B(k)$ with $k$ an algebraically closed field over $\Oo_{\bm{K}}$ of characteristic $p \mid d_B$.
There is an isomorphism of $\Oo_B/\frakm_p$-algebras $\End_{\Oo_B/\frakm_p}(A[\frakm_p]) \cong \Oo_B/\frakm_p$.
\end{lemma}

\begin{proof}
We will use Dieudonn\'e modules. 
Since $A[p]$ and $A[\frakm_p]$ are finite $p$-group schemes over $\Spec(k)$, they have associated
covariant Dieudonn\'e modules $D(A[p])$ and $D(A[\frakm_p])$, which are $\DDD_k$-modules of length $4$ and $2$ over $W(k)$,
respectively.
There is an exact sequence of group schemes
$$
0 \to A[\frakm_p] \to A[p] \mapp{i(x_p)} A[p] \to 0,
$$
where $x_p$ is any element of $\frakm_p$ whose image generates the principal ideal $\frakm_p/p\Oo_B \subset
\Oo_B/p\Oo_B$.
Since $D$ is an exact functor, we obtain an exact sequence of $\DDD_k$-modules
$$
0 \to D(A[\frakm_p]) \to D(A[p]) \mapp{i(x_p)} D(A[p]) \to 0.
$$
By definition, $D(A) = \mil_nD(A[p^n])$, and there is an isomorphism of $\DDD_k$-modules $D(A[p]) \cong D(A)/pD(A)$ (\cite[Proposition 7.2.6]{BC2}). Under this isomorphism the
map $i(x_p) : D(A[p]) \to D(A[p])$ is identified with the map $i(\Pi) : D(A)/pD(A) \to D(A)/pD(A)$, where $\Pi \in \Oo_{B, p}$
is a uniformizer. It follows that $D(A[\frakm_p]) \cong D_p$, where
$$
D_p = \ker(i(\Pi) : D(A)/pD(A) \to D(A)/pD(A)),
$$
and thus, since $D$ is an equivalence of categories,
$$
\End_{\Oo_B/\frakm_p}(A[\frakm_p]) \cong \End_{\Oo_B/\frakm_p \otimes_{\ZZ}\DDD_k}(D_p).
$$

Using the $W(k)$-basis for $D(A)$ as in Corollary \ref{basis 2}, and
considering the two possible forms (\ref{OB action}), a computation in coordinates, similar to that of Lemma \ref{Lie algebra}, shows $\End_{\Oo_B/\frakm_p \otimes_{\ZZ} \DDD_k}(D_p) \cong \FF_{p^2}$.
Therefore the action $i : \Oo_B/\frakm_p \to \End_{\Oo_B/\frakm_p}(A[\frakm_p])$ is an isomorphism of 
$\Oo_B/\frakm_p$-algebras.
\end{proof}

\begin{corollary}\label{max ideal}
Suppose $A \in \YYY^B(k)$ with $k$ an algebraically closed field over $\Oo_{\bm{K}}$. There is an isomorphism of 
$\Oo_B/\frakm_B$-algebras $\End_{\Oo_B/\frakm_B}(A[\frakm_B]) \cong \Oo_B/\frakm_B$.
\end{corollary}

\begin{proof}
Combine Lemmas \ref{ml torsion} and \ref{mp torsion} with the  isomorphism of group schemes $A[\frakm_B] \cong \prod_{\ell \mid d_B}A[\frakm_\ell]$.
\end{proof}

We conclude this section with giving an equivalent formulation of ``type" as stated in Definition \ref{type}.

\begin{lemma}\label{types II}
For $p \mid d_B$ let $w_p$ be the corresponding element of the Atkin-Lehner group $W_B$, so $w_p = \Pi$ is a 
uniformizer in $\Delta \cong \Oo_{B, p}$. Let $(A, i, \kappa) \in \YYY^B(\overline{\FF}_p)$ and set 
$A' = w_p\cdot A \in \YYY^B(\overline{\FF}_p)$. Then $D(A)$ and $D(A')$ are not isomorphic as $\Delta \otimes_{\ZZ_p}
\Oo_{\bm{k}, p}$-modules.
\end{lemma}

\begin{proof}
Recall that $A' = (A, i', \kappa)$ where $i' : \Delta \to \End_{\overline{\FF}_p}(A) \otimes_{\ZZ} \ZZ_p$ 
is given by $i'(x) = i(\Pi^{-1} x\Pi)$.
Suppose there is a $\Delta$-linear isomorphism 
$D(A) \cong D(A')$. Then there is
a $u \in \End_{\DDD}(D(A))^\times$ such that $i'(x) = u\circ i(x) \circ u^{-1}$  for all $x \in \Delta$, viewing
$i(x)$ and $i'(x)$ as their induced endomorphisms of $D(A)$. Therefore
conjugation by $u$ on $i(\Delta) \subset \End_{\DDD}(D(A))$ is equal to conjugation by $i(\Pi)$, which means $i(\Pi) = u\circ z$ for some $z \in Z(i(\Delta))$, the center of $i(\Delta)$. However, $i(\Delta) \cong \Delta$ has center
$\ZZ_p$, so $z \in \ZZ_p \subset \MM_2(\Delta)$. Viewing $i(\Pi)$, $u$, and $z$ as their corresponding
endomorphisms of $A$, we have $\deg(i(\Pi)) = \Nrd(\Pi)^2 = p^2$ and $\deg(u \circ z) = \deg(z) = p^{4k}$ for some integer $k \gqq 0$ since $\deg([p]) = p^4$. This is a contradiction.
\end{proof}

\begin{proposition}\label{reflex II}
Let $(\AAA_1, \AAA_2) \in \XXX_\Theta^B(\overline{\FF}_\frakP)$ with $\frakP$ lying over $p \mid d_B$. Then 
$\frakP$ divides $\ker(\Theta)$ if and only if $A_1$ and $A_2$ are of the same type.
\end{proposition}  

\begin{proof}
Suppose $A_1$ and $A_2$ are of the same type. The proof essentially follows the part of the proof of Proposition
\ref{quadratic II} starting around (\ref{proof}). Since $A_1$ and $A_2$ are of the same type, there is an isomorphism
of $\ZZ_p$-modules $L_p = L_p(\AAA_1, \AAA_2) \cong R_{11}$. Fix ring isomorphisms
$$
\End_{\Oo_B}(A_1) \otimes_{\ZZ} \ZZ_p \cong R_{11} \cong \End_{\Oo_B}(A_2) \otimes_{\ZZ} \ZZ_p.
$$
For $j \in \{1, 2\}$, under the action 
$$
R_{11} \to \End_{\Oo_B/\frakm_p}(A_j[\frakm_p]) \cong \End_{\Oo_B/\frakm_p \otimes_{\ZZ} \DDD}(D(A_j[\frakm_p])) \cong \FF_{p^2}, 
$$
any element of
$$
\frakM = \left\{\begin{bmatrix} px & y\Pi \\ py\Pi & px \end{bmatrix} : x, y \in \ZZ_{p^2} \right\} \subset R_{11}
$$
acts trivially on 
$$
D(A_j[\frakm_p]) = \ker(i_j(\Pi) : D(A_j)/pD(A_j) \to D(A_j)/pD(A_j)), 
$$
so $\frakM = \ker(R_{11} \to \FF_{p^2})$. It follows that the map $R_{11} \to \End_{\Oo_B/\frakm_p}(A_j[\frakm_p])$
determines an isomorphism $\gamma_j : R_{11}/\frakM \to \FF_{p^2}$ which allows us to identify
$\kappa_j^{\frakm_p} : \Oo_{K_j} \to \End_{\Oo_B/\frakm_p}(A_j[\frakm_p])$ with the composition
$$
\Oo_{K_j} \mapp{\kappa_j} R_{11} \to R_{11}/\frakM \mapp{\gamma_j} \FF_{p^2}.
$$
Let $\frakQ \subset \Oo_K$ be the prime over $p$ dividing $\ker(\Theta)$, so $\frakQ$ is the kernel of the map $\Oo_K \to \FF_{p^2}$ defined by $t_1 \otimes t_2 \mapsto
\kappa_1^{\frakm_p}(t_1)\kappa_2^{\frakm_p}(t_2)$. Using the factorization of $\kappa_j^{\frakm_p}$ given
above and following the rest of this part of the proof of Proposition \ref{quadratic II}, we find that
an element of $\Oo_{K, p}$ acts trivially on $L_p^\vee/L_p$ if and only if it is in $\frakQ^2\Oo_{K, p}$. However, the same
is true for $\frakP$ in place of $\frakQ$, so $\frakQ^2\Oo_{K, p} = \frakP^2\Oo_{K, p}$ and therefore $\frakP = \frakQ$.

Now suppose $A_1$ and $A_2$ are not of the same type and let $\theta_j = \Theta|_{\Oo_{K_j}}$. Define a ring homomorphism $\Omega : \Oo_K \to \Oo_B/\frakm_B$ with $\omega_j = \Omega|_{\Oo_{K_j}}$
being defined by $\omega_j^{\frakm_\ell} = \theta_j^{\frakm_\ell}$ for all $\ell \neq p$ and $j = 1, 2$, $\omega_1^{\frakm_p}
= \theta_1^{\frakm_p}$, and $\omega_2^{\frakm_p}(x) = \theta_2^{\frakm_p}(\overline{x})$.
Consider the CM pair $(\AAA_1, \AAA_2')$, where $\AAA_2' = w_p\cdot\AAA_2$ and $w_p$ is the Atkin-Lehner operator at $p$. The map
$$
(\kappa_2')^{\frakm_p} : \Oo_{K_2} \to \End_{\Oo_B/\frakm_p}(A_2'[\frakm_p]) \cong \Oo_B/\frakm_p
$$
is given by $(\kappa_2')^{\frakm_p}(x) = \kappa_2^{\frakm_p}(\overline{x})$. The resulting map $\Oo_K \to \Oo_B/\frakm_p$ for the pair $(\AAA_1, \AAA_2')$ is given by
$$
t_1 \otimes t_2 \mapsto \kappa_1^{\frakm_p}(t_1)\kappa_2^{\frakm_p}(\overline{t}_2),
$$
so $(\AAA_1, \AAA_2') \in \XXX^B_\Omega(\overline{\FF}_\frakP)$ and the kernel of this map
is $\overline{\frakQ}$, where $\frakQ$ is the prime over $p$ dividing $\ker(\Theta)$.
As $A_1$ and $w_p\cdot A_2$ are of the same type (Lemma \ref{types II}), $\overline{\frakQ} = \frakP$
by the first part of the proof applied to $(\AAA_1, \AAA_2')$, so $\frakP$ does not divide $\ker(\Theta)$.\end{proof}

\subsection{Cases combined}
Let $(\AAA_1, \AAA_2) \in \XXX_\Theta^B(\overline{\FF}_\frakP)$ with $\frakP$ lying over some prime $p$, and let
$\frakp = \frakP \cap \Oo_F$. Set $\fraka_\Theta = \ker(\Theta) \cap \Oo_F$.

\begin{theorem}\label{quadratic combined}
For any finite idele $\beta \in \widehat{F}^\times$ satisfying $\beta\widehat{\Oo}_F = \fraka_\Theta\frakp\frakD^{-1}\widehat{\Oo}_F$, there is a $\widehat{K}$-linear isomorphism
of $\widehat{F}$-quadratic spaces
$$
(\widehat{V}(\AAA_1, \AAA_2), \deg_{\CM}) \cong (\widehat{K}, \beta\cdot\NN_{K/F})
$$
taking $\widehat{L}(\AAA_1, \AAA_2)$ isomorphically to $\widehat{\Oo}_K$.
\end{theorem}

\begin{proof}
Combining Propositions \ref{quadratic I}, \ref{quadratic II}, and \ref{reflex II} proves the claim for some $\beta \in \widehat{F}^\times$ satisfying $\beta\widehat{\Oo}_F = \fraka_\Theta\frakp\frakD^{-1}\widehat{\Oo}_F$, and the surjectivity of the norm map $\widehat{\Oo}_K^\times \to
\widehat{\Oo}_F^\times$ gives the result for all such $\beta$.
\end{proof} 

Recall the definitions of the functions $\rho$ and $\rho_\ell$ from the introduction.

\begin{definition}
For each prime number $\ell$ and $\alpha \in F_{\ell}^{\times}$ define the \textit{orbital integral} at $\ell$ by
$$
O_{\ell}(\alpha, \AAA_1, \AAA_2) = \left\{\begin{array}{ll}
\rho_{\ell}(\alpha\frakD_{\ell}) & \text{if $\ell \neq p$, $\ell \nmid d_B$} \\
\rho_{\ell}(\alpha\frakl(\ell)^{-1}\frakD_{\ell}) & \text{if $\ell \neq p$, $\ell \mid d_B$} \\
\rho_p(\alpha\frakp^{-1}\frakl(p)^{-1}\frakD_p) & \text{if $\ell = p$},
\end{array} \right.
$$
where $\frakl(\ell)$ is the prime over $\ell$ dividing $\fraka_\Theta$, with the convention that 
$\frakl(p) = \Oo_F$ if $p \nmid d_B$.
\end{definition}

It is possible to give a definition of $O_{\ell}(\alpha, \AAA_1, \AAA_2)$ as a sum of characteristic functions, analogous 
to \cite[(2.11)]{Howard}, but we do not need the details of that here. This alternative definition agrees with the one given
above by a proof identical to that of \cite[Lemmas 2.19, 2.20]{Howard}, using Propositions \ref{quadratic I} and \ref{quadratic II} in place of Lemmas 2.10 and 2.11 of \cite{Howard}.

\begin{theorem}\label{orbital}
Let $p$ be a prime number that is nonsplit in $K_1$ and $K_2$ and suppose $(\AAA_1, \AAA_2)$ is a CM pair
over $\overline{\FF}_p$. Recall $\Gamma = \Cl(\Oo_{K_1}) \times \Cl(\Oo_{K_2})$ and $w_j = |\Oo_{K_j}^\times|$. For any $\alpha \in F^\times$ totally positive,
$$
\sum_{(\fraka_1, \fraka_2) \in \Gamma} \#\{f \in L(\fraka_1 \otimes_{\Oo_{K_1}} \AAA_1, \fraka_2 \otimes_{\Oo_{K_2}} \AAA_2) : 
\deg_{\CM}(f) = \alpha\} = \frac{w_1w_2}{2}\prod_{\ell}O_{\ell}(\alpha, \AAA_1, \AAA_2).
$$
\end{theorem}

\begin{proof}
The proof is formally the same as \cite[Proposition 2.18]{Howard}, replacing the definitions there with our
analogous definitions, and using the above comment to match up the different definitions of the orbital integral.
\end{proof}

\begin{proposition}\label{orbital II}
For any $\alpha \in F^\times$ we have
$$
\prod_{\ell}O_{\ell}(\alpha, \AAA_1, \AAA_2) = \rho(\alpha\fraka_\Theta^{-1}\frakp^{-1}\frakD).
$$
\end{proposition}

\begin{proof}
This follows from the definition of $O_{\ell}(\alpha, \AAA_1, \AAA_2)$ and the product expansion for $\rho$.
\end{proof}

\section{Deformation theory}

This section is devoted to the calculation of the length of the local ring $\OOO^{\text{sh}}_{\XXX_{\Theta, \alpha}^B, x}$,
which relies on the deformation theory of objects $(\AAA_1, \AAA_2, f)$ of $\XXX_{\Theta, \alpha}^B(
\overline{\FF}_\frakP)$. We continue with the notation of Section 3.3. Fix a prime ideal $\frakP \subset \Oo_K$
of residue characteristic $p$ and set $\WWW = \WWW_{K_\frakP}$ and $\mathbf{CLN} = \mathbf{CLN}_{K_\frakP}$.
Let $\frakg$ be the connected
$p$-divisible group of height $2$ and dimension $1$ over $\overline{\FF}_\frakP$.

\begin{definition}
Let $(\AAA_1, \AAA_2)$ be a CM pair over $\overline{\FF}_\frakP$
and $R \in \mathbf{CLN}$.
A \textit{deformation} of $(\AAA_1, \AAA_2)$ to $R$ is a CM pair $(\widetilde{\AAA}_1, \widetilde{\AAA}_2)$ over
$R$ together with an isomorphism of CM pairs $(\widetilde{\AAA}_1, \widetilde{\AAA}_2)_{/\overline{\FF}_\frakP} \cong
(\AAA_1, \AAA_2)$.
\end{definition}

Given a CM pair $(\AAA_1, \AAA_2)$ over $\overline{\FF}_\frakP$, define $\Def(\AAA_1, \AAA_2)$ to be the functor
$\mathbf{CLN} \to \mathbf{Sets}$ that assigns to each $R \in \mathbf{CLN}$ the set of isomorphism
classes of deformations of $(\AAA_1, \AAA_2)$ to $R$. By Proposition \ref{reduction}, 
$$
\Def(\AAA_1, \AAA_2) \cong \Def_{\Oo_B}(A_1, \Oo_{K_1}) \times \Def_{\Oo_B}(A_2, \Oo_{K_2})
$$
is represented by $\WWW \widehat\otimes_{\WWW} \WWW \cong \WWW$. Given a nonzero $f \in L(\AAA_1, \AAA_2)$ define
$\Def(\AAA_1, \AAA_2, f)$ to be the functor $\mathbf{CLN} \to \mathbf{Sets}$ that assigns to each $R \in
\mathbf{CLN}$ the set of isomorphism classes of deformations of $(\AAA_1, \AAA_2, f)$ to $R$.

\subsection{Deformations of CM pairs}
Fix a ring homomorphism $\Theta : \Oo_K \to \Oo_B/\frakm_B$, a CM pair $(\AAA_1, \AAA_2) \in \XXX_\Theta^B(\overline{\FF}_\frakP)$, and a nonzero $f \in L(\AAA_1, \AAA_2)$. Assume $p$ is nonsplit in $K_1$ and $K_2$, and let $\frakp = \frakP \cap \Oo_F$.

\begin{proposition}\label{representability I} Suppose $p \nmid d_B$. \\
{\upshape (a)} If $p$ is inert in $K_1$ and $K_2$, then the functor $\Def(\AAA_1, \AAA_2, f)$ is represented by a local Artinian $\WWW$-algebra of length
$$
\frac{\ord_{\frakp}(\deg_{\CM}(f)) + 1}{2}.
$$ \\
{\upshape (b)} If $p$ is ramified in $K_1$ or $K_2$, then $\Def(\AAA_1, \AAA_2, f)$ is represented by
a local Artinian $\WWW$-algebra of length
$$
\frac{\ord_{\frakp}(\deg_{\CM}(f)) + \ord_\frakp(\frakD) + 1}{2}.
$$
\end{proposition}

\begin{proof}
The proofs of (a) and (b) are the same as \cite[Lemmas 2.23, 2.24]{Howard}, respectively, using the following facts. If $A_j \cong M_j \otimes_{\Oo_{K_j}} E_j$ for some supersingular CM elliptic curve $E_j$ over $\overline{\FF}_\frakP$ then $\widetilde{E}_j \mapsto M_j \otimes_{\Oo_{K_j}} \widetilde{E}_j$ defines an isomorphism of functors
$\Def(E_j, \Oo_{K_j}) \to \Def_{\Oo_B}(A_j, \Oo_{K_j})$ from $\mathbf{CLN}$ to $\mathbf{Sets}$. Next, the idempotents $\varepsilon, \varepsilon' \in \MM_2(\ZZ_p) \cong \Oo_B \otimes_{\ZZ} \ZZ_p$ induce a splitting $A_j[p^{\infty}] \cong \varepsilon A_j[p^{\infty}] \times \varepsilon'A_j[p^{\infty}]$ with $\varepsilon A_j[p^{\infty}] \cong \varepsilon' A_j[p^{\infty}] \cong E_j[p^{\infty}]$.
\end{proof}

We will need an analogue for QM abelian surfaces of a result of Gross (\cite[Proposition 3.3]{Gross}) that gives 
the structure of the endomorphism ring of the modulo $m$ reduction of the universal deformation of the
$p$-divisible group $\frakg$.
This is what we prove next.

\begin{lemma}\label{lift}
Let $(A, i, \kappa) \in \YYY^B(\overline{\FF}_\frakP)$ for $p \mid d_B$. Set $R = \End_{\Oo_B}(A[p^{\infty}])$, let $\Aa$ be the universal deformation of $A$ to $\WWW = W$, and for each integer $m \gqq 1$ set
$$
R_m = \End_{\Oo_B}(\Aa[p^{\infty}]
\otimes_W W_m),
$$
where $W_m = W/(p^m)$. Then the reduction map $R_m \hookrightarrow R$ induces an isomorphism
$$
R_m \cong \Oo^p + p^{m-1}R,
$$
where $\Oo^p = \kappa(\Oo_{\bm{k}, p})$.
\end{lemma}

\begin{proof}
We will use Grothendieck-Messing deformation theory. Let $D = D(A)$ be the covariant Dieudonn\'e module 
of $A$ and set $\Oo = \Oo_B \otimes_{\ZZ} \Oo^p$. For any $m \gqq 1$ there 
are $\Oo$-linear isomorphisms of $W_m$-modules
$$
H_1^{\text{dR}}(\Aa \otimes_W W_m) \cong D \otimes_W W_m \cong D/p^mD.
$$

For any $m \gqq 1$ the surjection $W_m \to \overline{\FF}_\frakP$ has kernel $pW/p^mW$,
which has the canonical divided power structure. By Proposition \ref{reduction}, $(A, i, \kappa)$ has a 
unique deformation to $W_m$, namely $\Aa_m = \Aa \otimes_W W_m$. Therefore
there is a unique direct summand $M_m \subset \widetilde{H}_1^{\text{dR}}(A)$, where $\widetilde{H}_1^{\text{dR}}(A) = H_1^{\text{dR}}(\widetilde{A})$ for any deformation $\widetilde{A}$ of $A$
to $W_m$, stable under the action of $\Oo$ on
$\widetilde{H}_1^{\text{dR}}(A)$, that reduces to $\text{Fil}(A)$, the Hodge filtration of $A$, and such that the diagram, corresponding to the CM normalization condition, 
\begin{equation}\label{Hodge}
\xymatrix{
\Oo^p \ar[rr] \ar[dr] & & \End_{\Oo_B\otimes_{\ZZ}W_m}(\widetilde{H}_1^{\text{dR}}(A)/M_m) \\
& W_m \ar[ur] }
\end{equation}
commutes, namely $M_m = \text{Fil}(\Aa_m)$. The Hodge sequence for $A$ takes the form
$$
0 \to \text{Fil}(A) \to D/pD \to \Lie(A) \to 0.
$$
Using a $W$-basis $\{e_1, e_2, e_3, e_4\}$ for $D$ as in Proposition \ref{basis},
it also defines an $\overline{\FF}_\frakP$-basis for $D/pD$, and
$\text{Fil}(A) = \ker(D/pD \to D/\VVV D)$
has $\{e_2, e_4\}$ as an $\overline{\FF}_\frakP$-basis. 

Any $f \in R$ induces an endomorphism of $H_1^{\text{dR}}(A) \cong D \otimes_W \overline{\FF}_\frakP$, which lifts to an endomorphism $\widetilde{f}$ of
$\widetilde{H}_1^{\text{dR}}(A) \cong D/p^mD$, and $f$ lifts to an element of $R_m$ if and only if
$\widetilde{f}(M_m) \subset M_m$. The map $\widetilde{f}$
corresponds to the reduction modulo $p^m$ of $f : D \to D$. 
Consider the $W_m$-submodule
$N = \text{Span}_{W_m}(e_2, e_4) \subset D/p^mD$. In the basis $\{e_n\}$, the $\Oo_{\bm{k}}$-action on $D$ is given by
(\ref{OK action}) and the $\Oo_B$-action is given by one of the matrices in (\ref{OB action}). Each of these maps
stabilizes $N$, so $N$ is an $\Oo$-stable direct summand of $D/p^mD$ that reduces to $\text{Fil}(A) = \text{Span}_{\overline{\FF}_\frakP}(e_2, e_4)$ modulo $p$. Also, a computation in coordinates shows that the diagram (\ref{Hodge}) commutes with $N$ is place of $M_m$. Hence $N \cong M_m$ under the isomorphism
$D/p^mD \cong \widetilde{H}_1^{\text{dR}}(A)$. Expressing 
$$
f = \begin{bmatrix} x & y\Pi \\ py\Pi & x \end{bmatrix} \in R 
$$
as an element of $\MM_4(W)$ as in (\ref{endo matrix}), we have
\begin{align*}
\text{$f$ lifts to an element of $R_m$}  &\iff \widetilde{f}(N) \subset N \\
&\iff f(e_2), f(e_4) \in We_2 + We_4 + p^mD \\
&\iff y \in p^{m-1}\Oo_{\bm{k}, p} \\
&\iff f \in \Oo^p + p^{m-1}R.  \qedhere
\end{align*}
\end{proof}

\begin{proposition}\label{representability II}
If $p \mid d_B$ and $\frakP$ divides $\ker(\Theta)$, then $\Def(\AAA_1, \AAA_2, f)$ is represented by a local Artinian $\WWW$-algebra of length $\frac{1}{2}\ord_{\frakp}(\deg_{\CM}(f))$.
\end{proposition}

\begin{proof}
As usual $A_j \cong M_j \otimes_{\Oo_{K_j}} E_j$ 
for some supersingular elliptic curve $E_j$. Isomorphisms $E_j[p^{\infty}] \cong \frakg$ may be chosen so that the CM actions $\Oo_{K_1, p} \to 
\Delta$ and $\Oo_{K_2, p} \to \Delta$ on $E_1$ and $E_2$ have the same image $\Oo^p \cong \ZZ_{p^2}$.
Fix a uniformizer $\Pi \in \Delta$ satisfying $x\Pi = \Pi x^{\iota}$ for all $x \in \Oo^p \subset \Delta$. 
Since $\frakP \mid \ker(\Theta)$, there is an isomorphism of $\ZZ_p$-modules $L_p(\AAA_1, \AAA_2) \cong R$, where
$$
R = \left\{\begin{bmatrix} x & y\Pi \\ py\Pi & x \end{bmatrix} : x, y \in \Oo^p \right\},
$$
and the CM actions $\kappa_1$ and $\kappa_2$ are identified with a single action $\Oo^p \to R$ given by
$x \mapsto \text{diag}(x, x)$. Under the isomorphism $L_p(\AAA_1, \AAA_2)
\cong R$ the quadratic form $\deg^\ast$ on $L_p(\AAA_1, \AAA_2)$ is identified with the quadratic form $Q$
on $R$ defined in Proposition \ref{quadratic form}. 
There is a decomposition of left $\Oo^p$-modules $R = R_{+} \oplus R_{-}$, with
$R_{+} = \Oo^p$, embedded diagonally in $R$, and $R_{-} = \Oo^pP$, where
$$
P = \begin{bmatrix} 0 & \Pi \\ p\Pi & 0 \end{bmatrix},
$$
and this decomposition is orthogonal with respect to the quadratic form $\deg^\ast$. Define
$\varphi_{\pm} : \Oo_{K, p} \to \Oo^p \subset R$ by
\begin{align*}
&\varphi_{+}(x_1 \otimes x_2) = \kappa_2(x_2)\kappa_1(\overline{x}_1) \\
&\varphi_{-}(x_1 \otimes x_2) = \kappa_2(x_2)\kappa_1(x_1),
\end{align*}
and let $\Phi$ be the isomorphism
$\varphi_{+} \times \varphi_{-} : \Oo_{K, p} \to \Oo^p \times \Oo^p$.
Then the usual action of $\Oo_{K, p}$ on $R$ is given by
$$
x \bullet f = \varphi_{+}(x)f_{+} + \varphi_{-}(x)f_{-}
$$
for $f = f_{+} + f_{-} \in R$ since $P\kappa_1(\overline{x}_1) = \kappa_1(x_1)P$ by the choice of $\Pi$. 
As $p$ is split in $F$,
$$
\Tr_{F_p/\QQ_p}(\deg^\ast(f_{+}), \deg^\ast(f_{-})) = \deg^\ast(f_{+}) + \deg^\ast(f_{-}) = \deg^\ast(f_{+} + f_{-}) = 
\deg^\ast(f),
$$
the second equality coming from orthogonality, so
$\Phi(\deg_{\CM}(f)) = (\deg^\ast(f_{+}), \deg^\ast(f_{-}))$.

Let $\frakp_{-} = \frakp$ and $\frakp_{+} = \overline{\frakp}$ be the other prime of $F$ over $p$.
We saw in the proof of Proposition \ref{quadratic II} that
\begin{align*}
&x_1 \otimes x_2 \in \frakP\Oo_{K, p} \iff \kappa_2(x_2)\kappa_1(x_1) \in \frakM \cap \Oo^p = p\Oo^p \\
&x_1 \otimes x_2 \in \overline{\frakP}\Oo_{K, p} \iff \kappa_2(x_2)\kappa_1(\overline{x}_1) \in \frakM \cap \Oo^p = p\Oo^p
\end{align*}
where $\overline{\frakP}$ is the prime of $K$ over $\frakp_{+}$ and $\frakM$ defined in (\ref{proof}). This implies $\Phi(\frakP\Oo_{K, p}) = \Oo^p \times p\Oo^p$ and $\Phi(\overline{\frakP}\Oo_{K, p}) = p\Oo^p \times
\Oo^p$ and hence
\begin{align*}
&\ord_{\frakp_{+}}(\deg_{\CM}(f)) = \ord_p(\deg^\ast(f_{+})) \\
&\ord_{\frakp_{-}}(\deg_{\CM}(f)) = \ord_p(\deg^\ast(f_{-})).
\end{align*}
Since $\deg^\ast(P) = Q(P) = -p^2$, for any
integer $m \gqq 1$ and any $f \in R$ we have
\begin{align*}
f \in \Oo^p + p^{m-1}R &\iff f_{-} \in p^{m-1}\Oo^pP \\
&\iff \ord_p(\deg^\ast(f_{-})) \gqq 2m  \\
&\iff \tfrac{1}{2}\ord_{\frakp}(\deg_{\CM}(f)) \gqq m.
\end{align*}

The functor
$$
\Def(\AAA_1, \AAA_2) \cong \Def_{\Oo_B}(A_1[p^{\infty}], \Oo^p) \times \Def_{\Oo_B}(A_2[p^{\infty}], \Oo^p) 
$$
is represented by $\WWW \widehat\otimes_{\WWW} \WWW \cong \WWW$. Let $(\widetilde{\AAA}_1, \widetilde{\AAA}_2)$ 
be the universal deformation of $(\AAA_1, \AAA_2)$ to $\WWW = W$.
It follows from \cite[Proposition 2.9]{RZ} that the functor $\Def(\AAA_1, \AAA_2, f)$ is
represented by $W_m = W/(p^m)$, where $m$ is the largest integer such that
$f \in \Hom_{\Oo_B}(A_1[p^{\infty}], A_2[p^{\infty}]) \cong R$ lifts to an element of
$$
\Hom_{\Oo_B}(\widetilde{A}_1[p^{\infty}] \otimes_W W_m, \widetilde{A}_2[p^{\infty}] \otimes_W W_m).
$$
Since there are $\Oo_B \otimes_{\ZZ} \Oo^p$-linear isomorphisms $A_1[p^{\infty}] \cong A_2[p^{\infty}]$ (as $\frakP \mid \ker(\Theta)$)
and $\widetilde{A}_j \otimes_W \overline{\FF}_\frakP \cong A_j$, there is an $\Oo_B \otimes_{\ZZ} \Oo^p$-linear
isomorphism $\widetilde{A}_1[p^{\infty}] \cong \widetilde{A}_2[p^{\infty}]$ by the uniqueness of the universal
deformation. Hence
$$
\Hom_{\Oo_B}(\widetilde{A}_1[p^{\infty}] \otimes_W W_m, \widetilde{A}_2[p^{\infty}] \otimes_W W_m) 
\cong R_m \cong \Oo^p + p^{m-1}R
$$
in the notation of Lemma \ref{lift}. Then $m = \frac{1}{2}\ord_{\frakp}(\deg_{\CM}(f))$ by the above calculation, which proves the result since $\text{lg}(W/(p^m)) = \text{lg}_W(W/(p^m)) = m$.
\end{proof}

With $(\AAA_1, \AAA_2)$ as above, suppose $p \mid d_B$ and $\frakP$ does not divide $\ker(\Theta)$.
As usual $A_j \cong M_j \otimes_{\Oo_{K_j}} E_j$ for some supersingular $E_j$.
Choose isomorphisms $E_j[p^\infty] \cong \frakg$ so that the CM actions $g_1 : \Oo_{K_1, p} \to \Delta$
and $g_2 : \Oo_{K_2, p} \to \Delta$ on $E_1$ and $E_2$, where $\Delta = \End(\frakg)$, have the same image $\Oo^p \cong \ZZ_{p^2}$. 
Fix a uniformizer $\Pi \in \Delta$ satisfying $\Pi g_1(x) = g_1(\overline{x})\Pi$ for all $x \in \Oo_{K_1, p}$.
There is an isomorphism of $\ZZ_p$-modules $L_p(\AAA_1, \AAA_2) \cong R'$, where
$$
R' = \left\{\begin{bmatrix} px & y\Pi \\ y\Pi & x \end{bmatrix} : x, y \in \Oo^p\right\},
$$
and the quadratic form $\deg^\ast$ on $L_p(\AAA_1, \AAA_2)$ is identified with the quadratic form $uQ'$ on
$R'$ defined in Proposition \ref{mixed types}. 
There is a decomposition of left $\Oo^p$-modules $R' = R'_{+} \oplus R'_{-}$, where
$R'_{+} = \Oo^pP_1$ and $R'_{-} = \Oo^pP_2$, with
$$
P_1 = \begin{bmatrix} p & 0 \\ 0 & 1 \end{bmatrix}, \quad P_2 = \begin{bmatrix} 0 & \Pi \\ \Pi & 0 \end{bmatrix}.
$$

\begin{lemma}\label{lift II}
With notation as above, let $\Aa_j$ be the universal deformation of $A_j$ to $\WWW = W$, and for each integer
$m \gqq 1$ set
$$
R_m' = \Hom_{\Oo_B}(\Aa_1[p^{\infty}] \otimes_W W_m, \Aa_2[p^{\infty}] \otimes_W W_m).
$$
Then the reduction map $R_m' \hookrightarrow \Hom_{\Oo_B}(A_1[p^{\infty}], A_2[p^{\infty}]) \cong R'$ induces an isomorphism
$$
R_m' \cong \Oo^pP_1 + p^{m-1}\Oo^pP_2.
$$
\end{lemma}

\begin{proof}
The proof is very similar to that of Lemma \ref{lift}, using the following two facts. For each $j \in \{1, 2\}$ 
there is a unique $\Oo_B \otimes_{\ZZ} \Oo_{K_j}$-stable direct summand 
$M_j \subset \widetilde{H}_1^{\text{dR}}(A_j)$ whose image under the reduction map $\widetilde{H}_1^{\text{dR}}(A_j) \to
H_1^{\text{dR}}(A_j)$ is $\text{Fil}(A_j)$, and such that the diagram (\ref{Hodge}) commutes, corresponding to the unique deformation $\Aa_j \otimes_W W_m$ of $A_j$ to $W_m$. Any $f \in R'$ lifts to an element of $R_m'$ if and only if $\widetilde{f}(M_1) \subset M_2$, where $\widetilde{f} : \widetilde{H}_1^{\text{dR}}(A_1) \to  \widetilde{H}_1^{\text{dR}}(A_2)$ is the unique lift of $f : H_1^{\text{dR}}(A_1) \to H_1^{\text{dR}}(A_2)$.
\end{proof}

\begin{proposition} \label{representability III}
If $p \mid d_B$ and $\frakP$ does not divide $\ker(\Theta)$, then $\Def(\AAA_1, \AAA_2, f)$ is
represented by a local Artinian $\WWW$-algebra of length 
$$
\frac{\ord_\frakp(\deg_{\CM}(f)) + 1}{2}.
$$
\end{proposition}

\begin{proof}
The proof is the same as that of Proposition \ref{representability II}, using Lemma \ref{lift II}, the key difference being 
$\deg^\ast(P_2) = uQ'(P_2) = -up$. 
\end{proof}

\subsection{The \'etale local ring}
Let $\ZZZ$ be a stack over $\Spec(\Oo_K)$ and let $z \in \ZZZ(\overline{\FF}_\frakP)$ be a geometric point.
An \textit{\'etale neighborhood} of $z$ is a commutative diagram in the $2$-category of stacks over
$\Spec(\Oo_K)$
$$
\xymatrix{
& U \ar[d] \\
\Spec(\overline{\FF}_\frakP) \ar[ur]^{\widetilde{z}} \ar[r]_<<<<<z & \ZZZ }
$$
where $U$ is an $\Oo_K$-scheme and $U \to \ZZZ$ is an \'etale morphism. The \textit{strictly Henselian local ring}
of $\ZZZ$ at $z$ is the direct limit
$$
\OOO^{\text{sh}}_{\ZZZ, z} = \dlim_{(U, \widetilde{z})}\OOO_{U, \widetilde{z}}
$$
over all \'etale neighborhoods of $z$, where $\OOO_{U, \widetilde{z}}$ is the local ring of the scheme $U$ at the image
of $\widetilde{z}$. The ring $\OOO^{\text{sh}}_{\ZZZ, z}$ is a strictly Henselian local ring with residue field
$\overline{\FF}_\frakP$ and the completion
$\widehat{\OOO}^{\text{sh}}_{\ZZZ, z}$ is a $\WWW$-algebra.

\begin{theorem}\label{local ring}
Let $\alpha \in F^\times$, let $\Theta : \Oo_K \to \Oo_B/\frakm_B$ be a ring homomorphism, 
and suppose $\frakP \subset \Oo_K$ is a prime ideal lying over a prime $p$.
Set
$$
\nu_{\frakp}(\alpha) = \frac{1}{2}\ord_\frakp(\alpha\frakp\frakD), \quad  
\nu'_\frakp(\alpha) = \frac{1}{2}\ord_\frakp(\alpha),
$$
where $\frakp = \frakP \cap \Oo_F$.
For any $x = (\AAA_1, \AAA_2, f) \in \XXX_{\Theta, \alpha}^B(\overline{\FF}_\frakP)$,
the ring $\OOO^{\text{\upshape sh}}_{\XXX_{\Theta, \alpha}^B, x}$ is Artinian of length $\nu_\frakp(\alpha)$ if $p \nmid d_B$ or $p \mid d_B$ and $\frakP \nmid \ker(\Theta)$, and is Artinian of length $\nu'_\frakp(\alpha)$ if $p \mid d_B$ and $\frakP \mid \ker(\Theta)$.
\end{theorem}

\begin{proof}
Using Corollary \ref{decomp lift}, the same proof as in \cite[Proposition 2.25]{Howard} shows
the functor $\Def(\AAA_1, \AAA_2, f)$ is represented by the ring
$\widehat{\OOO}^{\text{sh}}_{\XXX_{\Theta, \alpha}^B, x}$. The result then follows from Propositions 
\ref{representability I}, \ref{representability II}, \ref{representability III}, and
the fact that $\text{lg}(\widehat{\OOO}^{\text{sh}}_{\XXX_{\Theta, \alpha}^B, x}) = \text{lg}(\OOO^{\text{sh}}_{\XXX_{\Theta, \alpha}^B, x})$.
\end{proof}

\section{Final formula}

As in the introduction, let $\chi$ be the quadratic Hecke character associated with the extension $K/F$, so $\chi = \prod_v\chi_v : I_F/F^\times \to \{\pm 1\}$, where the product is over all places of $F$, $I_F$ is the group of ideles of $F$, and
$$
\chi_v(x_v) = \left\{\begin{array}{ll}
1 & \text{if $x_v \in \NN_{K_v/F_v}(K_v^\times)$} \\
-1 & \text{if $x_v \notin \NN_{K_v/F_v}(K_v^\times)$}.
\end{array} \right.
$$
We may interpret $\chi$ as a character on ideals as follows. Since $K_v/F_v$ is unramified for any finite place $v$ of $F$,
the norm map $\NN_{K_v/F_v} : \Oo_{K_v}^\times \to \Oo_{F_v}^\times$ is surjective, so if $\fraka$ is a fractional
$\Oo_F$-ideal, then the definition $\chi_v(\fraka) = \chi_v(a_v)$ is independent of the choice
of $a_v \in F_v^\times$ satisfying $a_v\Oo_{F_v} = \fraka\Oo_{F_v}$.

For any $\alpha \in F^\times$ totally positive and any ring homomorphism $\Theta : \Oo_K \to \Oo_B/\frakm_B$, 
define a finite set of prime ideals
$$
\Diff_\Theta(\alpha) = \{\frakp \subset \Oo_F : \chi_\frakp(\alpha\fraka_\Theta\frakD) = -1\},
$$
where $\fraka_\Theta = \ker(\Theta) \cap \Oo_F$. Note that any prime in $\Diff_\Theta(\alpha)$
is inert in $K$ because if $\frakp$ is split in $K$ then $K_\frakp \cong F_\frakp \times F_\frakp$
and the norm map $K_\frakp \to F_\frakp$ is just multiplication, which is surjective.

\begin{lemma}
For any $\alpha \in F^\times$ totally positive, the set $\Diff_{\Theta}(\alpha)$ is nonempty.
\end{lemma}

\begin{proof}
Let $D$ be the discriminant of $F$, so $\frakD = \sqrt{D}\Oo_F$. If $v_1$ and $v_2$ are the two archimedean places of $F$ then $\chi_{v_1}(\alpha\sqrt{D})\chi_{v_2}(\alpha\sqrt{D}) = -1$.
Since  $\prod_v\chi_v(x) = 1$ for any $x \in F^\times$, it follows that the set 
$$
\Diff(\alpha) = \{\frakp \subset \Oo_F : \chi_\frakp(\alpha\frakD) = -1\}
$$
has odd cardinality, and in particular is nonempty. The set $P = \{\frakp \subset \Oo_F : \frakp \mid \fraka_\Theta\}$ has cardinality equal to the number of primes dividing $d_B$, which is even since $B$ is indefinite. Therefore, there exists a
prime $\frakp_0 \subset \Oo_F$ such that $\frakp_0 \in \Diff(\alpha)\sm P$ or $\frakp_0 \in P\sm\Diff(\alpha)$.
Note that if $\frakp \in P$ then $\fraka_\Theta\Oo_{F_\frakp} = \frakp\Oo_{F_\frakp} = \pi_\frakp\Oo_{F_\frakp}$, where $\pi_\frakp \in \Oo_{F_\frakp}$ is a uniformizer at $\frakp$. If $p$ is the prime below $\frakp$ then $p \mid d_B$ is inert in $K_1$ and $K_2$, so $\frakp$ is inert in $K$.
Hence $K_\frakp/F_\frakp$ is an unramified extension of local fields, which implies $\chi_\frakp(\fraka_\Theta) = \chi_\frakp(\pi_\frakp) = -1$. Also, $\chi_\frakp(\fraka_\Theta) = 1$ if $\frakp \notin P$, so in either case, $\frakp_0 \in \Diff_{\Theta}(\alpha)$.
\end{proof}

\begin{lemma}\label{reflex primes}
For any prime $\frakP \subset \Oo_K$ and any ring homomorphism $\Theta : \Oo_K \to \Oo_B/\frakm_B$, we have $\#[\XXX_\Theta^B(\overline{\FF}_\frakP)] = |\Gamma|$, where $\Gamma = \Cl(\Oo_{K_1}) \times \Cl(\Oo_{K_2})$.
\end{lemma}

\begin{proof}
Let $\theta_j = \Theta|_{\Oo_{K_j}}$. By definition, an object of $\XXX_\Theta^B(\overline{\FF}_\frakP)$ is a pair
$(\AAA_1, \AAA_2)$ with $\AAA_j$ an object of $\YYY_j^B(\theta_j)(\overline{\FF}_\frakP)$, so by what we proved in
Section 3.3,
\begin{equation*}
\#[\XXX_\Theta^B(\overline{\FF}_\frakP)] = \#[\YYY_1^B(\theta_1)(\overline{\FF}_\frakP)]\cdot\#[\YYY_2^B(\theta_2)(\overline{\FF}_\frakP)] = |\Cl(\Oo_{K_1})|\cdot|\Cl(\Oo_{K_2})| = |\Gamma|. \qedhere
\end{equation*}
\end{proof}

\begin{proposition}\label{local ring II} 
Suppose $\alpha \in F^\times$ and $\Theta : \Oo_K \to \Oo_B/\frakm_B$ is a ring homomorphism.
If $\#\Diff_\Theta(\alpha) > 1$ then $\XXX_{\Theta, \alpha}^B = \varnothing$. Suppose $\Diff_\Theta(\alpha) = \{\frakp\}$,
let $\frakP \subset \Oo_K$ be the prime over $\frakp$, and let $p\ZZ = \frakp \cap \ZZ$. 
Then the stack $\XXX_{\Theta, \alpha}^B$ is supported in characteristic $p$. More specifically, it only has
geometric points over the field $\overline{\FF}_\frakP$ {\upshape (}if it has any at all{\upshape )}. 
\end{proposition}

\begin{proof} 
By Proposition \ref{quadratic}, $\XXX_{\Theta, \alpha}^B$ has no geometric points in characteristic $0$.
Suppose $\XXX_{\Theta, \alpha}^B(\overline{\FF}_\frakP) \neq \varnothing$ for some prime ideal $\frakP \subset \Oo_K$.
Fix $(\AAA_1, \AAA_2, f) \in \XXX_{\Theta, \alpha}^B(\overline{\FF}_\frakP)$,
and let $\frakp = \frakP \cap \Oo_F$ and $p\ZZ = \frakp \cap \ZZ$. 
Any prime ideal $\frakq$ of $\Oo_F$ lying over $p$ is inert in $K$ (by Proposition \ref{quadratic}(d) and our assumption about the primes
dividing $d_B$), so for such a $\frakq$,
$$
\chi_\frakl(\frakq) = \left\{\begin{array}{ll}
-1 & \text{if $\frakl = \frakq$} \\
1 & \text{if $\frakl \neq \frakq$}
\end{array} \right.
$$
for any prime $\frakl \subset \Oo_F$. By Theorem \ref{quadratic combined}, the quadratic space
$(\widehat{K}, \beta\cdot\NN_{K/F})$ represents $\alpha$ for any $\beta \in \widehat{F}^\times$ satisfying
$\beta\widehat{\Oo}_F = \fraka_\Theta\frakp\frakD^{-1}\widehat{\Oo}_F$.
It follows that $\chi_\frakl(\alpha) = \chi_\frakl(\fraka_\Theta\frakp\frakD^{-1})$
for every prime $\frakl \subset \Oo_F$, so $\Diff_\Theta(\alpha) = \{\frakp\}$. This shows that if $\XXX_{\Theta, \alpha}^B(\overline{\FF}_\frakP) \neq \varnothing$ then $\Diff_\Theta(\alpha) = \{\frakp\}$, where $\frakp = \frakP \cap \Oo_F$. 
\end{proof}

Recall the definition of the arithmetic degree of $\XXX_{\Theta, \alpha}^B$ from the introduction:
$$
\deg(\XXX_{\Theta, \alpha}^B) = \sum_{\frakP \subset \Oo_K}\log(|\FF_\frakP|)\sum_{x \in [\XXX_{\Theta, \alpha}^B(\overline{\FF}_\frakP)]} \frac{\text{lg}(\OOO^{\text{sh}}_{\XXX_{\Theta, \alpha}^B, x})}{|\Aut(x)|}.
$$

\begin{theorem}\label{final formula}
Let $\alpha \in F^\times$ be totally positive and suppose $\alpha \in \frakD^{-1}$. Let $\Theta : \Oo_K \to \Oo_B/\frakm_B$ 
be a ring homomorphism with $\fraka_\Theta = \ker(\Theta) \cap \Oo_F$, suppose $\Diff_\Theta(\alpha) = \{\frakp\}$, and let $p\ZZ = \frakp \cap \ZZ$. \\
{\upshape (a)} If $p \nmid d_B$ then 
$$
\deg(\XXX_{\Theta, \alpha}^B) = \frac{1}{2}\log(p)\cdot\ord_\frakp(\alpha\frakp\frakD)\cdot\rho(\alpha\fraka_\Theta^{-1}
\frakp^{-1}\frakD).
$$
{\upshape (b)} Suppose $p \mid d_B$ and let $\frakP \subset \Oo_K$ be the prime over $\frakp$. If $\frakP$ 
divides $\ker(\Theta)$  then
$$
\deg(\XXX_{\Theta, \alpha}^B) = \frac{1}{2}\log(p)\cdot \ord_\frakp(\alpha)\cdot \rho(\alpha\fraka_\Theta^{-1}\frakp^{-1}
\frakD).
$$
If $\frakP$ does not divide $\ker(\Theta)$ then
$$
\deg(\XXX_{\Theta, \alpha}^B) = \frac{1}{2}\log(p)\cdot \ord_\frakp(\alpha\frakp)\cdot \rho(\alpha\fraka_\Theta^{-1}\frakp^{-1}\frakD).
$$
If $\alpha \notin \frakD^{-1}$ or if $\#\Diff_\Theta(\alpha) > 1$, then $\deg(\XXX_{\Theta, \alpha}^B) = 0$.
\end{theorem}

\begin{proof}
(a) The group $\Aut(\AAA_1, \AAA_2)$ acts on the set $L(\AAA_1, \AAA_2)$ by
$(g_1, g_2) \cdot f = g_2 \circ f \circ g_1^{-1}$
and the stabilizer of $f$ under this action is $\text{Stab}(f) = \Aut(\AAA_1, \AAA_2, f)$.
Using Theorem \ref{local ring}, Proposition \ref{local ring II}, Lemma \ref{automorphism group}, 
and $|\FF_\frakP| = \NN_{K/\QQ}(\frakP) = p^2$, 
\begin{align*}
\deg(\XXX_{\Theta, \alpha}^B) &= \log(|\FF_\frakP|)\sum_{x \in [\XXX_{\Theta, \alpha}^B(\overline{\FF}_\frakP)]}\frac{\text{lg}(
\OOO^{\text{sh}}_{\XXX_{\Theta, \alpha}^B, x})}{|\Aut(x)|} \\
&= 2\log(p)\nu_\frakp(\alpha)\sum_{(\AAA_1, \AAA_2, f) \in [\XXX_{\Theta, \alpha}^B(\overline{\FF}_\frakP)]}
\frac{1}{|\Aut(\AAA_1, \AAA_2, f)|} \\
&= 2\log(p)\nu_\frakp(\alpha)\sum_{(\AAA_1, \AAA_2) \in [\XXX^B_\Theta(\overline{\FF}_\frakP)]}
\sum_{\substack{f \in L(\AAA_1, \AAA_2) \\ \deg_{\CM}(f) = \alpha}}\frac{1}{|\Aut(\AAA_1, \AAA_2, f)|}\cdot
\frac{|\text{Stab}(f)|}{|\Aut(\AAA_1, \AAA_2)|} \\
&= 2\log(p)\nu_\frakp(\alpha)\sum_{(\AAA_1, \AAA_2) \in [\XXX_\Theta^B(\overline{\FF}_\frakP)]}
\sum_{\substack{f \in L(\AAA_1, \AAA_2) \\ \deg_{\CM}(f) = \alpha}}\frac{1}{w_1w_2}.
\end{align*}
Now using Theorem \ref{orbital}, Proposition \ref{orbital II}, and Lemma \ref{reflex primes}, we have 
\begin{align*}
\deg(\XXX_{\Theta, \alpha}^B) 
&= \frac{2\log(p)\nu_\frakp(\alpha)}{|\Gamma|}\sum_{(\AAA_1, \AAA_2) \in [\XXX^B_\Theta(\overline{\FF}_\frakP)]}
\sum_{(\fraka_1, \fraka_2) \in \Gamma}
\sum_{\substack{f \in L(\fraka_1 \otimes \AAA_1, \fraka_2 \otimes \AAA_2) \\ \deg_{\CM}(f) = \alpha}}\frac{1}{w_1w_2} \\
&= \log(p)\frac{\nu_\frakp(\alpha)}{|\Gamma|}\sum_{(\AAA_1, \AAA_2) \in [\XXX_\Theta^B(\overline{\FF}_\frakP)]}
\prod_{\ell}O_\ell(\alpha, \AAA_1, \AAA_2) \\
&= \log(p)\frac{\nu_\frakp(\alpha)}{|\Gamma|}\sum_{(\AAA_1, \AAA_2) \in [\XXX_\Theta^B(\overline{\FF}_\frakP)]}
\rho(\alpha\fraka_\Theta^{-1}\frakp^{-1}\frakD) \\
&= \frac{1}{2}\log(p)\cdot\ord_\frakp(\alpha\frakp\frakD)\cdot\rho(\alpha\fraka_\Theta^{-1}\frakp^{-1}\frakD).
\end{align*}

(b) Suppose $p \mid d_B$. If $\frakP$ divides $\ker(\Theta)$ then a similar calculation to
that in (a), replacing $\nu_\frakp(\alpha)$ with $\nu'_\frakp(\alpha)$, gives the desired result. If $\frakP$ does not
divide $\ker(\Theta)$ then the exact same calculation as in (a) gives the desired formula, noting that
$\nu_\frakp(\alpha) = \frac{1}{2}\ord_\frakp(\alpha\frakp)$ for $p \mid d_B$.
The final claim follows from Proposition \ref{local ring II} and the fact that $\deg_{\CM}$ takes values in $\frakD^{-1}$.
\end{proof}

\appendix
\section{Hecke correspondences}
In this section we will define the Hecke correspondences $T_m$ on $\MMM$ and $\MMM^B$, and prove
the equalities (\ref{Hecke 1}) and (\ref{Hecke 2}) in the introduction (we continue with the same notation as
in Sections 1.1 and 1.2). Suppose $\XXX$ is a Noetherian stack and $\ZZZ$ is
a closed substack of codimension $1$. Using an atlas on $\XXX$, the definition of a Weil divisor on a Noetherian scheme associated with a closed subscheme of codimension $1$ 
can be extended to stacks to give a divisor $[\ZZZ] \in \text{Div}(\XXX)$ (see \cite[Definition 3.5]{Vistoli}).
Suppose $h : \XXX \to \XXX'$ is a morphism of Noetherian stacks of the same dimension. 
In the case of $h$ finite and flat there is an induced group homomorphism 
$$
h^\ast : \text{Div}(\XXX') \to \text{Div}(\XXX)
$$
defined on prime divisors by $h^\ast\DDD = [\DDD \times_{\XXX'} \XXX]$ and extended linearly to all
of $\text{Div}(\XXX')$. If $h$ is proper and representable, there is a notion of the image of $h$, which is a 
closed substack of $\XXX'$, defined through an atlas and the scheme-theoretic image (see \cite[Definition 1.7]{Vistoli}). 
For $h$ finite, flat, and representable, this leads to a group homomorphism
$$
h_\ast : \text{Div}(\XXX) \to \text{Div}(\XXX')
$$
defined by sending a prime divisor $\DDD$ to $\deg(\DDD/\DDD')\cdot [\DDD']$,
where $\DDD'$ is the image of $\DDD$ under $h$ and $\deg(\DDD/\DDD')$ is the degree of the 
morphism $\DDD \to \DDD'$ (see \cite[Definition 3.6]{Vistoli}).

Fix a positive integer $m$. Let $\MMM(m)$ be the category fibered in groupoids over $\Spec(\Oo_K)$
with $\MMM(m)(S)$ the category of triples $(E_1, E_2, \varphi)$ with $E_i$ an object of $\MMM(S)$ and $\varphi \in \Hom_S(E_1, E_2)$ satisfying $\deg(\varphi) = m$ on every connected component of $S$. The category $\MMM(m)$ is a stack,
flat of relative dimension $1$ over $\Spec(\Oo_K)$, and there are two finite flat morphisms
$$
\xymatrix{
& \MMM(m) \hspace{.5mm} \ar@<-.7ex>[r]_<<<<<{\pi_2} \ar@<.7ex>[r]^<<<<<{\pi_1} & \hspace{1.5mm} \MMM }
$$
given by $\pi_i(E_1, E_2, \varphi) = E_i$. Define 
$T_m : \text{Div}(\MMM) \to \text{Div}(\MMM)$
by $T_m = (\pi_2)_\ast \circ (\pi_1)^\ast$. 

For $i \in \{1, 2\}$ let $f_i : \YYY_i \to \MMM$ be the finite morphism defined by forgetting the complex multiplication structure.
Consider $\DDD_1 = \YYY_1 \times_{f_1, \MMM, \pi_1} \MMM(m)$. 
Up to the obvious isomorphism of stacks, the objects of $\DDD_1$
can be described as triples $(E_1, E_2, \varphi)$ with $E_1 \in \YYY_1$, $E_2 \in \MMM$, and
$\varphi : E_1 \to E_2$ a degree $m$ isogeny. Now let $g$ be the composition $\DDD_1 \to \MMM(m) \map{\pi_2} \MMM$.
The fiber product $\DDD_1 \times_{g, \MMM, f_2} \YYY_2$ is easily seen to be isomorphic to $\TTT_m$.
Below is a diagram of these spaces and morphisms:
\begin{equation}\label{diagram}
\xymatrix{
&& \TTT_m \ar[dl] \ar[dr] & \\
& \DDD_1 \ar[dl] \ar[dr] \ar@{->}@/^2pc/[ddrr]^{g} & & \YYY_2 \ar[dd]^{f_2} \\
\YYY_1 \ar[dr]_{f_1} && \MMM(m) \ar[dl]^{\pi_1} \ar[dr]_{\pi_2} && \\
& \MMM & & \MMM. & }
\end{equation}

Viewing $\DDD_1$ as a closed substack of $\MMM(m)$ through the image of $\DDD_1 \to \MMM(m)$, the divisor $T_m\YYY_1$ on $\MMM$ is $(\pi_2)_\ast[\DDD_1]$, where $[\DDD_1]$ is the divisor associated with $\DDD_1$,
so to prove $\deg(\TTT_m) = I(T_m\YYY_1, \YYY_2)$, we need to show
\begin{equation}\label{Hecke 3}
\deg(\DDD_1 \times_{g, \MMM, f_2} \YYY_2) = I((\pi_2)_\ast[\DDD_1], [\YYY_2]),
\end{equation}
where we are writing $[\YYY_2]$ for the divisor on $\MMM$ determined by the image of $f_2$.

Let $k = \overline{\FF}_{\frakP}$ for $\frakP \subset \Oo_K$ a prime ideal and let $x \in \MMM(k)$ be a 
geometric point.
For any two prime divisors $\ZZZ$ and $\ZZZ'$ on $\MMM$ intersecting properly, define the \textit{Serre
intersection multiplicity} at $x$ by
$$
I_x^{\MMM}(\ZZZ, \ZZZ') = \sum_{d \gqq 0}(-1)^d\hspace{.5mm}\text{lg}_{\OOO^{\text{sh}}_{\MMM, x}}\big(\Tor_d^{\OOO^{\text{sh}}_{\MMM, x}}
(\OOO^{\text{sh}}_{\ZZZ, x}, \OOO^{\text{sh}}_{\ZZZ', x})\big)
$$
if $x \in (\ZZZ \cap \ZZZ')(k)$ and set $I_x^{\MMM}(\ZZZ, \ZZZ') = 0$ otherwise. Extend this definition bilinearly
to all divisors on $\MMM$. Again, if $\ZZZ$ and $\ZZZ'$ are prime divisors on $\MMM$ intersecting properly, there is a  
way of defining a $0$-dimensional cycle $\ZZZ\cdot\ZZZ'$ on $\MMM$ in such a way that 
$$
\text{Coef}_x(\ZZZ \cdot \ZZZ') = I_x^{\MMM}(\ZZZ, \ZZZ'),
$$
where $\text{Coef}_x(\ZZZ \cdot \ZZZ')$ is the coefficient in $\ZZZ\cdot\ZZZ'$ of the $0$-dimensional closed substack determined by the image of $x : \Spec(k) \to \MMM$ (see \cite[Chapter V]{Serre} and
\cite[Chapter I]{Soule}).

With notation as in (\ref{diagram}), let $\DDD_2 = \MMM(m) \times_{\pi_2, \MMM, f_2} \YYY_2$, which has  objects $(E_1, E_2, \varphi)$ with $E_1 \in \MMM$, $E_2 \in \YYY_2$, and
$\varphi : E_1 \to E_2$ a degree $m$ isogeny, so
$[\DDD_2] = (\pi_2)^\ast[\YYY_2]$. Also, let $x \in \MMM(m)(k)$ where $x = (E_1, E_2, \varphi)$
and $E_i \in \YYY_i(k)$ for $i \in \{1, 2\}$. We claim 
\begin{equation}\label{Tor}
\Tor_d^{\OOO^{\text{sh}}_{\MMM(m), x}}(\OOO^{\text{sh}}_{\DDD_1, x}, \OOO^{\text{sh}}_{\DDD_2, x}) = 0
\end{equation}
for all $d > 0$. To prove this, note that
$$
\OOO^{\text{sh}}_{\DDD_i, x} \cong
\OOO^{\text{sh}}_{\MMM(m), x} \otimes_{\OOO^{\text{sh}}_{\MMM, \pi_i(x)}} \OOO^{\text{sh}}_{\YYY_i, \pi_i(x)}
$$
for $i \in \{1, 2\}$, so from $\pi_1$ being flat,
$$
\Tor_d^{\OOO^{\text{sh}}_{\MMM(m), x}}(\OOO^{\text{sh}}_{\DDD_1, x}, \OOO^{\text{sh}}_{\DDD_2, x})
\cong
\Tor_d^{\OOO^{\text{sh}}_{\MMM, \pi_1(x)}}(\OOO^{\text{sh}}_{\YYY_1, \pi_1(x)}, \OOO^{\text{sh}}_{\DDD_2, x}). 
$$
As $\OOO^{\text{sh}}_{\MMM, \pi_1(x)}$ and $\OOO^{\text{sh}}_{\YYY_1, \pi_1(x)}$ are regular local rings of dimension
$2$ and $1$, respectively, and $\OOO^{\text{sh}}_{\DDD_2, x}$ is a Noetherian local ring of dimension $1$, the $\OOO^{\text{sh}}_{\MMM, \pi_1(x)}$-modules $\OOO^{\text{sh}}_{\YYY_1, \pi_1(x)}$ and
$\OOO^{\text{sh}}_{\DDD_2, x}$ are Cohen-Macaulay, so
(\ref{Tor}) holds for all $d > 0$ by \cite[Corollary on p. \hspace{-.5mm}111]{Serre}. 

There is a projection formula 
$$
((\pi_2)_\ast[\DDD_1]) \cdot [\YYY_2] = (\pi_2)_\ast([\DDD_1]\cdot ((\pi_2)^\ast[\YYY_2])).
$$
This is a special case of a more general formula, but it takes this form in our case since (\ref{Tor}) holds 
(see \cite[p. \hspace{-.5mm}118, formulas (10), (11)]{Serre}). 
It follows that for any $y \in \MMM(k)$,
\begin{align*}
I_y^{\MMM}((\pi_2)_\ast[\DDD_1], [\YYY_2]) &= \text{Coef}_y\big(((\pi_2)_\ast[\DDD_1])\cdot[\YYY_2]\big) \\
&= \sum_{x \in \pi_2^{-1}(y)}[\kappa(x) : \kappa(y)]\text{Coef}_x\big([\DDD_1]\cdot ((\pi_2)^\ast[\YYY_2])\big) \\
&= \sum_{x \in \pi_2^{-1}(y)}I_x^{\MMM(m)}([\DDD_1], [\DDD_2]),
\end{align*}
where we are using that $\kappa(x) = \kappa(y) = k$.
Letting $h_i : \DDD_i \to \MMM(m)$ be the natural projection, there is an isomorphism of stacks
$$
\DDD_1 \times_{h_1, \MMM(m), h_2} \DDD_2 \cong \DDD_1 \times_{g, \MMM, f_2} \YYY_2,
$$
and this stack has objects $(E_1, E_2, \varphi)$ where $E_i \in \YYY_i$ and $\varphi : E_1 \to E_2$ is an isogeny of degree $m$.
Also, by (\ref{Tor}) we have
$$
I^{\MMM(m)}_x([\DDD_1], [\DDD_2]) 
= \text{lg}(\OOO^{\text{sh}}_{\DDD_1, x}\otimes_{\OOO^{\text{sh}}_{\MMM(m), x}} \OOO^{\text{sh}}_{\DDD_2, x}).
$$
Therefore, for any $y \in \MMM(k)$,
$$
 I_y^{\MMM}((\pi_2)_\ast[\DDD_1], [\YYY_2]) 
=  \sum_{x \in \pi_2^{-1}(y)}\text{lg}(\OOO^{\text{sh}}_{\DDD_1 \times_{\MMM(m)}\DDD_2, x}) 
= \sum_{x \in \pi_2^{-1}(y)}\text{lg}(\OOO^{\text{sh}}_{\DDD_1 \times_{\MMM}\YYY_2, x}). 
$$
Since $\YYY_2$ is regular and the local ring at $y$ of any prime divisor appearing in $(\pi_2)_\ast[\DDD_1]$ is
a $1$-dimensional domain, hence Cohen-Macaulay, the $\Tor_d$ terms appearing in the sum 
$I_y^{\MMM}((\pi_2)_\ast[\DDD_1], [\YYY_2])$ are zero for all $d > 0$. Multiplying both sides of the above equality by 
$\log(|\FF_{\frakP}|)/|\Aut(y)|$ and summing over all $y \in [\YYY_2(\overline{\FF}_\frakP)]$ and over all primes $\frakP \subset \Oo_K$ 
gives
\begin{align*}
I((\pi_2)_\ast[\DDD_1], [\YYY_2]) &= \sum_{\frakP}\sum_{y \in [\YYY_2(\overline{\FF}_\frakP)]}\frac{\log(|\FF_{\frakP}|)}{|\Aut(y)|}\sum_{x \in \pi_2^{-1}(y)}\text{lg}(\OOO^{\text{sh}}_{\DDD_1 \times_{\MMM}\YYY_2, x}) \\
&= \sum_{\frakP}\sum_{y \in [\YYY_2(\overline{\FF}_\frakP)]}\sum_{\substack{(\mathbf{E}_1, \mathbf{E}_2, \varphi) \in [(\DDD_1 \times_{\MMM} \YYY_2)(\overline{\FF}_\frakP)] \\ [\mathbf{E}_2] = y}}\frac{\log(|\FF_{\frakP}|)}{|\Aut(y)|}\text{lg}(\OOO^{\text{sh}}_{\DDD_1 \times_{\MMM}\YYY_2, (\mathbf{E}_1, \mathbf{E}_2, \varphi)}),
\end{align*}
where $[\mathbf{E}_i]$ is the isomorphism class of the object $\mathbf{E}_i$ of $\YYY_i(\overline{\FF}_\frakP)$
which has underlying elliptic curve $E_i$.
For a fixed $(\mathbf{E}_1, \mathbf{E}_2) \in (\YYY_1 \times \YYY_2)(\overline{\FF}_\frakP)$, the group $\Aut(\mathbf{E}_1)$ acts on the set of all degree $m$ isogenies $\varphi : E_1 \to E_2$ by $g\cdot \varphi = \varphi \circ g^{-1}$ and this action is free. Hence
$$
I((\pi_2)_\ast[\DDD_1], [\YYY_2]) = \sum_{\frakP}\sum_{(\mathbf{E}_1, \mathbf{E}_2) \in [(\YYY_1 \times \YYY_2)(\overline{\FF}_\frakP)]}
\sum_{\substack{\varphi : E_1 \to E_2 \\ \deg(\varphi) = m}}\frac{\log(|\FF_{\frakP}|)}{|\Aut(\mathbf{E}_1)||\Aut(\mathbf{E}_2)|}\text{lg}(\OOO^{\text{sh}}_{\DDD_1 \times_{\MMM}\YYY_2, (\mathbf{E}_1, \mathbf{E}_2, \varphi)}).
$$
Similarly, for a fixed $(\mathbf{E}_1, \mathbf{E}_2) \in (\YYY_1 \times \YYY_2)(\overline{\FF}_\frakP)$, the group
$\Aut(\mathbf{E}_1, \mathbf{E}_2)$ acts on the set of all degree $m$ isogenies $\varphi : E_1 \to E_2$ by $(g_1, g_2) \cdot \varphi = g_2 \circ \varphi \circ g_1^{-1}$, and the stabilizer of $\varphi$ is
$\text{Stab}(\varphi) = \Aut(\mathbf{E}_1, \mathbf{E}_2, \varphi)$. Therefore
\begin{align*}
I((\pi_2)_\ast[\DDD_1], [\YYY_2]) &= \sum_{\frakP}\sum_{(\mathbf{E}_1, \mathbf{E}_2) \in [(\YYY_1 \times \YYY_2)(\overline{\FF}_\frakP)]}
\sum_{\substack{\varphi : E_1 \to E_2 \\ \deg(\varphi) = m}}\frac{\log(|\FF_{\frakP}|)\text{lg}(\OOO^{\text{sh}}_{\DDD_1 \times_{\MMM}\YYY_2, (\mathbf{E}_1, \mathbf{E}_2, \varphi)})}{|\Aut(\mathbf{E}_1, \mathbf{E}_2, \varphi)|}\cdot\frac{|\text{Stab}(\varphi)|}{|\Aut(\mathbf{E}_1, \mathbf{E}_2)|} \\
&=  \sum_{\frakP}\sum_{x \in [(\DDD_1\times_{\MMM}\YYY_2)(\overline{\FF}_\frakP)]}\frac{\log(|\FF_\frakP|)}{|\Aut(x)|}\text{lg}(\OOO^{\text{sh}}_{\DDD_1 \times_{\MMM}\YYY_2, x}),
\end{align*}
which proves (\ref{Hecke 3}).

The definition of $T_m : \text{Div}(\MMM^B) \to \text{Div}(\MMM^B)$ and the proof of
the equality $\deg(\TTT_m^B) = I(T_m\YYY_1^B, \YYY_2^B)$ are exactly the same as the elliptic curve case.
The equality (\ref{Hecke 2}) then follows from (\ref{TM decomp}).

\end{document}